\documentclass[11pt,a4paper]{amsart}
\usepackage[centering, margin=1in]{geometry}
\usepackage[latin1]{inputenc}
\usepackage[english]{babel}
\usepackage[T1]{fontenc}
\usepackage{amsmath}
\usepackage{amssymb}
\usepackage{amsthm}
\usepackage{array,booktabs}
\usepackage{graphicx}
\usepackage{pagecolor}
\usepackage{todonotes}
\usepackage{hyperref}
\usepackage{lmodern}
\usepackage{tcolorbox}
\usepackage{quoting}
\usepackage{enumerate}
\usepackage{mathtools}
\usepackage{csquotes}
\usepackage{stmaryrd}
\usepackage{cite}
\usepackage{tikz}
\usepackage{tikz-cd}
\usetikzlibrary{matrix,arrows,decorations.pathmorphing}
\usepackage{wrapfig}
\renewcommand{\Re}{\operatorname{Re}}
\renewcommand{\Im}{\operatorname{Im}}
\def\R{\ensuremath\mathbb{R}}
\def\C{\ensuremath\mathbb{C}}

\def\AA{\ensuremath\mathcal{A}}
\def\Z{\ensuremath\mathbb{Z}}
\def\Q{\ensuremath\mathbb{Q}}
\def\N{\ensuremath\mathbb{N}}
\def\Hb{\ensuremath\mathbb{H}}
\def\F{\ensuremath\mathbb{F}}
\def\tr{\ensuremath\mathrm{tr}}

\def\min{\ensuremath\mathrm{min}}

\def\Conj{\ensuremath\mathrm{Conj}}

\usepackage{hyperref}
\usepackage[final]{pdfpages}
\usepackage{verbatim}
\newtheorem{thm}{Theorem}[section]
\newtheorem{defi}[thm]{Definition}
\newtheorem{cor}[thm]{Corollary}
\newtheorem{lemma}[thm]{Lemma}
\newtheorem{conj}[thm]{Conjecture}
\newtheorem{prop}[thm]{Proposition}
\theoremstyle{remark}
\newtheorem{remark}{Remark}[section]

\def\eps{\ensuremath\varepsilon}
\def\0{\emptyset}

\def\Cl{\text{\rm Cl}}

\def\sgn{\mathrm{sgn}}

\def\sym{\mathrm{sym}}
\def\SO{\hbox{\rm SO}}

\def\SL{\hbox{\rm SL}}
\def\PSL{\mathrm{PSL}}

\def\CC{\mathcal{C}}

\def\BB{\mathcal{B}}

\def\LL{\mathcal{L}}

\def\cc{\mathbf{c}}
\def\modulo{\text{ \rm mod }}

\def\pmod{\text{ \rm mod }}
\numberwithin{equation}{section}

\numberwithin{equation}{section}
\begin{document}
\title{Equidistribution of $q$-orbits of closed geodesics}
\author{Asbj\o rn Christian Nordentoft}

\address{University of Copenhagen, Department of Mathematical Sciences, Universitetsparken 5, 2100 Copenhagen \O, Denmark}

\email{\href{mailto:nordentoft@math.ku.dk}{nordentoft@math.ku.dk}}

\date{\today}
\subjclass[2010]{11F67(primary)}
\maketitle
\begin{abstract}
Given an element of $(\Z/q\Z)^\times$ we associate a closed geodesic on the modular surface. We prove that the closed geodesics associated to cosets of sufficiently large subgroups equidistribute in the unit tangent bundle as $q$ tends to infinity. This is a $q$-orbit analogue of Duke's theorem for real quadratic fields as extended to subgroups by Popa. Finally, we show that the homology classes of the $q$-orbits of oriented closed geodesics concentrate around the Eisenstein line and present group theoretic applications.
\end{abstract}
\section{Introduction}
The modular surface $Y_0(1):=\PSL_2(\Z)\backslash \Hb$ is naturally endowed with the structure of a non-compact Riemannian orbifold with line and volume elements given by $ds=\tfrac{|dz|}{\Im z}$ and $d\mu=\tfrac{-dzd\overline{z}}{2i(\Im z)^2}$ in terms of the standard coordinate $z\in \Hb$. The unit tangent bundle admits a natural description as a homogeneous space
$$\mathbf{T}^1(Y_0(1))=\PSL_2(\Z)\backslash \PSL_2(\R),$$ 
which we equip with the Haar probability measure $dg$. The distribution of closed geodesics on the modular surface is a classical and rich topic starting with work of Linnik \cite{Li68} and Skubenko \cite{Skubenko62} who proved equidistribution in the unit tangent bundle for individual length packets of closed geodesics under certain congruence conditions. These congruence conditions were later removed in the celebrated work of Duke \cite{Du88}. More precisely, let $D>0$ be a positive discriminant with narrow class group $\Cl_D^+$. Given a narrow ideal class $B\in \Cl_D^+$ there is an associated oriented closed geodesic $\CC_B\subset \mathbf{T}^1(Y_0(1))$ in the unit tangent bundle of the modular surface with geodesic length $|\CC_B|=2\log \epsilon_D$ where $\epsilon_D$ denotes the fundamental positive unit of discriminant $D$, see \cite{Sarnak82}. The closed geodesic $\CC_B$ is a closed $A$-orbit where $A$ denotes the diagonal subgroup of $\PSL_2(\R)$. We equip $\CC_B$ with the unique $A$-invariant measure $\mu_B$ which descends to $ds$ on $Y_0(1)$. Duke's theorem states that for any smooth and compactly supported $\phi: \mathbf{T}^1(Y_0(1))\rightarrow \C$,
$$ \frac{\sum_{B\in \Cl_D^+}\int_{\CC_B} \phi (t) d\mu_B(t)  }{ 2|\Cl_D^+|\log \epsilon_D}\rightarrow \int_{\mathbf{T}^1(Y_0(1))} \phi (g) dg,\quad \text{as }D\rightarrow \infty.$$
It is natural to ask to what extent other collections of closed geodesics equidistribute. This is a special case of the \emph{Basic Question} posed in \cite[Sec. 1.1]{ELMV09}. In this context, Michel and Venkatesh in their ICM address \cite{MichelVenk06} put forth a number of ``sparse'' equidistribution conjectures refining Duke's theorem. Progress towards these conjectures include the works of Harcos--Michel \cite{HarcosMichel06} and Popa \cite[Theorem 6.5.1]{Popa06} obtaining equidistribution for closed geodesics associated to cosets of sufficiently large subgroups of class groups $\Cl_D^+$ (of index $\leq D^\delta$ for some $\delta>0$) and the work of Einsiedler--Lindenstrauss--Michel--Venkatesh \cite{EinLindMichVenk12} obtaining equidistribution of sufficiently large subcollections (of index $\leq D^{o(1)}$) using ergodic techniques, see also \cite{AkaEinsiedler16}. The Mixing Conjecture \cite[Conjecture 2]{MichelVenk06} has also seen recent progress, see \cite{Khayutin19}, \cite{BlomerBrumleyKhayutin22}, \cite{BlomerMichelUnipotentMixing}. In a different direction, Erlandsson--Souto \cite{ErlandssonSouto22} proved equidistribution for the special set of reciprocal geodesics answering a question of Sarnak \cite{Sarnak07}. In the other extreme, Bourgain--Kontorovich \cite{BourgainKontorovich17} showed the existence of many ($\gg_\eps D^{1/2-\eps}$ for any $\eps>0$) low-lying closed fundamental geodesics which in particular do not equidistribute. 

In this paper, we introduce {\lq\lq}$q$-orbits{\rq\rq} of closed geodesics in analogy with the imaginary quadratic case of Heegner points (see Remark \ref{rem:qorbit} below). We obtain an analogue of Duke's theorem (and it's extension to subgroups due to Harcos--Michel \cite{HarcosMichel06} and Popa \cite{Popa06}), as well as concentration results for the homology classes of $q$-orbits of closed geodesics as studied in the setting of Duke's theorem by the author \cite{Nor22}. In view of the \emph{Basic Question} \cite[Sec.\ 1.1]{ELMV09}, our results can be interpreted as saying that the $q$-orbit structure on the closed geodesics does not conspire with the low-lying geodesics. The key technical result proved in this paper, Proposition \ref{prop:main}, is a quantitative relation between geodesic periods and vertical periods of automorphic forms, which might be of independent interest. 

\subsection{Equidistribution of $q$-orbits of closed geodesics}
To state our results, denote by $\Gamma_0(N)\subset \PSL_2(\Z)$ the Hecke congruence subgroup of level $N\geq 1$ and by $Y_0(N):=\Gamma_0(N)\backslash \Hb$ the modular surface of level $N$ equipped with the line element $ds$ and volume element $d\mu$ as above. Let $q\geq 2$ be an integer and consider an embedding (of sets) 
$$\psi: (\Z/q\Z)^\times \hookrightarrow \Gamma_0(q)\subset \PSL_2(\Z),$$
satisfying 
\begin{enumerate}
\item \label{item:1} $\psi(a)\infty\equiv \tfrac{a}{q}\modulo 1$ (i.e. $\psi(a)$ is of the shape $\begin{psmallmatrix} a & \ast \\ q & \ast \end{psmallmatrix}$),
\item \label{item:2}$|\tr(\psi(a))|>2$ (i.e. $\psi(a)$ is a hyperbolic matrix).
\end{enumerate}
Notice that for $N|q$ we have $\psi(a)\in \Gamma_0(N)$ and we denote by 
$$\CC^\psi_a(N)\subset \mathbf{T}^1(Y_0(N))=\Gamma_0(N)\backslash \PSL_2(\R),$$ 
the oriented closed geodesic associated to the conjugacy class of $\psi(a)$ inside $\Gamma_0(N)$ (see Section \ref{sec:closedgeo} for details on this correspondence). This way we obtain the following \emph{$q$-orbit of oriented closed geodesics};
\begin{equation}\label{eq:qorbit}\{\CC^\psi_a(N)\subset  \mathbf{T}^1(Y_0(N)): a\in (\Z/q\Z)^\times\},\end{equation}
by which we mean a {\lq\lq}natural family{\rq\rq} of closed geodesics parametrized by the multiplicative group $(\Z/q\Z)^\times$ (in our case the parametrization stems from the double coset decomposition, see Section \ref{sec:DCE}). We equip $\CC^\psi_a(N)$ with the unique $A$-invariant measure $\mu^\psi_a$ which descends to the line element $ds$ when $\CC^\psi_a(N)$ is projected to $Y_0(N)$, see the introduction of \cite{Zelditch92} for details. Notice that in any $q$-orbit each geodesic length is attained at most twice. 

Our first result is the following equidistribution result in the unit tangent bundle for closed geodesics associated to cosets of sufficiently large subgroups using the parametrization of the $q$-orbits in terms of $(\Z/q\Z)^\times$.
\begin{thm}\label{thm:first} Fix $N\geq 1$, $\delta>7/8$ and $M\geq 1$. For every $q\geq 2$ such that $N|q$ let
\[\psi:(\Z/q\Z)^\times \hookrightarrow \Gamma_0(N),\] 
be an embedding satisfying the conditions {\rm (\ref{item:1}), (\ref{item:2})} as well as $|\tr(\psi(a))|\leq M q$ for all $a\in (\Z/q\Z)^\times$, and let $cH\subset (\Z/q\Z)^\times$ be a coset of a subgroup $H\leq (\Z/q\Z)^\times$ of size $\geq q^{\delta}$. Then for any smooth and compactly supported function $\phi : \mathbf{T}^1(Y_0(N))\rightarrow \C$ and $\eps>0$ we have that 
\begin{equation}\label{eq:first} \frac{\sum_{a\in cH}\int_{\CC^\psi_a(N)} \phi (t) d\mu^\psi_a(t)  }{\sum_{a\in cH}|\CC^\psi_a(N)|}= \int_{\mathbf{T}^1(Y_0(N))} \phi (g) dg+O_{\phi,\eps}((\log q)^{-1+\eps}),\quad \text{as }q\rightarrow \infty.\end{equation}
Here $|\CC|$, for a curve $\CC\subset Y_0(N)$, denotes the arc length with respect to the hyperbolic line element $ds$x and $dg$ denotes the Haar probability measure on $\Gamma_0(N)\backslash \PSL_2(\R)$.
\end{thm} 
\begin{remark}\label{rem:qorbit}The study of $q$-orbits (or $p$-orbits if $q=p$ is prime) is a very natural one with the most elementary example being the equidistribution of residues and their inverses on the torus;
\begin{equation}\{(\tfrac{a}{q},\tfrac{\overline{a}}{q}):a\in (\Z/q\Z)^\times\}\subset (\R/\Z)^2.\end{equation}
We refer to \cite{Humphries22} for sparse equidistribution results in this setting. Another classical example is the equidistribution of Hecke orbits (see e.g. \cite[Section 2.3]{FouvryKowMich15}). To explain this we observe that for a prime $p$ we have a natural parametrization of the Hecke operator $T_p$ in terms of $ \mathbf{P}^1(\F_p)$ given by
$$(T_pf)(g)= \frac{1}{p^{1/2}} \sum_{t\in  \mathbf{P}^1(\F_p)} f(\gamma_t g), \quad f:\PSL_2(\Z)\backslash \PSL_2(\R)\rightarrow \C,$$
where 
$$ \gamma_\infty=\begin{psmallmatrix}p & 0\\ 0 & 1 \end{psmallmatrix},\quad \gamma_t= \begin{psmallmatrix}1 &t\\ 0 & p \end{psmallmatrix},\,\, t\in \F_p. $$
The collection of points 
\begin{equation}\label{eq:Heckeorbit} \{\gamma_t i: t\in\mathbf{P}^1(\F_p)\}\subset \PSL_2(\Z)\backslash \Hb, \end{equation}
we define as the \emph{$p$-orbit} of the point $i\in \Hb$. Note that all but one of the points (\ref{eq:Heckeorbit}) lie on the horocycle $\{x+\tfrac{i}{p}: 0\leq x\leq 1\}$. Any non-trivial bound for the Hecke eigenvalues implies the well-known equidistribution of $p$-orbits; the probability measures with point masses at the points (\ref{eq:Heckeorbit}) converge weakly to the normalized hyperbolic measure $\tfrac{3}{\pi}\tfrac{dxdy}{y^2}$ as $p\rightarrow \infty$. This is in fact closely related to the situation studied in the present paper; Theorem \ref{thm:first} in the case where $H=(\Z/q\Z)^\times$ is the full multiplicative group reduces (morally at least) to a non-trivial bound of Hecke eigevalues. More precisely, the Weyl sums in our setting are (essentially) obtained by integrating the Maa{\ss} forms along  $\{\gamma_t i\R: t\in\mathbf{P}^1(\F_p)\}$ (see Remark \ref{rem:reduce}).
   
Recall \cite[Section 4.2]{Sh94} that the modular surface $\PSL_2(\Z)\backslash \Hb$ can be viewed as the moduli space for elliptic curves over $\C$ and that elliptic curves with complex multiplication (CM) correspond exactly to the \emph{Heegner points} as defined in e.g.\ \cite[Section 2.2]{Darmon06}. From this perspective, we note as in \cite[Remark 1.5]{BlomerMichelUnipotentMixing} that the collection of points (\ref{eq:Heckeorbit}) coincides, after removing $\gamma_t i$ for $t^2\equiv -1\modulo p$, with the set of Heegner points associated to the non-fundamental discriminant $-4p^2$. One can now study all sorts of sparse equidistribution questions using the group structure on $\F_p^\times \subset \mathbf{P}^1(\F_p)$ (i.e.\ using the $p$-orbit structure). Notice however that this group structure does not seem to admit a natural description in terms of moduli of CM elliptic curves since $\gamma_t i$ and $\gamma_{-t^{-1}} i$ parametrize the same elliptic curve but in general $\gamma_{at} i$ and $\gamma_{-at^{-1}} i$ with $a\in \F_p^\times$ do not.  Fouvry--Kowalski--Michel \cite[Section 2.3]{FouvryKowMich15} studied the equidistribution problem of $p$-orbits with certain algebraic weights (trace functions) associated to each of the Hecke points using the parametrization in terms of $\mathbf{P}^1(\F_p)$. An example being the indicator function of some coset of a subgroup of $\F_p^\times$ of sufficiently small index which is a direct analogue of Theorem \ref{thm:first}. Recently Blomer and Michel \cite{BlomerMichelUnipotentMixing} resolved the analogue of the Mixing Conjecture of Michel--Venkatesh \cite{MichelVenk06} for $p$-orbits (unipotent mixing). It would be interesting to see if their methods extends to the setting of this paper. \end{remark} 
\subsection{The homology classes associated to closed geodesics}    
As investigated in \cite{Nor22} there is another natural distribution problem that one can study associated to oriented closed geodesics, namely the distribution of their homology classes. To set this up, let 
$$V_N:=H_1(Y_0(N),\R),$$ 
denote the singular homology group with $\R$-coefficients of the (open) modular surface of level $N$. Given an oriented closed curve $\CC\subset Y_0(N)$ we denote by 
$$[\CC]\in H_1(Y_0(N),\Z)\subset V_N,$$ 
the associated homology class. Given $q\geq 2$ such that $N|q$ and an embedding 
$$\psi:(\Z/q\Z)^\times \hookrightarrow \Gamma_0(q),$$ 
satisfying the conditions (\ref{item:1}), (\ref{item:2}), we get a map
$$(\Z/q\Z)^\times\rightarrow V_N,\quad  a\mapsto [\CC_a^\psi(N)].$$
This map will depend very subtly on the sign of the traces of the image of $\psi$ and thus we will furthermore assume that the absolute trace is \emph{minimal} (see equation (\ref{eq:ta}) and the discussion after Corollary \ref{cor:mainhomo} for further details);
\begin{enumerate}
 \item[(3)]\label{item:3}  the absolute trace $|\tr\, \psi(a)|$ is minimal\footnote{For $a\in (\Z/q\Z)^\times$ such that $a+a^{-1}\not\in \{-2,-1,0,1,2\}\modulo q$ this means exactly that $|\tr\, \psi(a)|\leq \frac{q}{2}$.} subject to the conditions (\ref{item:1}) and (\ref{item:2}).
 \end{enumerate} 
In order to understand the distribution inside the vector space $V_N$ we will have to normalize our classes. A convenient way to do this is to use the natural projection map  
$$V_N-\{0\}\twoheadrightarrow  \mathbf{S}(V_N):=(V_N-\{0\})/\R_{>0},\quad v\mapsto \overline{v}.$$ 
We equip $\mathbf{S}(V_N)$ with the quotient topology which makes it a topological $2g$-dimensional sphere. Finally, denote by 
$$v_{E}(N):=[(x+i: 0\leq x \leq 1 )]\in H_1(Y_0(N),\Z),$$ 
the class of a simple loop going around the cusp at $\infty$. It is a general fact that we have an isomorphism 
\begin{equation}\label{eq:homabelian}H_1(Y_0(N),\Z)\xrightarrow{\sim} \Gamma_0(N)^\mathrm{ab}/(\Gamma_0(N)^\mathrm{ab})_\mathrm{tor},\end{equation} 
which identitfies $v_{E}(N)$ with the image of the matrix $\begin{psmallmatrix} 1& 1\\ 0 & 1\end{psmallmatrix}$ in the abelinization. 

Our second result is the following concentration result for homology classes of closed geodesics associated to subgroups of $(\Z/q\Z)^\times$.
\begin{thm}\label{thm:second} Fix $N\geq 3$ prime. For each $q\geq 2$ such that $N|q$, choose an embedding
$$\psi:(\Z/q\Z)^\times \hookrightarrow \Gamma_0(q),$$ 
satisfying {\rm (\ref{item:1}), (\ref{item:2}), (\hyperref[item:3]{3})} and a
 subgroup $H\leq (\Z/q\Z)^\times$ such that $-1\notin H$ and 
 $$\{a\in (\Z/q\Z)^\times: a\equiv 1\modulo q/N\}\not\subset H.$$ 
 Then we have
$$\overline{\sum_{a\in H} [\mathcal{C}^\psi_{a}(N)]}\rightarrow \overline{v_E(N)},\quad \text{as }q\rightarrow \infty, $$
in the quotient topology of $\mathbf{S}(V_N)$.
\end{thm}
Furthermore, we obtain a version of Theorem \ref{thm:second} which is uniform in the level $N$ (see Theorem \ref{thm:mainhomo}),  which is a $q$-orbit analogue of the result of Liu--Masri--Young \cite[Theorem 1.1]{LMY15} in the context of supersingular reduction of CM elliptic curves. 

Finally we will present the following group theoretic interpretation and application of Theorem \ref{thm:second}. We assume for simplicity that $N$ is prime and $N\equiv -1\modulo 12$ so that $\Gamma_0(N)$ is a free group on $2g+1$ generators with $g=\tfrac{N}{12}+O(1)$ (this follows from the discussion in \cite[Section 2.3]{Iw}). 
Then it has recently been shown by Doan--Kim--Lang--Tan \cite{DKLPT22} that one can find a set of free generators contained in 
\begin{equation}\label{eq:S2}\left\{\begin{psmallmatrix} 1 & 1 \\0 & 1 \end{psmallmatrix}\right\}\cup \left\{\begin{psmallmatrix} a & -(aa^{\ast} +1)/N \\N & -a^\ast \end{psmallmatrix}: 0< a<N \right\},\end{equation}
where $0< a^\ast <N$ is such that $aa^{\ast} \equiv -1\modulo N$ (here we consider the matrices (\ref{eq:S2}) as elements of $\PSL_2(\R)$ in the obvious way). It is natural to try to understand the representation of matrices in terms of such generating sets. If we look at a random word in the generators then we expect that there is about the same appearance of each generator. Theorem \ref{thm:second} implies that the $q$-orbits of closed geodesics are \emph{not} random in this sense since one coordinate dominates (see also Theorem \ref{thm:mainhomo} below). Geometrically this means that the $q$-orbits of closed geodesics ``wind around the cusp a lot''. In particular, we obtain the following elementary group theoretic non-containment result, which seems hard to prove by elementary means (see \cite{GolubevKamber23} on Sarnak's \emph{Optimal Lifting Property} for another application of automorphic forms to group theory).
\begin{cor}\label{cor:groupintro}
Let $N\equiv -1\modulo 12$ be prime and let 
$$\left\{\begin{psmallmatrix} 1& 1\\0&1 \end{psmallmatrix}, \sigma_1,\ldots, \sigma_{2g}\right\},$$ 
be a subset of (\ref{eq:S2}) which define a set of free generators of $\Gamma_0(N)$. Then for each $\delta>18$ there exists a constant $c(\delta)>0$ such that the following holds: Let $q\geq c(\delta) N^{\delta}$ be an integer such that $N|q$ and let  $0<a_1,\ldots, a_m<q-1$ be pairwise distinct integers coprime to $q$ such that 
$$ \{a_i \modulo q: 1\leq i\leq m\}\subset (\Z/q\Z)^\times,$$
is a subgroup not containing $\{a\in(\Z/q\Z)^\times: a\equiv 1\modulo q/N\}$ (nor $-1$). For $1\leq i\leq m$, let $0<a_i^\ast<q$ be defined by $q|a_ia_i^\ast+1$. 
Then \underline{no} conjugate of  
\begin{equation}\label{eq:specificmatrices2}
\begin{psmallmatrix} a_1 & -\frac{a_1a_1^\ast+1}{q}\\q&-a_1^\ast \end{psmallmatrix}\begin{psmallmatrix} a_2 & -\frac{a_2a_2^\ast+1}{q}\\q&-a_2^\ast \end{psmallmatrix}\cdots \begin{psmallmatrix} a_{m} & -\frac{a_{m}a_{m}^\ast+1}{q}\\q&-a_{m}^\ast \end{psmallmatrix}\in \Gamma_0(N),\end{equation}
is contained in the subgroup of $\Gamma_0(N)$ generated by the matrices $\sigma_1,\ldots, \sigma_{2g}$.

\end{cor}
From the analogy with CM elliptic curves this should be compared with \cite[Corollary 1.3]{LMY15}, which as explained in the introduction of \emph{loc.\ cit.}\ one should think of as an analogue of Linnik's Theorem on the least prime in an arithmetic progression (see \cite[Section 2]{Nor22} for an in-depth discussion). We also obtain results in the case where $-1\in H$ if we force the traces of the image of $\psi$ to be unbalanced (see Corollaries \ref{cor:-1} and \ref{cor:group-1}).


\subsection{The idea of the proof}
Our proof strategy for Theorem \ref{thm:first} follows Duke's original spectral approach as in \cite{Du88}. The starting point is Weyl's Criterion (Lemma \ref{lem:WC}), which reduces the problem to showing that the \emph{Weyl sums}; 
$$ \sum_{A\in \Cl_K^+}\int_{\CC_A} f(t) d\mu_{A}(t)$$
associated to an element $f$ of an orthonormal Hecke basis of the orthogonal complement of the constant function are negligible compared to the contribution from the constant;
\begin{equation}\label{eq:Siegel}\sum_{A\in \Cl_K^+}|\CC_A|=2|\Cl_K^+|\log \epsilon_K\gg_\eps D_K^{1/2-\eps},\end{equation}
where $\epsilon_K$ is the fundamental unit of $K$, and the ineffective lower bound is due to Siegel (via the class number formula). A slight reinterpretation of Duke's original treatment of the Weyl sums is then to relate their square 
to the product of $L$-functions 
$$L(1/2,f)L(1/2,f , \chi_K),$$ 
where $\chi_K$ denotes the quadratic character associated to $K$ via class field theory. This is obtained using an explicit version of Waldspurger's formula \cite{Waldspurger85} (actually Duke related the Weyl sums to Fourier coeffecients of half integral weight Maa{\ss} forms using an explicit correspondence of Maa{\ss}). The result now follows from a subconvexity bound for $L(1/2,f , \chi_K)$ in terms of the discriminant of $K$. This was obtained by Duke building on a breakthrough of Iwaniec \cite{Iw87}. 

Our approach follows the same route; on the one hand we need \emph{lower} bounds for the contribution from the constant function, and on the other hand we need \emph{upper} bounds for the non-constant Weyl sums. The lower bound for the constant function is achieved via a study of the small scale distribution of $(\tfrac{a}{q},\tfrac{\overline{a}}{q})$ on the torus $(\R/\Z)^2$ where $a\overline{a}\equiv 1\modulo q$, see Section \ref{sec:distri}. To upper bound the Weyl sums we will use as a substitute for Waldspurger's formula the \emph{Birch--Stevens formula} \cite[Proposition 6.1]{No19} which in its original form states;
\begin{equation}\label{eq:BirchStevens} \tau(\overline{\chi})L(E, \chi ,1)= \sum_{a\in (\Z/q\Z)^\times} \overline{\chi(a)}2\pi i \int_{a/q}^{i\infty} f_E(z)dz,   \end{equation}
where $\chi$ is a primitive Dirichlet character modulo $q$, $\tau(\overline{\chi})$ is a Gauss sum and $f_E$ is a weight 2 holomorphic form associated to an elliptic curve $E/\Q$. The period integrals above are known as \emph{modular symbols}. This formula generalizes to non-primitive characters and to general Hecke--Maa{\ss} forms  $f$ by taking in place of the modular symbol, the central value of the \emph{additive twist  $L$-series} of $f$ defined by analytic continuation of
$$ L(s,f,x)=\sum_{n\geq 1} \frac{\lambda_f(n)e(nx)}{n^s},\quad x\in \Q,$$
see Proposition \ref{BirchStevens}. The key insight in the present work (Proposition \ref{prop:main}) is a relation between the geodesic integrals
$$\int_{\CC^\psi_a(N)} f(t) d\mu_a^\psi(t)$$
and the additive twists $L(1/2,f,x)$ where $x=\psi(a)\infty$. Proving this requires a delicate analytic estimate relying on the Fourier expansion of automorphic forms and is responsible for the fact that we only get $\log$-savings in (\ref{eq:first}). 
The result now follows by the subconvexity bound $L(1/2,f,\chi)\ll_{f,\eps} q^{3/8+\eps}$ for $\chi \modulo q$ due to Blomer and Harcos \cite{BlomerHarcos08} as well as another application of the fine scale distribution of $(\tfrac{a}{q},\tfrac{\overline{a}}{q})$ on the torus $(\R/\Z)^2$ to deal with the singular case where $(a+\overline{a}\modulo q)$ is {\lq\lq}small{\rq\rq}. The condition that $|H|>q^\delta$ with $\delta>7/8$ in Theorem \ref{thm:first} comes exactly from the subconvexity exponent available, see the proof of Theorem \ref{thm:weylsums} for more details. If one assumes the best-possible bound, the Generalized Lindel\"{o}f Hypothesis, then equidistribution holds as long as $\delta>1/2$. We however believe that any $\delta>0$ suffices, see Conjecture \ref{conj:sparse} below.    

In the proof of Theorem \ref{thm:second} we apply a vector space version of Weyl's Criterion (Lemma \ref{lem:WCvector}) which reduces the problem to checking a certain convergence after pairing with a basis for the dual of $H_1(Y_0(N),\R)$. This dual is given by weight $2$ holomorphic forms and the main term comes from the Eisenstein element. The contributions from holomorphic cuspidal Hecke eigenforms of weight $2$ are bounded exactly as in the context of Theorem \ref{thm:first} (in fact the Weyl sums are literally the same). The lower bound for the Eisenstein contribution is achieved using an explicit version of the Birch--Stevens formula for the Eisenstein class as well as a positivity argument and lower bounds for Dirichlet $L$-functions at $s=1$, see Section \ref{sec:Eiscontri}.
\subsection{Final remarks}\label{sec:conjectures}
We will end the introduction by raising a number of natural questions in view of the real quadratic analogue (i.e. Duke's Theorem for closed geodesics). 

\begin{conj}\label{conj:sparse}Fix $\delta>0$. The conclusion of Theorem \ref{thm:first} holds for any subgroup $H\leq (\Z/q\Z)^\times$ such that $|H|\geq q^{\delta}$.\end{conj}
This is a $q$-orbit analogue of the sparse equidistribution conjecture of Michel--Venkatesh \cite[Conjecture 1]{MichelVenk06} stated in the context of Heegner points. Secondly it seems natural to try to obtain a power saving as was obtained by Duke in his original paper \cite{Du88}.
\begin{conj}The conclusion of Theorem \ref{thm:first} holds with a power saving error-term (in terms of $q$).\end{conj}

Finally, it would be interesting to obtain a statement in the level aspect, meaning a version of Theorem \ref{thm:first} where the level $N$ is allowed to vary. This was studied by Liu--Masri--Young \cite{LiuMasriYoung13} in the context of Heegner points and by Humphries and the author in \cite{HumphriesNordentoft22} in the context of hyperbolic orbifolds associated to ideal classes of real quadratic fields.
\section*{Acknowledgements}
The author would like to thank Sary Drappeau, Daniel Kriz and Farrell Brumley for useful discussions, as well as Igor Shparlinski for suggesting the simplified proof of Corollary \ref{cor:Klooster}. The author would also like to give a special thanks to the referees for exceptionally thorough readings, insightful remarks and for pointing out many imprecisions. The comments helped to raise the quality of the paper significantly. The author's research was supported by the Independent Research Fund Denmark DFF-1025-00020B.



\section{Background}
Let 
$$\Gamma_0(N):=\left\{ \gamma\in \PSL_2(\Z): \gamma\equiv \begin{psmallmatrix}\ast&\ast\\ 0&\ast\end{psmallmatrix}\modulo N\right\}\leq \PSL_2(\R),$$ be the Hecke congruence group of level $N$ considered here as a subgroup of $\PSL_2(\R)$. The group $\Gamma_0(N)$ acts on $\Hb\cup \mathbf{P}^1(\R)$ by linear fractional transformations which preserves the cusps $\mathbf{P}^1(\Q)$. Given a matrix in $\SL_2(\R)$ we will without further mentioning consider it as an element of $\PSL_2(\R)$ by taking the image under the projection map.  

\subsection{Closed geodesics on modular surfaces}\label{sec:closedgeo} It follows from a general fact for Fuchsian groups that there is a one-to-one correspondence;
\begin{equation}\label{eq:geoconj} \{\text{oriented closed geodesics on $Y_0(N)$}\}\leftrightarrow \mathrm{hConj}(\Gamma_0(N)),\end{equation}
where $\mathrm{hConj}(\Gamma_0(N))$ denotes the set of hyperbolic conjugacy classes of $\Gamma_0(N)$. Explicitly, given a hyperbolic matrix
$$\gamma=\begin{psmallmatrix}a& b\\c&d\end{psmallmatrix}\in \Gamma_0(N)$$ denote by $S_\gamma\subset \Hb$ the axis of $\gamma$, i.e. the fixed semi-circle of $\gamma$ which has endpoints on $\R$ given by
\begin{equation}\tfrac{a-d\pm \sqrt{(a+d)^2-4}}{2c}.\end{equation} 
The bijection (\ref{eq:geoconj}) is now given by mapping the conjugacy class of $\gamma$ to the projection to $Y_0(N)$ of the oriented geodesic segment connecting $w$ and $\gamma w$ for some $w\in S_\gamma$. Note that this projection is indeed independent of the choice of $w\in S_\gamma$ and depends only on the conjugacy class of $\gamma$. Given an oriented closed geodesic $\CC\subset \Gamma\backslash \Hb$, we can consider the obvious lift to the unit tangent bundle 
$$\mathbf{T}^1(Y_0(N))=\Gamma_0(N)\backslash \PSL_2(\R),$$
recalling that $\Hb\cong \PSL_2(\R)/\mathrm{PSO}_2$. By slight abuse of notation we will also denote the lift by the same symbol (e.g. $\CC$) and will refer to these as \emph{oriented closed geodesics} on $\mathbf{T}^1(Y_0(N))$. Let 
$$a_t:=\begin{psmallmatrix}e^{t/2}& 0 \\ 0 & e^{-t/2} \end{psmallmatrix},\quad t\in \R,$$
and denote by 
$$A:=\left\{a_t: t\in \R \right\}\leq \PSL_2(\R),$$ 
the diagonal subgroup equipped with the Haar measure $dt$. An oriented closed geodesics $\CC\subset \mathbf{T}^1(Y_0(N))$ is a closed and compact $A$-orbit and as such the space of $A$-invariant measures on $\CC$ is one-dimensional. We will denote by $\mu_\CC$ the unique $A$-invariant measure on $\CC$ which descends to the hyperbolic line element $ds$ when $\CC$ is projected to $Y_0(N)$. 

Given an oriented closed geodesic $\CC\subset Y_0(N)$ corresponding to the hyperbolic conjugacy class of $\gamma=\begin{psmallmatrix} a&b\\ c & d\end{psmallmatrix}\in \Gamma_0(N)$, we define a number of key quantities (see Figure \ref{fig:S} for a useful picture to have in mind). Let $\lambda_\gamma> 1$ be defined by $|\tr (\gamma)|=\lambda_\gamma+(\lambda_\gamma)^{-1}$ and put $\ell_\gamma:=\log \lambda_\gamma$. The geodesic length of $\CC$ then satisfies 
\begin{equation}\label{eq:geodesiclength}|\CC|:= \int_\CC ds(z)=2\ell_\gamma\asymp \log |\tr (\gamma) | .\end{equation}      
Furthermore, we define
\begin{equation}\label{eq:epsgamma}\epsilon_\gamma:=\frac{c(a+d)}{|c(a+d)|}\in \{\pm 1\},\qquad r_\gamma:=\frac{\sqrt{(a+d)^2-4}}{2|c|},\end{equation}
(note that $r_\gamma$ is the radius of $S_\gamma$) and 
\begin{equation}\label{eq:ggamma}g_\gamma:=(2r_\gamma)^{-1/2}
\begin{pmatrix} \frac{a-d}{2c}+\epsilon_\gamma r_\gamma &\epsilon_\gamma\frac{a-d}{2c}-r_\gamma\\ 1 & \epsilon_\gamma\end{pmatrix}\in \PSL_2(\R).
\end{equation}
(Note that $c\neq 0$ since $\gamma$ is hyperbolic and that the above quantities do not depend on the matrix in $\SL_2(\Z)$ representing $\gamma\in \PSL_2(\Z)$). A direct calculation with matrices shows that
$$ \gamma g_\gamma a_t =g_\gamma a_{t+|\CC|} ,\quad t\in \R,  $$
which implies that we have the following standard parametrization of $\CC$:
\begin{equation}\label{eq:param}
\int_\CC \phi(t) d\mu_\CC(t)=\int_{-|\CC|/2}^{|\CC|/2} \phi(g_\gamma a_t) dt=\int_{-\ell_\gamma}^{\ell_\gamma} \phi(g_\gamma a_t) dt,
\end{equation}
for any $\phi:\mathbf{T}^1(Y_0(N))\rightarrow \C$ which is integrable with respect to $\mu_\CC$. Here the left-hand side of (\ref{eq:param}) should be thought of as ``integrating $\phi$ along $\CC$ with unit speed''. Finally we define 
\begin{align}
&y_{\gamma,t}:= \Im (g_\gamma a_t i)=  \tfrac{2r_\gamma}{e^t+e^{-t}},\quad t\in \R,\\
\label{eq:ygammat}&h_\gamma:=y_{\gamma,\pm |\CC|/2}= \tfrac{2r_\gamma}{|a+d|}.\end{align} 
\begin{figure}
  \begin{tikzpicture}
    \begin{scope}
      \clip (4,0) rectangle (-4,4);
      \draw (0,0) circle(3);
      \draw[dotted] (-3.5,3) --(3.2,3);
       \draw[dotted] (-3.5,0.6) --(3.2,0.6);
       \draw[dotted] (-3.2,0.74) --(3.5,0.74);

      \draw (4,0) -- (-4,0);
      \draw (-3.2,0) --(-3.2,8);
      \draw (3.2,0) --(3.2,8);

    \end{scope}
    \draw[dotted] (0,0) --(0,-0.2);
    \draw[dotted] (-3.2,0) --(-3.3,-0.3);
    \draw[dotted] (-3,0) --(-2.8,-0.3);
        \draw[dotted] (3.2,0) --(3.3,-0.3);
    \draw[dotted] (3,0) --(2.8,-0.3);
           \node at (-3.6,-0.24) [below] {$\tfrac{-d}{c}$};
           \node at (0,-0.24) [below] {$\tfrac{a-d}{2c}$};
           \node at (-3.4,3) [left] {$r_\gamma$};
           \node at (-3.4,0.6) [left] {$h_\gamma$};
           \node at (3.4,0.8) [right] {$c^{-1}$};
        \node at (3.6,-0.24) [below] {$\tfrac{a}{c}$};
         \node at (2.0,-0.24) [below] {$\tfrac{a-d}{2c}+\epsilon_\gamma r_\gamma$};
         \node at (0,3) [below] {$S_\gamma$};

         \node at (-2.0,-0.24) [below] {$\tfrac{a-d}{2c}-\epsilon_\gamma r_\gamma$};
  \end{tikzpicture}
  \caption{Key quantities associated to $\gamma=\begin{psmallmatrix} a&b\\c&d\end{psmallmatrix}$ (with the orientation corresponding to the case where $\epsilon_\gamma>0$).}\label{fig:S}
\end{figure}  
\subsubsection{Double coset embeddings}\label{sec:DCE}
The packets (or $q$-orbits) of oriented closed geodesics (\ref{eq:qorbit}) that we are considering are determined by maps $\psi:(\Z/q\Z)^\times \rightarrow \Gamma_0(N)$ satisfying the conditions (\ref{item:1}), (\ref{item:2}) above. Clearly, these conditions do not determine the packets uniquely. In this section we will characterize all possible packets of geodesics that can be obtained and prove some basic properties.
 
More precisely, given two positive integers $N,q$ such that $N|q$, we will now give a complete description of the possible compositions 
\begin{equation}\label{eq:doublecosetchar}(\Z/q\Z)^\times\xrightarrow{\psi} \Gamma_0(N)\xrightarrow{\pi_N}\mathrm{Conj}(\Gamma_0(N)),\end{equation} 
where $\psi(a)\infty\equiv \tfrac{a}{q}\modulo 1$ for all $a\in (\Z/q\Z)^\times $
and $ \pi_N:\Gamma_0(N)\rightarrow \Conj(\Gamma_0(N))$ denotes the map 
$$\gamma\mapsto \{\gamma\}_{\Gamma_0(N)},$$ 
sending a matrix to its conjugacy class.
The classical double coset decomposition (see e.g. \cite[Sec. 2.4]{Iw}) gives a bijection of sets:
\begin{equation}\label{eq:docode}\Gamma_\infty\backslash\Gamma_0(N)/\Gamma_\infty\leftrightarrow 
\{\Gamma_\infty\begin{psmallmatrix} 1&  1 \\ 0 & 1\end{psmallmatrix} \Gamma_\infty\}\cup \bigcup_{q\geq 1: N|q}\,\, \bigcup_{a\in (\Z/q\Z)^\times }  \{\Gamma_\infty\begin{psmallmatrix} a&  \ast \\ q & \ast\end{psmallmatrix} \Gamma_\infty \}.
 \end{equation}
From this we see that the matrix $\psi(a)$ is determined by the above conditions up to left and right multiplication by $\Gamma_\infty$ and we arrive at the following characterization.
\begin{lemma}\label{lem:DCE}
Any composition $\pi_N\circ \psi: (\Z/q\Z)^\times \rightarrow \Conj(\Gamma_0(N))$ as in (\ref{eq:doublecosetchar}) is a specialization of the map 
$$ \Psi_N:\Z\oplus\Gamma_\infty\backslash\Gamma_0(N)/\Gamma_\infty\twoheadrightarrow \mathrm{Conj}(\Gamma_0(N)),$$
given by 
$$ \left(n,\Gamma_\infty\begin{psmallmatrix} a&  \ast \\ q & \ast\end{psmallmatrix} \Gamma_\infty \right)\mapsto  \left\{\begin{psmallmatrix} a&  b\\ q & d\end{psmallmatrix}\right\}_{\Gamma_0(N)}, $$
where 
\begin{equation}\label{eq:Psi}ad\equiv 1\, \modulo q,\quad (n-\tfrac{1}{2})q\leq a+d<(n+\tfrac{1}{2})q,\quad b=\tfrac{ad-1}{q}.\end{equation} 
\end{lemma}
\begin{proof} This follows directly from the double coset decomposition (\ref{eq:docode}) since multiplication of $\psi(a)$ by $\Gamma_\infty$ on the left or right shifts the trace by multiples of $q$.\end{proof}
 In other words; the double coset together with the choice of trace define a unique conjugacy class. Notice that all the restrictions of $\Psi_N$ to the first component $\Z$ are injective (since all of the images have different traces). 
In our setting we will also require that the image is contained in the set of hyperbolic conjugacy classes so that they give rise to closed geodesics. Motivated by the above discussion we arrive at the following definition.
\begin{defi}\label{def:psi}
A \emph{double coset embedding of level $q$} is a map $\psi:(\Z/q\Z)^\times \rightarrow \Gamma_0(q)$ satisfying for all $a\in (\Z/q\Z)^\times$ that
\begin{enumerate}
\item \label{item:1.1} $\psi(a)\infty\equiv \tfrac{a}{q}\modulo 1$,
\item \label{item:2.1} $|\tr(\psi(a))|>2 $.
\end{enumerate} 
\end{defi}
Given a double coset embedding $\psi$ of level $q$ and an integer $N|q$, we denote by 
\begin{equation} \label{eq:CCapsiN}\CC_a^\psi(N)\subset \mathbf{T}^1(Y_0(N)), \end{equation}
the oriented closed geodesic of level $N$ associated to the $\Gamma_0(N)$-conjugacy class of $\psi(a)\in \Gamma_0(q)\subset \Gamma_0(N)$ under the identification (\ref{eq:geoconj}). For notational simplicity we will denote the canonical measure on $\CC_a^\psi(N)$ as defined in the previous section by 
\begin{equation}\label{eq:muapsiN} \mu_a^\psi:= \mu_{\CC_a^\psi(N)}, \end{equation}
suppressing the level $N$ in the notation. For $a\in (\Z/q\Z)^\times$ denote by  
\begin{equation}\label{eq:ta} t_a:=q\left(\{\tfrac{a+\overline{a}}{q}+\tfrac{1}{2}\}-\tfrac{1}{2}\right)\in \Z,\end{equation}
the \emph{minimal trace of $a$}, where $\{x\}=x-\lfloor x\rfloor$ denotes the fractional part and $\overline{a}a\equiv 1\modulo q$. Given a double coset embedding $\psi$ of level $q$ and $a\in (\Z/q\Z)^\times $   we define $n_\psi(a)\in \Z$ by the property:  
\begin{equation} \{\psi(a)\}_{\Gamma_0(N)}= \Psi_N\left(n_\psi(a),\Gamma_\infty\begin{psmallmatrix} a&  \ast \\ q & \ast\end{psmallmatrix} \Gamma_\infty \right),\end{equation} 
which makes sense by Lemma \ref{lem:DCE}. The following lemma justifies the fact that we call $t_a$ the minimal trace. 
\begin{lemma}\label{lem:lowerbndtrace}
Let $\psi$ be a double coset embedding of level $q$. Then for all $a\in (\Z/q\Z)^\times $ it holds that 
\begin{align}\label{eq:lowerbndtrace}\tr(\psi(a))=t_a+qn_\psi(a),\quad\text{and}\quad  |\tr(\psi(a))|\geq |t_a|.\end{align}
\end{lemma}
\begin{proof}
We observe that $-q/2\leq t_a<q/2 $ and $t_a\equiv a+\overline{a}\modulo q$. Thus the first equality in (\ref{eq:lowerbndtrace}) follows by the definition (\ref{eq:Psi}) of $\Psi_N$ combined with that of $n_\psi(a)$. Furthermore we see that indeed $|t_a|$ is the minimal possible absolute trace of $\psi(a)$ for any $\psi$ satisfying (\ref{item:1.1}) which implies the lower bound in (\ref{eq:lowerbndtrace}).
\end{proof}


\subsection{Spectral theory of automorphic forms}
For a detailed account of the following material we refer to~\cite[Chapter~2]{Iw}, \cite[Section 4]{DuFrIw02}. For a cusp $\mathfrak{a}\in \mathbb{P}(\Q)$ of $\Gamma_0(N)$, let $\Gamma_\mathfrak{a}\subset \Gamma_0(N) $ be the stabilizer of $\mathfrak{a}$ and fix a scaling matrix $\sigma_\mathfrak{a}$, i.e. any matrix satisfying 
$$\sigma_\mathfrak{a}^{-1} \Gamma_\mathfrak{a} \sigma_\mathfrak{a} = \{\begin{psmallmatrix} 1 & \Z \\ & 1\end{psmallmatrix} \}.$$

For $k\in 2\Z$, we denote by $\AA(\Gamma_0(N), k)$ the vector space of all \emph{weight $k$ automorphic forms of level $N$ (with trivial nebentypus)}, i.e. smooth maps $f: \Hb \rightarrow \C$ satisfying for all~$\gamma \in \Gamma_0(N)$ and~$z\in\Hb$:
  \begin{equation}f(\gamma z) = j_\gamma(z)^k f(z),\qquad 
    j_\gamma(z) = \frac{j(\gamma, z)}{|j(\gamma,z)|}= \frac{cz+d}{|cz+d|}.\label{eq:def-ugamma}
  \end{equation}
Let
\begin{equation}\label{eq:levelraising} R_k = \frac k2 + (z-\bar z) \frac{\partial}{\partial z},\quad \Lambda_k = \frac k2 + (z-\bar z) \frac{\partial}{\partial \overline{z}}  \end{equation}
be respectively the weight~$k$ raising and  lowering operator, as defined in~\cite[eqs.~(4.3)-(4.4)]{DuFrIw02}. These define maps 
$$R_k:\AA(\Gamma_0(N) , k)\rightarrow \AA(\Gamma_0(N) , k+2),\quad \Lambda_k:\AA(\Gamma_0(N) , k)\rightarrow \AA(\Gamma_0(N) , k-2).$$ The weight~$k$ Laplacian is defined by
$$ -\Delta_k :=R_{k-2} \Lambda_{k}+\frac{k}{2}\left(1-\frac{k}{2}\right)=\Lambda_{k+2} R_{k}-\frac{k}{2}\left(1+\frac{k}{2}\right)= -y^2\Big(\frac{\partial^2}{\partial x^2} + \frac{\partial^2}{\partial y^2}\Big) + iky \frac{\partial}{\partial x}. $$
For $k\in 2\Z_{\geq 0}$ and $s\in \C$, we define  an operator 
\begin{equation}\label{eq:involution}
  Q_{s,k}:\AA(\Gamma_0(N),  k)\rightarrow \AA(\Gamma_0(N),  k),
\end{equation} 
as in~\cite[eq.~(4.65)]{DuFrIw02} by 
$$ (Q_{s,k}f )(z)=\frac{\Gamma(s-k/2)}{\Gamma(s+k/2)} (\Lambda_{-k+2} \cdots\Lambda_{k-2} \Lambda_{k} f )(-\overline{z}),   $$
where we put $Q_{s,k}=0$ if~$s-k/2\in \Z_{\leq 0}$. Notice that for $k=0$ we have $(Q_{s,0}f )(z)=f (-\overline{z})$ which is the usual reflection operator. The operator $Q_{s,k}$ preserves the eigenspace of $-\Delta_k$ with eigenvalue $s(1-s)$, and is an involution for $s\not\in k/2+\Z_{\leq 0}$.
Similarly, we define $Q_{s,k}$ for negative $k$ using the raising operators.  

We say that $f  \in \AA(\Gamma_0(N), k)$ is a \emph{Maa{\ss} form} if it satisfies;
\begin{itemize}
\item For all~$\mathfrak{a} \in \mathbb{P}(\Q)$, $f (\sigma_\mathfrak{a} z) = o(e^{2\pi y})$ as $y = \Im z\rightarrow \infty$.
\item $f $ is an eigenfunction of $-\Delta_k$ with eigenvalue $\lambda_f  = s_f (1-s_f ) = 1/4+t_f ^2$ with $\Re s_f  \geq 1/2$ and $s_f  = 1/2+it_f $.
\item $f $ is an eigenfunction of the operator~$Q_{s_f ,k}$ with eigenvalue~$\epsilon_f  \in \{\pm 1,0\}$. Note that if $k=0$ then $\epsilon_f $ is the \emph{sign} of the Maa{\ss} form~\cite[p.~106]{Bump97}, and if $f $ comes from a holomorphic form then $\epsilon_f =0$. \end{itemize}

We define the \emph{spectral datum} of a Maa{\ss} form $f\in \AA(\Gamma_0(N), k)$ as
\begin{equation}\label{eq:tf}\mathbf{t}_f:=(it_f,k,\epsilon_f),\end{equation}  
and we define the following convenient quantity measuring the archimedean complexity of $f$;
\begin{equation}
|\mathbf{t}_f|:=|t_f|+|k|+1.
\end{equation}
Recall that the Whittaker function $W_{\alpha,\beta}:\R_{>0}\rightarrow \C$ with parameters $\alpha,\beta\in \C$ is the unique solution $W$ to
$$  \frac{d^2 W}{dy^2}+ \left(-\frac{1}{4}+\frac{\alpha}{y}+\frac{1/4-\beta^2}{y^2}\right)W=0, $$
satisfying $W(y)\sim y^{\alpha}e^{-y/2}$ as $y\rightarrow \infty$ (with $\alpha,\beta$ fixed). In particular, we have 
\begin{equation}\label{eq:specialWhit}W_{0,\beta}(4\pi y)=2 y^{1/2}K_{\beta}(2\pi y),\quad W_{\alpha,\alpha-1/2}(4\pi y)= (4\pi y)^{\alpha}e^{-2\pi y}\end{equation} 
where $K_\beta(y)$ denotes the $K$-Bessel function.  We consider the following convenient normalization of the Whittaker function: 
\begin{equation}\tilde{W}_{\alpha ,\beta}(y):=\frac{W_{\sgn(y)\alpha,\beta }(4\pi |y|)}{\left|\Gamma(\frac{1}{2}+\beta+\sgn(y)\alpha)\Gamma(\frac{1}{2}-\beta+\sgn(y)\alpha)\right|^{1/2}} ,\quad y\in \R^\times,
\end{equation}
where the right-hand side is understood as zero if $\frac{1}{2}\pm \beta+\sgn(y)\alpha \in \Z_{\leq 0}$ (cf.\ \cite[eq.\ (22)]{BlomerHarcos08.2} where we are using a trivial phase function). As explained in \cite[Section 4]{BruggemanMotohashi05} the functions $\tilde{W}_{\alpha ,\beta}(y)$ are $L^2$-normalized with respect to the measure  $\frac{dy}{y}$ on $\R^\times$. 
The above conditions on the Maa{\ss} form $f$ imply by \cite[eq.\ (4.36)]{DuFrIw02} that at any cusp $\mathfrak{a}\in \mathbf{P}^1(\Q)$ we have the Fourier expansion
\begin{equation}\label{fexp}
  f(\sigma_\mathfrak{a} z)= f_\mathfrak{a}(y)+\sum_{n\neq 0} a_{f,\mathfrak{a}}(n)  \tilde{W}_{\frac k2 ,it_f }(ny)e(nx),\qquad z=x+iy, \quad e(x)=e^{2\pi i x},
\end{equation}
with constant term of the shape
\begin{equation}\label{eq:constantterm}
  f_\mathfrak{a}(y)= \begin{cases}A_\mathfrak{a} y^{1/2+it_f }+B_\mathfrak{a} y^{1/2-it_f },& t_f \neq 0,\\ A_\mathfrak{a} y^{1/2}+B_\mathfrak{a} y^{1/2}\log y ,& t_f = 0, \end{cases}
\end{equation}
for some constants $A_\mathfrak{a},B_\mathfrak{a}\in \C$. For 
$$\alpha\in \Z, \quad\text{and}\quad \beta\in i\R \cup (0,1/2)\cup \{\ell-\tfrac{1}{2}:\ell\in \Z_{>0} \},$$ 
we have the following uniform bounds due to Blomer--Harcos \cite[eq.\ (24)-(26)]{BlomerHarcos08.2}:
\begin{align}\label{eq:whitbound1} \tilde{W}_{\alpha,\beta}(y)\ll |y|^{1/2}\left(\frac{|y|}{|\alpha|+|\beta|+1}\right)^{-1-\Re \beta}\exp\left(-\frac{|y|}{|\alpha|+|\beta|+1}\right),  \end{align}
and for $0<\eps<1/4$
\begin{align}\label{eq:whitbound2} \tilde{W}_{\alpha,\beta}(y)\ll_\eps |y|^{1/2-\mathfrak{d}_\beta-\eps}\left(|\alpha|+|\beta|+1\right)^{1+\mathfrak{d}_{\beta} },  \end{align}
where $\mathfrak{d}_\beta=\beta$ if $0<\beta<1/2$ (corresponding to the case of complementary series) and $\mathfrak{d}_\beta=0$ otherwise. For a Maa{\ss} form $f$ we define
\begin{equation}\label{eq:kappa} \mathfrak{d}_f:=\mathfrak{d}_{it_f}.\end{equation}  

We will be interested in the spectral theory of the $L^2$-space
$$ \LL^2(\Gamma_0(N),k):=\overline{\{f\in\AA(\Gamma_0(N),k): \langle f,f\rangle <\infty \}},$$
where 
\begin{equation}  \langle f,g\rangle =\int_{\Gamma_0(N)\backslash \Hb} f(z)\overline{g(z)}\frac{dxdy}{y^2}, \end{equation}
is the Petersson inner-product (suppressing $N$ in the notation). We will say that $f$ is \emph{cuspidal at the cusp $\mathfrak{a}$} if $f _\mathfrak{a}(y)=0$. Finally we say that $f $ is a \emph{Maa{\ss} cusp form} if it is cuspidal at all cusps $\mathfrak{a}\in \mathbb{P}(\Q)$. From~\cite[Corollary~4.4]{DuFrIw02} we know that Maa\ss{} cusp forms arise in two ways:
\begin{itemize}
\item $\Re(s_f )<1$ and $f $ is obtained from repeated applications of raising or lowering operators from a weight~$0$ Maa{\ss} cusp form. 
\item $s_f  = \ell/2$ with~$\ell\equiv k\pmod{2}$, and~$f $ is associated, through raising or lowering operators, to a form~$g$ of weight~$\ell$ for which~$z\mapsto y^{-\ell/2} g(z)$ is holomorphic. 
\end{itemize}


We will abbreviate throughout $ a_f(n) = a_{f,\infty}(n)$. Moreover, we have by~\cite[eq.~(4.70)]{DuFrIw02} (taking into account our normalization) \begin{equation}\label{eq:fouriercoeff-sign}
  a_f(-n) = \epsilon_f  \frac{\left|\Gamma(s_f-\frac{k}{2})\Gamma(1-s_f-\frac{k}{2})\right|^{1/2} \Gamma(s_f+\frac{k}{2})}{\left|\Gamma(s_f+\frac{k}{2})\Gamma(1-s_f+\frac{k}{2}) \right|^{1/2} \Gamma(s_f-\frac{k}{2})}a_f(n)=:\kappa_f a_f(n).
\end{equation}
Note that $\kappa_f=0$ for $f$ associated to a holomorphic form, whereas in the remaining cases we have $s_f\notin \Z$ and so by 
 the reflection formula
\begin{equation}\label{eq:kappaf}|\kappa_f|=\left(\frac{\left|\Gamma(s_f+\frac{k}{2})\Gamma(1-s_f-\frac{k}{2})\right|}{\left|\Gamma(s_f-\frac{k}{2})\Gamma(1-s_f+\frac{k}{2})\right|}\right)^{1/2} =\left(\frac{\left|\pi/\sin \pi(s_f+\frac{k}{2})\right|}{\left|\pi/\sin \pi(s_f-\frac{k}{2})\right|}\right)^{1/2}=1,  \end{equation}
using that $\epsilon_f\in \{\pm 1\}$ and the $2\pi$-periodicity of sine.
\subsubsection{Eisenstein series}
Examples of non-cuspidal Maa{\ss} forms of even weight $k$ include the Eisenstein series. These are not $L^2$-integrable but will show up in the spectral decomposition of $\LL^2(\Gamma_0(N),k)$. There are two natural bases for the space of Eisenstein series. The first basis is parametrized in terms of the $\Gamma_0(N)$-equivalence classes of cusps $\mathfrak{a}\in \mathbf{P}^1(\Q)$; 
$$E_{\mathfrak{a},k}(z,s):=\sum_{\gamma\in \Gamma_\mathfrak{a}\backslash \Gamma_0(N)} j_{\sigma_\mathfrak{a}^{-1}\gamma}(z)^{-k}\Im (\sigma_\mathfrak{a}^{-1}\gamma z)^s, \quad \Re s>1,$$
which satisfies meromorphic continuation with no poles on $\Re s=1/2$. The other basis is in terms of characters. Let $ m^2|N$ and let $\chi$ be a Dirichlet character modulo $m$. Then the following Eisenstein series  belongs to $\AA(\Gamma_0(N),k)$;
\begin{equation}\label{eq:Eisensteinchar}E_{\chi,k}(z,s):=\frac{1}{2}\sum_{(c,d)=1} \left(\frac{cmz+d}{|cmz+d|}\right)^
{-k}\frac{(my)^s\chi(c)\chi(d)}{|cmz+d|^{2s}},\quad s\in \C, \end{equation}
see \cite[eq.\ (3.3)]{Young19}, with Laplace eigenvalue $s(1-s)$ and $Q_{s,k}E_{\chi,k}=\chi(-1)E_{\chi,k}$. 
When $\chi \modulo m$ is primitive and $s=\tfrac{1}{2}+it$ then we have the following Fourier expansion \cite[Proposition 4.1]{Young19} for the Eisenstein series:
\begin{align}\nonumber \tfrac{i^k(m/\pi)^s}{\tau(\chi)} L(2s,\chi^2) E_{\chi, k}(z,s)=& i^k\delta_{m=1}\left(\pi^{-s} \zeta(2s)y^s+\pi^{s-1} \tfrac{\Gamma(1-s+\frac{k}{2})}{\Gamma(s+\frac{k}{2})}\zeta(2-2s)y^{1-s}\right)\\
\label{eq:EisensteinFourier}&+\sum_{n\neq 0} \frac{\lambda_{it}(n,\chi)}{|n|^{1/2}} \tfrac{|\Gamma(s+\sgn(n)\frac{k}{2})\Gamma(1-s+\sgn(n)\frac{k}{2})|^{1/2}}{\Gamma(s+\sgn(n)\frac{k}{2})}\tilde{W}_{\frac{k}{2},it}(n y)e(nx),  \end{align}
where $\delta_{m=1}=1$ if $\chi$ is the trivial character and $0$ otherwise, and
\begin{equation}
\lambda_{it}(n,\chi):=\chi(\sgn(n))\sigma_{it,it}(|n|,\chi,\chi),\end{equation} with
\begin{equation}\label{eq:sigmachichi} \sigma_{s_1,s_2}(n,\chi_1,\chi_2):=[(\chi_1\,|\cdot|^{-s_1})\ast (\overline{\chi_2}\,|\cdot|^{s_2})](n)=\sum_{ab=n}\chi_1(a)a^{-s_1}\overline{\chi_2(b)}b^{s_2}, \quad n\geq1. \end{equation}
It will be convenient for us to consider the following renormalization:
\begin{equation}\label{eq:EisensteinFouriertilde}
\tilde{E}_{\chi, k}(z,s):=\tfrac{\Gamma(s+\frac{k}{2})}{|\Gamma(s+\frac{k}{2})\Gamma(1-s+\frac{k}{2})|^{1/2}} \tfrac{(m/\pi)^s}{i^k\tau(\chi)} L(2s,\chi^2) E_{\chi, k}(z,s),
\end{equation}
which by the above satisfies $a_{\tilde{E}_{\chi, k}}(1)=1$, i.e.\ the first Fourier coefficient is one. This will be useful when studying \emph{additive twist $L$-series} of Eisenstein series in Section \ref{sec:addtwists}. 

Consider now the weight $2$ Eisenstein series $E_{\infty, 2}(z,s)$ for $\PSL_2(\Z)$ associated to the cusp at $\infty$. Then with our normalizations we have $R_0 E_{\infty, 0}(z,s)=sE_{\infty, 2}(z,s)$ and $R_0= 2iy\partial_z $, so that Hecke's limit formula \cite[p. 16]{DuImTo18} amounts to
\begin{align}\label{eq:limitformula} \lim_{s\rightarrow 0} \pi^{-s} \Gamma(s+1) \zeta(2s)E_{\infty, 2}(z,s)=\frac{\pi  y}{6} E^\ast_2(z),  \end{align}
where 
\begin{align}\label{eq:E2ast}E^\ast_2(z)=1-\frac{3}{\pi y}-24\sum_{n\geq 1}\sigma_1(n)e^{2\pi i nz},\end{align}
is the \emph{modified Eisenstein series of weight $2$}. Here $\sigma_1(n)=\sum_{d|n}d$ denotes the sum of divisors function. Note that by Hecke's limit formula $E^\ast_2(z)$ is a (non-holomophic) modular form of weight $2$ and level $1$, meaning that $E^\ast_2(\gamma z)=j(\gamma,z)^2E^\ast_2(z)$ for $\gamma\in \PSL_2(\Z)$.

\subsection{Lift to the group}
Consider the homogeneous space $\Gamma_0(N)\backslash \PSL_2(\R)$, which we identified with the unit tangent bundle $\mathbf{T}^1(Y_0(N))$, equipped with Haar measure $dg$, the Petersson inner product
$$\langle F,G\rangle = \int_{\Gamma_0(N)\backslash \PSL_2(\R)}F(g)\overline{G(g)}dg,$$ 
and the Laplace--Beltrami operator
$$\Delta= -y^2 \left(\frac{\partial^2}{\partial x^2}+\frac{\partial^2}{\partial y^2}\right)+y\frac{\partial^2}{\partial x\partial \theta},$$
in terms of the Iwasawa coordinates
$$g=\begin{psmallmatrix} 1 & x\\ 0 & 1 \end{psmallmatrix}\begin{psmallmatrix} y^{1/2} & 0\\ 0 & y^{-1/2} \end{psmallmatrix}\begin{psmallmatrix} \cos\theta & \sin\theta \\ -\sin\theta & \cos\theta \end{psmallmatrix}.$$
We denote by $\LL^2(\Gamma_0(N)\backslash \PSL_2(\R), dg)$ the associated $L^2$-space. For $k\in 2\Z$ we have a Hilbert space embedding  
\begin{align}\label{eq:weightkembedding}\LL^2(\Gamma_0(N),k) &\hookrightarrow \LL^2(\Gamma_0(N)\backslash \PSL_2(\R), dg),\\ f &\mapsto (g\mapsto j_g(i)^{-k}f(g.i)),\end{align}
with image equal to the weight $k$ isotypic component of $\LL^2(\Gamma_0(N)\backslash \PSL_2(\R), dg)$ under the (right regular) action of the maximal compact subgroup
\begin{equation} K=\mathrm{PSO}_2=\left\{\begin{psmallmatrix} \cos\theta & \sin\theta \\ -\sin\theta & \cos\theta \end{psmallmatrix}: 0\leq \theta<2\pi\right\}/\{\pm I\},\end{equation}  
i.e. for $F\in \LL^2(\Gamma_0(N)\backslash \PSL_2(\R), dg)$ in the image of the weight $k$ embedding  (\ref{eq:weightkembedding}) we have 
$$F\left(g \begin{psmallmatrix} \cos\theta & \sin\theta \\ -\sin\theta & \cos\theta \end{psmallmatrix}\right)=e^{ik\theta}F(g).$$
This embedding intertwines the action of the weight $k$ Laplacian $-\Delta_k$ and the Laplace--Beltrami operator $\Delta$. We say again that the image under this embedding are of \emph{weight $k$} and that the image of a  Maa{\ss} form is again a \emph{Maa{\ss} form}. We similarly extend the definition of \emph{cuspidal}, \emph{spectral datum} in the obvious way, and put 
$\mathbf{t}_F=\mathbf{t}_f$ if $F$ is the lift of $f$ (with notation as in (\ref{eq:tf})).
 
\subsubsection{Hecke operators}
For $n\geq 1$ we define the \emph{$n$th Hecke operator of level $N$} as the map 
$$T_n:\CC^\infty(\Gamma_0(N)\backslash \PSL_2(\R))\rightarrow \CC^\infty(\Gamma_0(N)\backslash \PSL_2(\R)),   $$
defined by (cf.\ \cite[equation (6.1)]{DuFrIw02})
\begin{equation}\label{eq:heckeoperator} (T_n F)(g)= \frac{1}{n^{1/2}}\sum_{\substack{ad=n\\(a,N)=1}}\sum_{0\leq b<d} F\left(\begin{psmallmatrix} a & b\\0&d\end{psmallmatrix}g\right).\end{equation}
We say that a Maa{\ss} form $F\in \CC^\infty(\Gamma_0(N)\backslash \PSL_2(\R))$ is a \emph{Hecke--Maa{\ss} form} if $F$ is an eigenfunction for all $T_n$ with $(n,N)=1$. For $\ell,N'\geq 1$ such that $\ell N'|N$, we define an operator
\begin{equation}\nu_{\ell,N'}:\CC^\infty(\Gamma_0(N')\backslash \PSL_2(\R))\rightarrow \CC^\infty(\Gamma_0(N)\backslash \PSL_2(\R)), \end{equation}
by
\begin{equation}\label{eq:nulevel} (\nu_{\ell,N'}F)(g):=F(\ell g)=F\left(\begin{psmallmatrix} \ell^{1/2}& 0 \\ 0 & \ell^{-1/2} \end{psmallmatrix} g\right),\quad g\in \PSL_2(\R). \end{equation}
We say that a Hecke--Maa{\ss} form $F$ is a \emph{Hecke--Maa{\ss} newform of level $N$} if it is not contained in the span of the images of the maps $\nu_{\ell,N'}$ with  $N'<N$ and $\ell N'|N$ and normalized so that the first Fourier coefficient (at $\infty$) is $1$. Note that the rescaled Eisenstein series $\tilde{E}_{\chi,k}(z,s)$ as defined in (\ref{eq:EisensteinFouriertilde}) with $\chi\mod m$ a primitive Dirichlet character is a Hecke--Maa{\ss} newform with Hecke eigenvalues $\lambda_{it}(n,\chi)$. For $k\in 2\Z$ we denote the set of cuspidal newforms by 
$$\BB^\ast_k(\Gamma_0(N)):=\{F\in \CC^\infty(\Gamma_0(N)\backslash \PSL_2(\R)): \text{cuspidal Hecke--Maa{\ss} newform of weight $k$}\}.$$ 
Let $F\in \BB^\ast_k(\Gamma_0(N))$ be a Hecke--Maa{\ss} newform and let $L(s,\sym^2 F)$ denote the symmetric square $L$-function satisfying
\begin{equation}\label{eq:HoffsteinLockhart}  (N+|\mathbf{t}_F|)^{-\eps} \ll_\eps L(1,\sym^2 F)\ll_\eps (N+|\mathbf{t}_F|)^\eps,\end{equation}
for any $\eps>0$ by the classical estimates of Hoffstein--Lockhart \cite[Theorem 0.2]{HoffLock94} and \cite[Proposition 19.6]{DuFrIw02}. It follows from the Rankin--Selberg method \cite[Section 2.1]{BlomerHarcos08.2} that there exists a constant $c_N>0$ depending only on the level $N$ such that
\begin{equation}\label{eq:RankinSelberg} \langle F,F \rangle= c_N L(1,\sym^2 F), \end{equation}
using here that $\tilde{W}_{\frac{k}{2} ,it_f}$ is $L^2$-normalized with respect to the standard Haar measure $\tfrac{dy}{y}$ on $\R^\times$ as mentioned above. The upshot is that since the level is fixed the Hecke normalization and the $L^2$-normalization are essentially equivalent (up to a factor of size $|\mathbf{t}_F|^\eps$).      

The spectral theorem \cite[Proposition 4.1+4.2]{DuFrIw02} states that for a smooth function \\$\phi:\Gamma_0(N)\backslash \PSL_2(\R)\rightarrow \C$ of compact support we have the pointwise equality:
\begin{align}\label{eq:spectralthm}\phi(g)= \frac{\langle \phi,1\rangle}{\langle 1, 1\rangle}+
\sum_{F\in \BB} \langle \phi,F\rangle F(g)+\frac{1}{4\pi}\sum_{k\in 2\Z}\sum_\mathfrak{a} \int_{\R} \langle \phi, E_{\mathfrak{a},k,it}\rangle E_{\mathfrak{a},k,it}(g) dt,    
\end{align}
where $\BB$ is an orthonormal basis for the cuspidal subspace of $\LL^2(\Gamma_0(N)\backslash \PSL_2(\R),dg)$, $\mathfrak{a}$ runs over $\Gamma_0(N)\backslash\mathbb{P}^1(\Q)$ (i.e. inequivalent cusps of $\Gamma_0(N)$) and 
$$E_{\mathfrak{a},k,it}(g):=j_g(i)^{-k}E_{\mathfrak{a},k}(g.i , \tfrac{1}{2}+it),$$
with $E_{\mathfrak{a},k}$ defined as above. Here we are using that there is no non-constant residual spectrum for $\Gamma_0(N)$. It is a classical fact of Atkin--Lehner theory that the cuspidal part of the $L^2$-space is spanned by cuspidal Hecke--Maa{\ss} forms and by \cite[Lemma 9]{BlomerMilicevic15} we have the following orthonormal basis for the cuspidal subspace:
\begin{equation}\label{eq:BB0}\BB_0(\Gamma_0(N)):=\bigcup_{k\in 2\Z} \bigcup_{\ell N'|N}\{F_\ell: F\in\BB^\ast_k(\Gamma_0(N')\},\end{equation}
where
\begin{align} \label{eq:coefONB}F_\ell(g):= \sum_{d|\ell} c_{F}(d,\ell)F\left(dg\right), \end{align}
with $c_{F}(d,\ell)\in \C$ satisfying $ |c_{F}(d,\ell)|\ll_{N,\eps} |\mathbf{t}_F|^\eps$ (this last bound follows from the explicit formula for the basis given in \cite[Proposition 2.6]{ILS00} and the equality (\ref{eq:RankinSelberg}) in view of the estimate (\ref{eq:HoffsteinLockhart})). Similarly by combining \cite[Theorem 8.1]{Young19} with the last equation in Section 8.2 of \emph{loc.\ cit.}, the Eisenstein part of the spectrum has an orthonormal basis (in the sense of continuous spectrum) given by the packets 
\begin{equation}\label{eq:BBE}\BB_e(\Gamma_0(N)):=\bigcup_{m^2|N}\,\,\bigcup_{\substack{\chi\modulo m\\\text{primitive}}}\, \bigcup_{\ell|Nm^{-2}} \{E_{\chi,k,it,\ell}: t\in \R\} ,\end{equation}
where
\begin{equation}\label{eq:Eisensteinlift}E_{\chi,k,it,\ell}(g):=\sum_{d|\ell}c_{t,\chi}(d,\ell)j_g(i)^{-k}\tilde{E}_{\chi,k}\left(d\,g.i,1/2+it\right),\end{equation}
and the coefficients satisfy $|c_{t,\chi}(d,\ell)|\ll_{N,\eps} (|t|+1)^\eps $, which  follows directly from the last equation in \cite[Section 8.2]{Young19}, the definition (\ref{eq:EisensteinFouriertilde}) of $\tilde{E}_{\chi,k}$ and the standard bound $|L(1+2it,\chi^2)|\gg_\eps (|t|+1)^{-\eps}$ (note also that $\Gamma(s+\frac{k}{2})|\Gamma(s+\frac{k}{2})\Gamma(1-s+\frac{k}{2})|^{-1/2}$ has absolute value $1$ for $\Re s=1/2$). By similar considerations the constant term (\ref{eq:constantterm}) in the Fourier expansion at $\infty$ of a basis elements $E\in \BB_e(\Gamma_0(N))$ satisfies for all $\eps>0$ that 
\begin{equation}\label{eq:boundtilde}
|A_{E}|,|B_{E}|\ll_{N,\eps} |\mathbf{t}_E|^\eps.
\end{equation}
\subsection{(Co)homology of modular surfaces}\label{sec:cohom}
Let $N$ be prime  and consider the level $N$ modular surface $Y_0(N)=\Gamma_0(N)\backslash \Hb$. Then the compactification $X_0(N)=\overline{Y_0(N)}$ is a compact Riemann surface of genus $g=\frac{N}{12}+O(1)$ (see e.g. \cite[Proposition 1.40]{Sh94}). Recall that the integral singular homology group sits as a lattice inside the real homology (see e.g. \cite[Chapter 2]{Hatcher02})  
$$H_1(Y_0(N),\Z)\subset H_1(Y_0(N),\R)\cong \R^{2g+1},$$
which in turn sits inside the complex homology group $H_1(Y_0(N),\C)$, and similarly for cohomology. We have the cap product pairing
\begin{equation}\label{eq:cap} \langle\cdot, \cdot \rangle_\mathrm{cap}: H_1(Y_0(N),\Z)\times H^1(Y_0(N),\Z)\rightarrow \Z,\end{equation}
between integral homology and cohomology which identifies $H^1(Y_0(N),\Z)$ with the linear dual $H_1(Y_0(N),\Z)^\ast$. This extends by linearity to a pairing between real and complex (co)homology.  

Given a harmonic differential $1$-form on $Y_0(N)$ we get a complex cohomology class defined by integration against cycles. Explicitly, if $f\in \mathcal{M}_2(\Gamma_0(N))$ is a holomorphic form of weight $2$ and level $N$ then 
$$  2\pi i f(z)dz \quad \text{and}\quad \overline{2\pi i f(z)dz} $$
are  harmonic $1$-forms on $Y_0(N)$. It is a consequence of the Hodge theorem that the associated cohomology classes generate $H^1(Y_0(N),\C)$. As the Hecke operators are given in terms of correspondences on $Y_0(N)$, we get an action of the Hecke algebra (of level $N$) on the (co)homology groups (see \cite[Section 3.2.1]{Nor22} for explicit formulas for the actions).  
A Hecke eigenbasis for $H^1(Y_0(N),\R)$ is given by 
\begin{align}\label{eq:Heckebasis} B_\mathrm{Hecke}(N):=\{\omega_f^{\epsilon}: f\in \mathcal{B}_N,\epsilon\in \{\pm\}\}\cup\{ \omega_{E}(N) \}, \end{align}
where 
\begin{equation}\label{eq:omegapmclass} \omega_f^{\pm}= \left( 2\pi i f(z)dz\pm \overline{2\pi i f(z)dz}\right)/(1+i \pm \overline{1+i})\in H^1(Y_0(N), \R),\end{equation}
with $\mathcal{B}_N:=\{f_1,\ldots, f_g\}\subset \mathcal{S}_2(\Gamma_0(N))$ a basis of holomorphic cuspidal Hecke eigenforms of weight $2$ and level $N$ normalized so that the first Fourier coefficient is $1$, and 
\begin{equation}\label{eq:eisensteinclass} \omega_{E}(N):=E_{2,N}(z) dz\in H^1(Y_0(N), \R),\end{equation}
is the normalized \emph{Eisenstein class} defined from the  Eisenstein series of weight $2$ and level $N$;
\begin{equation}\label{eq:EislevelN} E_{2,N}(z):=\frac{NE^\ast_2(Nz)-E^\ast_2(z)}{N-1}=1+\frac{24}{N-1}e^{2\pi i z}+\ldots,\end{equation} 
where $E^\ast_2$ denotes the modified Eisenstein series as defined in (\ref{eq:E2ast}), which is a non-holomorphic modular form of weight 2. By the Fourier expansion of $E_2^\ast$  we have that $E_{2,N}$ is holomorphic so that indeed $E_{2,N}\in \mathcal{M}_2(\Gamma_0(N))$ and we have for $(n,N)=1$ that 
$$T_n \omega_{E}(N)= \sigma_1(n)\omega_{E}(N)=\left(\Sigma_{d|n}d\right)\omega_{E}(N).$$ 
Since the cap product pairing is Hecke equivariant, one obtains a Hecke eigenbasis for the real homology by taking the dual basis of $B_\mathrm{Hecke}(N)$ with respect to the cap product pairing (\ref{eq:cap}); 
$$\{v_f^\epsilon: f\in\mathcal{B}_N ,\epsilon\in\{\pm\} \}\cup \{v_E(N)\}$$
where
$$\langle v^\pm_f, \omega_f^\pm \rangle_\mathrm{cap}=\langle v_E(N), \omega_E(N) \rangle_\mathrm{cap}=1. $$
One can check by hand that in fact the dual of the Eisenstein element is given by  
\begin{equation}\label{eq:eisensteinclasshom}v_{E}(N)=[\text{class of }\{t+i: 0\leq t<1\}]\in H_1(Y_0(N),\Z), \end{equation}
since the constant term of $E_{2,N}$ is $1$. Notice that geometrically $v_{E}(N)$ is the class of a simple loop around the cusp at $\infty$ and is equal to the image of $\begin{psmallmatrix} 1 & 1 \\ 0 & 1\end{psmallmatrix}$ inside the  abelinization $\Gamma_0(N)^\mathrm{ab}$ modulo torsion under the isomorphism (\ref{eq:homabelian}). 

\subsection{Weyl's Criterion}\label{sec:Weylcriterion}
The abstract setting in Theorem \ref{thm:first} is as follows: we have a sequence of probability measures $\mu_1,\mu_2,\ldots$ on $\mathbf{T}^1(Y_0(N))=\Gamma_0(N)\backslash \PSL_2(\R)$ which we want to prove converge in the weak$^\ast$ topology to the Haar measure $dg$. This means that for a smooth and compactly supported function $\phi:\mathbf{T}^1(Y_0(N))\rightarrow \C$ we want to prove that
\begin{equation} \label{eq:WC}\int_{\mathbf{T}^1(Y_0(N))}\phi(g) d\mu_n(g)\rightarrow \int_{\mathbf{T}^1(Y_0(N))}\phi(g) dg,\quad \text{as }n\rightarrow \infty.\end{equation}
By the spectral theorem (\ref{eq:spectralthm}) and the remarks following it, one obtains Weyl's Criterion.
\begin{lemma}[Weyl Criterion]\label{lem:WC}
Let $N\geq 1$ be a fixed integer and let $\mu_1,\mu_2,\ldots$ be a sequence of probability measures on $\mathbf{T}^1(Y_0(N))$. Let $A\geq 1$ be a constant and $\Phi: \Z_{>0}\rightarrow \R_{\geq 0}$ a function satisfying $\Phi(n)\rightarrow 0$ as $n\rightarrow \infty$. Assume that the following is satisfied 
$$\left|\int_{\mathbf{T}^1(Y_0(N))}F(g) d\mu_n(g)\right|\leq |\mathbf{t}_F|^A \Phi(n)$$
for $F\in \BB_0(\Gamma_0(N))$, and 
$$\left|\int_{\mathbf{T}^1(Y_0(N))}E(g) d\mu_n(g)\right|\leq |\mathbf{t}_E|^A \Phi(n)$$
 for $E\in \BB_e(\Gamma_0(N))$. Here $\mathbf{t}_F$ (resp. $\mathbf{t}_E$) denotes the spectral datum for the Hecke--Maa{\ss} cusp form $F$ (resp. the Eisenstein series $E$) as defined in (\ref{eq:tf}). Then for any smooth and compactly supported function $\phi:\mathbf{T}^1(Y_0(N))\rightarrow \C$ and $\eps>0$, we have
 $$ \int_{\mathbf{T}^1(Y_0(N))}\phi(g) d\mu_n(g)= \int_{\mathbf{T}^1(Y_0(N))}\phi(g) dg+O_{A,\phi,\eps}(\Phi(n)^{1-\eps}),\quad \text{as }n\rightarrow \infty. $$
\end{lemma}
\begin{proof}
By Fourier expansion on the maximal compact we get  
$$\int_{\mathbf{T}^1(Y_0(N))}\phi(g) d\mu_n(g)=\sum_{k\in 2\Z}\int_{\mathbf{T}^1(Y_0(N))}\phi_k(g) d\mu_n(g),$$
in terms of $\SO_2(\R)$ isotypic components; 
$$\phi_k(g):=\frac{1}{2\pi}\int_{0}^{2\pi}\phi\left(g\begin{psmallmatrix}\cos \theta& -\sin \theta\\ \sin \theta&\cos\theta \end{psmallmatrix}\right)e^{-\tfrac{k}{2}\theta i}d\theta.$$
By partial integration we get
$$|\phi_k(g)|\ll_{m,\phi} (|k|+1)^{-m},\quad m\geq 1,$$
which means that we can truncate the sum at $k\ll \Phi(n)^{-\eps}$ at a negligible cost (for $n$ large enough). For the remaining terms we spectrally expand each $\phi_k$ in terms of the weight $k$ subspaces of $\BB_0(Y_0(N))\cup\BB_e(Y_0(N))$ using the spectral theorem (\ref{eq:spectralthm}) and insert the bounds assumed for the Weyl sums. Now for any $m\geq 1$ and $F\in \BB_0(Y_0(N))\cup\BB_e(Y_0(N))$ we have  
$$|\langle F, \phi_k\rangle|= |\lambda_F|^{-m}|\langle F, \Delta^m \phi_k\rangle|=|\lambda_F|^{-m}|\langle F, (\Delta^m \phi)_k\rangle|\ll_{\phi,m} (|t_F|+1)^{-2m}, $$
by self-adjointness and invariance of the Laplace-Beltrami operator. Thus we may truncate the sum at $|t_F|\ll \Phi(n)^{-\eps}$ (for $n$ large enough) at a negligible cost by the Weyl law \cite[Corollary 11.2]{Iw} (actually any polynomial bound on the measure of the spectrum with bounded spectral data suffices). For the remaining terms we get the wanted by the assumptions on the Weyl sums.
\end{proof}
\begin{remark}Note that complex conjugation defines an isomorphism 
$$\AA(\Gamma_0(N),k)\xrightarrow{\sim} \AA(\Gamma_0(N),-k),$$
which implies that we may assume that the weight satisfies $k\geq 0$ when applying Weyl's criterion.\end{remark}

We will now derive a Weyl type criterion in the setting of Theorem \ref{thm:second}. Given a finite dimensional real vector space $V$, define the \emph{(topological) sphere associated to $V$} as
$$\mathbf{S}(V):= (V-\{0\})/\R_{>0},$$
equipped with the quotient topology, and denote by $v\mapsto \overline{v}$ the canonical map $V-\{0\}\rightarrow \mathbf{S}(V)$. Given a sequence of elements $v_0,v_1,v_2,\ldots \in V-\{0\}$ we want to give a criterion which implies that 
\begin{equation}\label{eq:wc}\overline{v_n}\rightarrow \overline{v_0},\quad \text{as }n\rightarrow \infty,\end{equation}
in $\mathbf{S}(V)$. Given a basis $B^\ast\subset V^\ast$ of the dual vector space $V^\ast$ of $V$, we define the following norm on $V$:
$$ |\!|v|\!|_{B^\ast}:=\max_{\omega\in B^\ast}|\langle v, \omega\rangle_V|,$$
where $\langle \cdot,\cdot \rangle_V:V\times V^\ast\rightarrow \R$ denotes the canonical pairing. Then the corresponding Weyl criterion is given by the following.
\begin{lemma}\label{lem:WCvector} Let $v_0,v_1,\ldots$ be a sequence of non-zero elements in a finite dimensional real vector space $V$. Let $B^\ast $ be a basis for the dual $V^\ast$ and $\Phi:\Z_{>0}\rightarrow \R_{\geq 0}$ a function such that
\begin{equation}\label{eq:vnnorm}\left|\langle v_n,\omega\rangle_V- \langle v_0,\omega\rangle_V\right|\leq \Phi(n)|\!|v_0|\!|_{B^\ast} ,\end{equation}
for all $\omega\in B^\ast$. Then we have 
\begin{align}\label{eq:WeylV}\left|\!\left|\frac{v_n}{|\!|v_n|\!|_{B^\ast}}-\frac{v_0}{|\!|v_0|\!|_{B^\ast}}\right|\!\right|_{B^\ast}\ll \Phi(n) .\end{align}
If furthermore $\Phi(n)\rightarrow 0$ as $n\rightarrow \infty$, then 
\begin{equation}\label{eq:Weylfinalconclusion}\overline{v_n}\rightarrow \overline{v_0} \text{ in }\mathbf{S}(V),\text{ as }n\rightarrow \infty.\end{equation}   
 \end{lemma}
 \begin{proof}
The condition (\ref{eq:vnnorm}) implies that  
$$|\!|v_n|\!|_{B^\ast}=|\!|v_0|\!|_{B^\ast}(1+O(\Phi(n))).$$
Thus by the triangle inequality
$$\left|\!\left|\frac{v_n}{|\!|v_n|\!|_{B^\ast}}-\frac{v_0}{|\!|v_0|\!|_{B^\ast}}\right|\!\right|_{B^\ast} \leq \left|\!\left|\frac{v_n-v_0}{|\!|v_0|\!|_{B^\ast}}\right|\!\right|_{B^\ast}+|\!|v_n|\!|_{B^\ast}\frac{||\!|v_n|\!|_{B^\ast}-|\!|v_0|\!|_{B^\ast}|}{|\!|v_n|\!|_{B^\ast}|\!|v_0|\!|_{B^\ast}}\ll \Phi(n). $$
We have a continuous map
$$V-\{0\}\rightarrow \{v\in V: |\!|v|\!|_{B^\ast}=1\}\subset V,\quad v\mapsto \tfrac{v}{|\!|v|\!|_{B^\ast}},$$
which factors through $\mathbf{S}(V)$ and the resulting map  
$$ \mathbf{S}(V)\xrightarrow{\sim} \{v\in V: |\!|v|\!|_{B^\ast}=1\},   $$
is a homeomorphism. Now the topology on $V$ is induced by $|\!|\cdot |\!|_{B^\ast}$ which implies that indeed the convergence $\overline{v_n}\rightarrow \overline{v_0}$ is equivalent to the left-hand side of (\ref{eq:WeylV}) converging to $0$.  
 \end{proof}
In our case we will have $V=H_1(Y_0(N),\R)$ and identify $V^\ast$ with the real cohomology $H^1(Y_0(N),\R)$ by the cap product pairing and let $B^\ast=B_\mathrm{Hecke}(N)$ be the Hecke basis of cohomology. 
\begin{remark}
Note here that one could apply any norm on $\R^{\dim V}$ (in place of the sup norm) to the vector $(\langle v, \omega\rangle_V)_{\omega\in B^\ast}$ and a similar conclusion would still hold, since all norms on finite dimensional vector spaces are equivalent. 
\end{remark}

\section{Additive twist $L$-series of automorphic forms}\label{sec:addtwists}
Let $f\in \AA(\Gamma_0(N), k)$ be a Maa{\ss} form of even weight $k$ with  Fourier expansion at $\infty$
\begin{equation}\label{eq:f(z)Fourierexp}f(z)= f_\infty(y)+\sum_{n\neq 0} a_f(n) \tilde{W}_{\frac{k}{2},it_f}(ny)e(nx),\end{equation}
and assume that for some $\delta>0$ the Fourier coefficients satisfy the bound 
\begin{equation}
\label{eq:Fourierbound} \sum_{0<|n|<X} |a_f(n)|^2\ll_f X^{2\delta},\quad \text{as }X\rightarrow \infty. 
\end{equation}
In particular, this is satisfied if $f$ is $L^2$-integrable \cite[Thm.\ 3.2]{Iw} (with any $\delta>0$) or if $f$ is a linear combination of Eisenstein series with spectral parameter $s_f\in \C$, $\Re s_f\geq 1/2$ \cite[eq.\ (6.19)]{Iw} (with $\delta=\Re s_f+1/2$). Denote by $F:\mathbf{T}^1(Y_0(N))\rightarrow \C$ the lift of $f$ to the unit tangent bundle. Given $x\in \Q$ we define the \emph{additive twist $L$-series of $f$} (and of $F$)  as  
\begin{equation}L(s,F,x)=L(s,f,x):= \sum_{n\geq 1} \frac{a_f(n)e(nx)}{n^{s-1/2}},\quad  \Re s>1+\delta,   \end{equation}
which admit meromorphic continuation to $\Re s>\Re it_f$ and for $L^2$-integrable Maa{\ss} forms and for Eisenstein series with spectral parameter $s=1/2+it$ is regular at $s=1/2$, see \cite[Lemma 5.4]{DrNo22}. Recall from e.g.\ \cite[eq.\ (12)]{BlomerHarcos08.2} that for $f$ a Hecke--Maa{\ss} newform we have 
\begin{equation}\label{eq:FChecke}a_f(n)=\frac{\lambda_f(n)}{n^{1/2}}\end{equation}
and so
\begin{equation}\label{eq:defadditive}L(s,f,x)= \sum_{n\geq 1} \frac{\lambda_f(n)e(nx)}{n^{s}}, \quad  \Re s>1+\delta,\end{equation}
where $\lambda_f(n)$ denotes the $n$th Hecke eigenvalue of $f$. Notice also that for $\ell N'|N$, $f\in \AA(\Gamma_0(N'), k)$ and $\nu_{\ell,N'}$ defined as in (\ref{eq:nulevel}) we have the following formula for the Fourier coefficients:
\begin{equation}\label{eq:oldformFC}a_{\nu_{\ell,N'}f}(n)=\delta_{\ell|n} a_f(\tfrac{n}{\ell}),\end{equation}
where $\delta_{\ell|n}=1$ if $\ell$ divides $n$ and $\delta_{\ell|n}=0$ otherwise, which implies that 
\begin{equation}\label{eq:oldformAT} L(s,\nu_{\ell,N'}f,x)=\ell^{1/2-s} L(s,f,\ell x). \end{equation}
This will be useful when dealing with the additive twist $L$-series of the bases defined in (\ref{eq:BB0}) and (\ref{eq:BBE}). We denote by
\begin{equation}\label{eq:MellinWhit} \gamma_{\alpha,\beta}(s):= \int_0^\infty \tilde{W}_{\alpha,\beta}(y) y^{s-1/2}\frac{dy}{y}, \end{equation}
the Mellin transform of Whittaker functions which converge absolutely for $\Re s>\mathfrak{d}_\beta$ by (\ref{eq:whitbound1}) and (\ref{eq:whitbound2}). Recall the definition of the constant term $f_\infty\left(y\right)$ at $\infty$ as in (\ref{eq:constantterm}).  

\begin{prop}\label{prop:linetoadd}
Let $\gamma\in \Gamma_0(N)$ be a hyperbolic element represented by a matrix in $\SL_2(\Z)$ with lower left entry equal to $q\geq 2$. Let $f\in \AA(\Gamma_0(N),k)$ be a Maa{\ss} form of even weight $k$, level $N$ and spectral parameter 
\begin{equation}\label{eq:spectralparameterassump}it_f\in i\R\cup(0,1/2)\cup\{\ell-\tfrac{1}{2}:\ell\in \Z_{>0}\},\end{equation}
and put
$$\gamma^\pm_s(\mathbf{t}_f)=(\kappa_f)^{(1\mp 1)/2}q^{s-1/2}\gamma_{\pm\frac{k}{2},it_f }(s),$$
 with $\mathbf{t}_f$ the spectral datum of $f$, $\kappa_f$ as defined in (\ref{eq:fouriercoeff-sign})\footnote{Here it is understood that $(\kappa_f)^{(1-1)/2}=1$ in the case $\kappa_f=0$.} and $\gamma_{\pm\frac{k}{2},it_f }(s)$ as in (\ref{eq:MellinWhit}). Assume that $f$ satisfies (\ref{eq:Fourierbound}). Then  for $\Re s>\Re it_f$ we have the following equality of meromorphic functions:
   \begin{align} 
\nonumber &\int_{1}^\infty \left(f\left(\gamma \infty+i\tfrac{y}{q}\right)-f_\infty\left(\tfrac{y}{q}\right)\right)y^{s-1/2}\frac{dy}{y}+(-1)^{k/2}\int_{1}^\infty \left(f\left(\gamma^{-1} \infty+i\tfrac{y}{q}\right)-f_\infty\left(\tfrac{y}{q}\right)\right)y^{1/2-s}\frac{dy}{y}\\
\nonumber &=\gamma^+_s(\mathbf{t}_f) L(s, f,\gamma \infty)+ \gamma^-_s(\mathbf{t}_f)L(s, f,-\gamma \infty)\\
\label{eq:Eichlerint} &\qquad \qquad +A_f q^{-1/2-it_f}\left(\frac{1}{s+it_f}+\frac{(-1)^{k/2}}{1-s+it_f}\right)+ B_f q^{-1/2+it_f}\left(\frac{1}{s-it_f}+\frac{(-1)^{k/2}}{1-s-it_f}\right) ,
 \end{align} 
where in the case $t_f=0$ the last term is replaced by 
$$B_f q^{-1/2}\left(-\frac{1}{s^{2}}-\frac{\log q}{s}
+(-1)^{k/2}\!\left(-\frac{1}{(s-1)^{2}}+\frac{\log q}{s-1}\right)
\right).$$
Furthermore, the left-hand side of (\ref{eq:Eichlerint}) defines an analytic function in $s\in \C$.
\end{prop}

\begin{proof}
Note first of all that by inserting the bound (\ref{eq:whitbound1}) into the Fourier expansion (\ref{fexp}) and using the bound for the Fourier coefficients (\ref{eq:Fourierbound}) the left-hand side of (\ref{eq:Eichlerint}) defines an entire function in $s\in \C$. By weight $k$ automorphy of $f$, we have
\begin{align} 
\nonumber f\left(\gamma^{-1}\infty+i\tfrac{y}{q}\right)= f\left(\gamma^{- 1}\left(\gamma\infty+i\tfrac{y^{-1}}{q}\right)\right)&= j_{\gamma^{-1}}(\gamma\infty+i\tfrac{y^{-1}}{q})^k f\left(\gamma\infty+i\tfrac{y^{-1}}{q}\right)\\
\label{eq:modularity}&= (-1)^{k/2} f\left(\gamma\infty+i\tfrac{y^{-1}}{q}\right).\end{align}
Thus we conclude by the change of variable $y\mapsto 1/y$ that 
\begin{align}
&\int_{1}^\infty \left(f\left(\gamma^{-1} \infty+i\tfrac{y}{q}\right)-f_\infty\left(\tfrac{y}{q}\right)\right)y^{1/2-s}\frac{dy}{y}\\
&=(-1)^{k/2}\int_{1}^\infty \left(f\left(\gamma \infty+i\tfrac{y^{-1}}{q}\right)-(-1)^{k/2}f_\infty\left(\tfrac{y}{q}\right)\right)y^{1/2-s}\frac{dy}{y}\\
&= (-1)^{k/2} \int_{0}^1 \left(f\left(\gamma \infty+i\tfrac{y}{q}\right)-(-1)^{k/2}f_\infty\left(\tfrac{y^{-1}}{q}\right)\right)y^{s-1/2}\frac{dy}{y}.
\end{align}
In view of the bounds (\ref{eq:whitbound1}) and (\ref{eq:whitbound2}), (\ref{eq:Fourierbound}) and the Fourier expansion (\ref{fexp}), we conclude that for $\Re s\gg_f 1$ the  left-hand side of (\ref{eq:Eichlerint}) can be expressed as the sum of the following two absolutely convergent integrals:
\begin{align}
&\int_0^\infty \left(f\left(\gamma \infty+i\tfrac{y}{q}\right)-f_\infty\left(\tfrac{y}{q}\right)\right)y^{s-1/2}\frac{dy}{y}+\int_0^1 \left(f_\infty\left(\tfrac{y}{q}\right)-(-1)^{k/2}f_\infty\left(\tfrac{y^{-1}}{q}\right)\right)y^{s-1/2}\frac{dy}{y}\\
&= q^{s-1/2}\int_0^\infty \left(f\left(\gamma \infty+iy\right)-f_\infty\left(y\right)\right)y^{s-1/2}\frac{dy}{y}\\
&+A_f q^{-1/2-it_f}\left(\frac{1}{s+it_f}+\frac{(-1)^{k/2}}{1-s+it_f}\right)+ B_f q^{-1/2+it_f}\left(\frac{1}{s-it_f}+\frac{(-1)^{k/2}}{1-s-it_f}\right),
\end{align}
by inserting (\ref{eq:constantterm}), with the last term replaced by 
$$B_f q^{-1/2}\left(-\frac{1}{s^{2}}-\frac{\log q}{s}
-(-1)^{k/2}\!\left(\frac{1}{(s-1)^{2}}-\frac{\log q}{s-1}\right)\right),$$
when $t_f=0$. Now by inserting the Fourier expansion and interchanging sum and integral (cf. proof of \cite[Lemma 5.3]{DrNo22})  we get for $\Re s\gg_f 1$
\begin{align}\nonumber \int_0^\infty \left(f\left(\gamma \infty+iy\right)-f_\infty\left(y\right)\right)y^{s-1/2}\frac{dy}{y}&=\sum_{n\neq 0} a_f(n)e(n\gamma \infty)\int_0^\infty \tilde{W}_{\frac{k}{2},it_f}(ny) y^{s-1/2}\frac{dy}{y} \\
\label{eq:useful}&= \gamma_{\frac{k}{2},it_f}(s) L(s,f,\gamma \infty) + \kappa_f \gamma_{- \frac{k}{2},it_f}(s) L(s,f,-\gamma \infty),\end{align}
using that $a_f(-n)=\kappa_f a_f(n)$ by (\ref{eq:fouriercoeff-sign}) and $\tilde{W}_{\alpha,\beta}(-y)=\tilde{W}_{-\alpha,\beta}(y)$. Now the result follows for $\Re s>\Re it_f$ by meromorphic continuation\footnote{Note that in fact the formula (\ref{eq:Eichlerint}) \emph{gives} the meromorphic continuation of the right-hand side for all $s\in \C$.}.\end{proof}
\begin{remark}
Note that the integral representation for the additive twist $L$-series in Proposition \ref{prop:linetoadd} is proved in the context of general Maa{\ss} forms for $\Gamma_0(N)$ satisfying (\ref{eq:spectralparameterassump}), i.e.\ not restricted to elements occurring in the spectral decomposition of $\LL^2(\Gamma_0(N),k)$. In particular we will later apply it to the modified Eisenstein series (\ref{eq:E2ast}) (a weight $2$ Eisenstein series evaluated at $s=0$).
\end{remark}
Finally we will record a useful upper bound for the archimedean factors $\gamma_{\pm \frac{k}{2},it_f}(s)$ at the central point $s=1/2$. 

\begin{lemma}\label{lem:boundwhittaker}
Let $0<\delta<1/2$ and 
$$\alpha\in \Z,\quad \beta\in i\R \cup (0,1/2-\delta)\cup \{\ell-\tfrac{1}{2}:\ell\in \Z_{>0} \}.$$
Then we have the uniform bound 
$$\gamma_{\alpha,\beta}(1/2)\ll_\delta (|\alpha|+|\beta|+1)^{3/2}.$$
\end{lemma}
\begin{proof}
Put $M:=|\alpha|+|\beta|+1$. Then by (\ref{eq:whitbound1}) and (\ref{eq:whitbound2}) we have for $0<\eps<1/4$ 
$$ \gamma_{\alpha,\beta}(1/2)\ll_\eps M^{1+\mathfrak{d}_\beta} \int_0^M y^{1/2-\eps-\mathfrak{d}_\beta} \frac{dy}{y}+\int_M^\infty y^{1/2}\exp(-\frac{y}{M}) \frac{dy}{y}, $$
in all cases. If $\beta\notin (0,1/2)$ then we get the wanted by $\eps=1/8$, say. If $0<\beta<1/2-\delta$ then we pick $\eps=\delta/2$.  
\end{proof}

\section{From geodesic periods to additive twist $L$-series}
In this section we will relate the integrals over closed geodesics to additive twist $L$-series. To motivate matters let $f\in \mathcal{S}_2(\Gamma_0(N))$ be a holomorphic cusp form of weight $2$ and level $N$ (normalized so that that the automorphy factor is $(cz+d)^2$) and denote by 
$$F:\Gamma_0(N)\backslash \PSL_2(\R)\rightarrow  \C,$$ 
the lift of $f$ to the unit tangent bundle (or more precisely the lift of $z\mapsto (\Im z) f(z)$). Let $\CC\subset \Gamma_0(N)\backslash \PSL_2(\R)$ be an oriented closed geodesic corresponding to the conjugacy class of $\gamma\in \Gamma_0(N)$. Then it follows from the standard parametrization (\ref{eq:param}) that   
$$\int_\CC F(t)d\mu_\CC(t)=-i \int_\CC f(z)dz=-i \int_{w_\gamma}^{\gamma w_\gamma} f(z)dz,$$
where $w_\gamma$ is any point on the fixed half-circle $S_\gamma\subset \Hb$ of $\gamma$ and the measure $\mu_\CC$ is defined as in Section \ref{sec:closedgeo}.  Now by shifting contours (using the holomorphicity of $f$) and a change of variable we have that 
\begin{equation}\label{eq:contourshift}\int_{w_\gamma}^{\gamma w_\gamma} f(z)dz= -\left(\int_{\gamma w_\gamma}^{\gamma w} f(z)dz+\int_{\gamma w}^{w} f(z)dz+\int_{w}^{w_\gamma} f(z)dz\right)=\int_{w}^{\gamma w} f(z)dz,\end{equation}
for any $w\in \Hb\cup \mathbf{P}^1(\Q)$. Combining with (the proof of) Proposition \ref{prop:linetoadd}, we arrive at the following. 

\begin{prop}\label{prop:holoaddtwist}
Let $F:\Gamma_0(N)\backslash \PSL_2(\R)\rightarrow \C$ denote the lift to the unit tangent bundle of a holomorphic cusp form $f\in \mathcal{S}_2(\Gamma_0(N))$ of weight $2$. Let $\gamma\in \Gamma_0(N)$ be a hyperbolic matrix with $\gamma\infty=a/q$ with $q>0$ and $(a,q)=1$, and let $\CC\subset \Gamma_0(N)\backslash \PSL_2(\R)$ be the oriented closed geodesic corresponding to the $\Gamma_0(N)$-conjugacy class of $\gamma$. Then it holds that 
$$  \int_{\CC} F(t) d\mu_{\CC}(t)= -i\int_{\CC} f(z) dz= 2L(1/2, f, \tfrac{a}{q}). $$
\end{prop} 
\begin{proof}
The proof follows by putting $w=\infty$ in (\ref{eq:contourshift}), using equation (\ref{eq:useful}) in the proof of Proposition \ref{prop:linetoadd} and recalling that $\kappa_f=0$ and
\begin{equation}\label{eq:explicitWtilde}\gamma_{1,1/2}(1/2)=\int_{0}^\infty \tilde{W}_{1,1/2}(y)\frac{dy}{y}=\frac{1}{\sqrt{\Gamma(2)\Gamma(1)}}\int_{0}^\infty 4\pi ye^{-2\pi y} \frac{dy}{y}=2, \end{equation} 
in view of (\ref{eq:specialWhit}). Note here that $f(z)$ is normalized so that $z\mapsto (\Im z) f(z)$ is a Maa{\ss} form of weight $2$.   \end{proof}
We would like to prove something like this for the entire automorphic spectrum of the unit tangent bundle $\Gamma_0(N)\backslash \PSL_2(\R)$. However, in the absence of holomorphicity, we will have to settle for an approximate equality. Our main result is the following, recalling the definitions of our orthonormal bases $\mathcal{B}_0(\Gamma_0(N))$ and $\mathcal{B}_e(\Gamma_0(N))$ defined in (\ref{eq:BB0}) and (\ref{eq:BBE}).
\begin{prop}\label{prop:main}
Let $F\in \mathcal{B}_0(\Gamma_0(N))$ be an $L^2$-normalized cuspidal Hecke--Maa{\ss} form of level $N$ and even weight $k$. Let $\gamma\in \Gamma_0(N)$ be a hyperbolic matrix with $\gamma\infty=a/q$ where  $q\geq 2$ and $(a,q)=1$, and let $\CC\subset \Gamma_0(N)\backslash \PSL_2(\R)$ be  the oriented closed geodesic corresponding to the $\Gamma_0(N)$-conjugacy class of $\gamma$. Then we have for $0<\eps<1/4$ that
\begin{align}\label{eq:cuspmain}  \int_{\CC} F(t) d\mu_{\CC}(t)= c_F^+L(1/2, F, \tfrac{a}{q})+c_F^- L(1/2, F, \tfrac{-a}{q})+O_{N,\eps}\left(|\mathbf{t}_F|^{3+\eps}\left(\frac{q}{|\tr (\gamma)|}\right)^{1/2+\eps}\right),\end{align}
where $|c_F^\pm|\ll_{N}|\mathbf{t}_F|^{3/2}$.
Similarly for $E\in \mathcal{B}_e(\Gamma_0(N))$ an Eisenstein series, we have 
\begin{align}\label{eq:eisensteinmain}  \int_{\CC} E(t) d\mu_\CC(t)=&c_E^+ L(1/2,E, \tfrac{a}{q})+c_E^-L(1/2,E, \tfrac{-a}{q})\\ \nonumber
&+O_{N,\eps}\left(|\mathbf{t}_E|^{3+\eps}\left(\left(\frac{q}{|\tr (\gamma)|}\right)^{1/2+\eps}+\left(\frac{q}{|\tr (\gamma)|}\right)^{-1/2-\eps}\right)\right), \end{align}
where $|c_E^\pm|\ll_{N}|\mathbf{t}_E|^{3/2}$.
\end{prop} 

The rest of the section is devoted to the proof of this result. Our approach is to use the Fourier expansion to compare the geodesic integral to the vertical line integrals which we have seen are connected to additive twist  $L$-series. Recall from Section \ref{sec:closedgeo} the definitions of $\epsilon_\gamma, g_\gamma, r_\gamma, y_{\gamma,t}$ (\ref{eq:epsgamma})-(\ref{eq:ygammat}) as well as the standard parametrization (\ref{eq:param}) of closed geodesics. The following estimates can be interpreted as saying that the closed geodesic is close to the two vertical segments as illustrated in Figure \ref{fig:S} and that the automorphy factors do not oscillate too much.
\begin{lemma}\label{lem:keybounds}Let $\gamma\in \PSL_2(\Z)$ be hyperbolic and let $\ell_\gamma$ be the hyperbolic length defined in (\ref{eq:geodesiclength}). For $k\in 2\Z_{\geq 0}$, $0<|t|<\ell_\gamma$ and $\epsilon_t=t/|t|\in\{\pm 1\}$, we have  
\begin{align}\label{eq:keyest1} j_{g_\gamma a_t}(i)^{-k}-(-\epsilon_t)^{k/2}&\ll k\frac{y_{\gamma,t}}{r_\gamma},\\
\label{eq:keyest2}  \frac{-e^{t}+e^{-t}}{e^t+e^{-t}} +\epsilon_t&\ll  \frac{(y_{\gamma,t})^2}{(r_\gamma)^2},\\
\label{eq:keyest3}  \Re (g_\gamma a_t i)-\gamma^{\epsilon_t} \infty &\ll \frac{(y_{\gamma,t})^2}{r_\gamma}. \end{align}
\end{lemma}
\begin{proof} We start by noticing that for $-\ell_\gamma\leq t\leq \ell_\gamma$ we have  
$$e^{-|t|}\ll \frac{1}{e^t+e^{-t}}\ll \frac{y_{\gamma,t}}{r_\gamma}.$$
By direct calculation one shows that 
$$j_{g_\gamma a_t}(i)= \frac{ie^t+\epsilon_\gamma}{\sqrt{e^{2t}+1}},$$
and thus by using the equality $x^n-y^n=(x-y)(x^{n-1}+\ldots+y^{n-1})$, we arrive at
\begin{align}j_{g_\gamma a_t}(i)^{-k}-(-\epsilon_t)^{k/2}&\ll k (j_{g_\gamma a_t}(i)^{-2}+\epsilon_t)\ll k \frac{|e^{-t}-e^t+2i\epsilon_\gamma+\epsilon_t(e^t+e^{-t})|}{e^t+e^{-t}}\ll k e^{-|t|} \ll k\frac{y_{\gamma,t}}{r_\gamma},\end{align}
which proves (\ref{eq:keyest1}). By a similar consideration, we obtain (\ref{eq:keyest2}); 
\begin{align} \frac{-e^{t}+e^{-t}}{e^t+e^{-t}} +\epsilon_t&\ll \frac{|-e^{t}+e^{-t}+\epsilon_t(e^t+e^{-t})|}{e^t+e^{-t}}\ll e^{-2|t|}\ll  \frac{(y_{\gamma,t})^2}{(r_\gamma)^2}.\end{align}
Finally, we have by direct computation that  
\begin{align}  \Re (g_\gamma a_t i)-\gamma^{\epsilon_t} \infty =\frac{a-d}{2c}+\epsilon_\gamma r_\gamma \frac{e^{2t}-1}{e^{2t}+1}-\gamma^{\epsilon_t} \infty= \frac{a-d}{2c}+\epsilon_\gamma r_\gamma \frac{e^{t}-e^{-t}}{e^{t}+e^{-t}}-\frac{a-d+\epsilon_t(a+d)}{2c},\end{align}
which by (\ref{eq:keyest2}) is equal to
\begin{align} \epsilon_\gamma \epsilon_t r_\gamma-\frac{\epsilon_t(a+d)}{2c}+ O\left(\frac{(y_{\gamma,t})^2}{r_\gamma}\right) &\ll 
\frac{|a+d|-\sqrt{(a+d)^2-4}}{2|c|}+\frac{(y_{\gamma,t})^2}{r_\gamma}\\ &\ll \frac{1}{|(a+d)c|}+\frac{(y_{\gamma, t})^2}{r_\gamma }.\end{align}
Noticing that $|a+d|^{-1}\ll (r_\gamma)^{-1}|c|^{-1}$ and $|c|^{-1}\ll y_{\gamma,t}$ we obtain 
$$\frac{1}{|(a+d)c|} \ll \frac{c^{-2}}{r_\gamma}\ll \frac{(y_{\gamma, t})^2}{r_\gamma } , $$
and so the final estimate (\ref{eq:keyest3}) follows.
\end{proof}

Finally we will need the following weak but uniform estimates. Similar estimates were obtained in \cite{DrNo22} but we will need uniformity in the spectral data similar to \cite[Section 2.3]{BlomerHarcos08.2}. 
   
\begin{lemma}\label{lem:Fourierbnd}
Let $f\in\AA(\Gamma_0(N),k)$ be a Hecke--Maa{\ss} form with constant term at infinity $ f_\infty(y)$ and such that the lift to the unit tangent bundle lies in $\BB_0(\Gamma_0(N))\cup \BB_e(\Gamma_0(N))$ as defined in (\ref{eq:BB0}) and (\ref{eq:BBE}). Then for $0<\eps<1/4$ we have  
 \begin{equation}\label{eq:Sary1} |f(x+iy)-f_\infty(y)|\ll_{N,\eps} |\mathbf{t}_f|^{2+\eps}y^{-1/2-\eps},  \end{equation}
\begin{equation}\label{eq:Sary2} |f(x+iy)-f(x'+iy)|\ll_{N,\eps} |\mathbf{t}_f|^{3+\eps}y^{-3/2-\eps} |x-x'|. \end{equation}

\end{lemma}
\begin{proof}
Inserting the bounds (\ref{eq:whitbound1}) and (\ref{eq:whitbound2}) into the Fourier expansion (\ref{fexp}) we arrive at
\begin{align}\nonumber
&|f(x+iy)-f_\infty(y)|=|\sum_{n\neq 0} a_f(n) \tilde{W}_{k/2,it_f}(ny)e(nx)|\\
\nonumber &\ll_\eps |\mathbf{t}_f|^{1+\mathfrak{d}_f}\sum_{0<|n|\leq |\mathbf{t}_f|/y} |a_f(n)| (|n|y)^{1/2-\mathfrak{d}_f -\eps} +\sum_{|n|> |\mathbf{t}_f|/y} |a_f(n)| (|n|y)^{1/2}\exp\left(-\frac{|n|y}{|\mathbf{t}_f|}\right).\end{align}
By the expressions (\ref{eq:coefONB}), (\ref{eq:Eisensteinlift}) and the formulas (\ref{eq:oldformFC}), (\ref{eq:FChecke}), we see that 
$$|a_f(n)|\ll_{N,\eps} |\mathbf{t}_f|^\eps |n|^{-1/2}\sum_{d|N}|\lambda_{f^\ast}(\tfrac{n}{(n,d)})| ,$$ 
where $\lambda_{f^\ast}(n)$ denotes the $n$th Hecke eigenvalue of the Hecke--Maa{\ss} newform $f^\ast$ (cuspidal or Eisenstein) underlying $f$. Thus we conclude from \cite[Proposition 19.6]{DuFrIw02} the following uniform Rankin--Selberg bound;
$$\sum_{0<|n|\leq X} |a_f(n)|^2\ll_{N,\eps} |\mathbf{t}_f|^\eps X^{\eps},\quad \text{as }X\rightarrow \infty.$$
The estimate (\ref{eq:Sary1}) now follows by an application of Cauchy--Schwarz. The second claim follows by the same estimates after writing 
\begin{align}
f(x+iy)-f(x'+iy)=\sum_{n\neq 0} a_f(n) \tilde{W}_{k/2,it_f}(ny)(e(nx)-e(nx')),
\end{align} 
and using the inequality $|e(nx)-e(nx')|\ll |n| |x-x'|$. 
\end{proof}
With this at hand we are ready to prove our main result of this section.
\begin{proof}[Proof of Proposition \ref{prop:main}]
Let $F:\Gamma_0(N)\backslash \PSL_2(\R)\rightarrow \C$ be the lift of $f$, either a Hecke cusp form or an Eisenstein series as in the statement of the proposition. To lighten notation we will simply write $r=r_\gamma$, $h=h_\gamma$, $y_t=y_{\gamma,t}$ and $\epsilon=\epsilon_t$. Using the standard parametrization (\ref{eq:ggamma}) of $\CC$ we arrive at    
\begin{align}
\int_{\CC} F(g) d\mu_{\CC}(g)= \int_{-\ell}^\ell F(g_\gamma a_t)dt= \int_{-\ell}^\ell j_{g_\gamma a_t}(i)^{-k}f(g_\gamma a_t i)dt 
\end{align}
where $2\ell=2\ell_\gamma=|\CC|$ is the geodesic length of $\CC$ and $g_\gamma$ as in (\ref{eq:ggamma}). 
In view of Proposition \ref{prop:linetoadd}, we want to compare this geodesic integral to the line integral of $f$. By a change of variable we have for $y_t$ as above the following equality of measures:
\begin{equation}\label{eq:COV} \frac{dy_t}{y_t}= \frac{-e^{t}+e^{-t}}{e^t+e^{-t}} dt,\end{equation}
and in view of Figure \ref{fig:S} we see that $y_t$ maps both intervals $(0,\ell)$ and $(-\ell,0)$ to the interval $(h,r)$. 
The idea is now to split up the geodesic integral according to whether $t\in (-\ell,\ell)$ is positive or negative. We will show that each of these two contributions are equal to one of the left-hand side of (\ref{eq:Eichlerint}) in Proposition \ref{prop:linetoadd} up to a small error. More precisely, consider  
\begin{align}
\int_{0}^{\epsilon \ell} \left(j_{g_\gamma a_t}(i)^{-k}f(g_\gamma a_t i)-(-\epsilon)^{k/2+1}f(\gamma^{\epsilon}\infty+iy_t) \frac{-e^{t}+e^{-t}}{e^t+e^{-t}}\right) dt, 
\end{align}
for $\epsilon\in \{\pm 1\}$ (recall that $y_t=\Im(g_\gamma a_t i)$). We express the above as a sum of the following three terms:
\begin{align}\label{eq:term1}
(-\epsilon)^{k/2+1}\int_{0}^{\epsilon \ell} \left(f(g_\gamma a_t i)-f(\gamma^{\epsilon}\infty+i\Im (g_\gamma a_t i))\right)\frac{-e^{t}+e^{-t}}{e^t+e^{-t}} dt, 
 \end{align}
 \begin{align}\label{eq:term2}
 \int_{0}^{\epsilon\ell} (f(g_\gamma a_t i)-f_\infty(y_t))\left(j_{g_\gamma a_t}(i)^{-k}-(-\epsilon)^{k/2+1} \frac{-e^{t}+e^{-t}}{e^t+e^{-t}}\right) dt,
  \end{align}
   \begin{align}\label{eq:term3}
 \int_{0}^{\epsilon\ell} f_\infty(y_t)\left(j_{g_\gamma a_t}(i)^{-k}-(-\epsilon)^{k/2+1} \frac{-e^{t}+e^{-t}}{e^t+e^{-t}}\right) dt. 
  \end{align}
Note that the last term vanishes if $f$ is cuspidal. Applying the bounds (\ref{eq:Sary2}) and (\ref{eq:keyest3}) and doing the change of variable to $y_t$ (corresponding to (\ref{eq:COV})) we see that (\ref{eq:term1}) is bounded by
\begin{align}
|\mathbf{t}_f|^{3+\eps}\int_{h}^r y^{-3/2-\eps} \frac{y^2}{r }\frac{dy}{y}\ll_\eps |\mathbf{t}_f|^{3+\eps}\, r^{-1/2-\eps}. 
 \end{align}
To bound (\ref{eq:term2}) we start by observing that for $0\leq |t| \leq 1$ we have $y_t\asymp r$. Thus by the bound (\ref{eq:Sary1}) (and bounding the remaining terms of norm $1$ trivially) we get
  \begin{align}
 \nonumber \int_{0}^{\epsilon} (f(g_\gamma a_t i)-f_\infty(y_t))\left(j_{g_\gamma a_t}(i)^{-k}-(-\epsilon)^{k/2+1} \frac{-e^{t}+e^{-t}}{e^t+e^{-t}}\right) dt\ll_\eps |\mathbf{t}_f|^{2+\eps} r^{-1/2-\eps}. 
  \end{align}
To bound the remaining contribution from $(\epsilon,\epsilon \ell)$, we start by observing that by (\ref{eq:keyest1}) and another application of (\ref{eq:keyest2}) it holds that  
\begin{equation}\label{eq:smalloscc}j_{g_\gamma a_t}(i)^{-k}-(-\epsilon)^{k/2+1} \frac{-e^{t}+e^{-t}}{e^t+e^{-t}}=(-\epsilon)^{k/2}-(-\epsilon)^{k/2+1} \frac{-e^{t}+e^{-t}}{e^t+e^{-t}}+O\left(\frac{y_t}{r}\right)\ll \frac{y_t}{r},   \end{equation}
noting that $y_t\leq r$. Furthermore, for $|t|>1$, we have $|-e^t+e^{-t}|\gg 1$ which by (\ref{eq:keyest2}) gives that 
 $$\frac{e^t+e^{-t}}{-e^t+e^{-t}}=-\epsilon_t+O\left(y_t^2/r^2\right),$$
where again $\epsilon_t=t/|t|$. This implies by (\ref{eq:COV}) that 
$$dt= \frac{e^t+e^{-t}}{-e^t+e^{-t}}\frac{dy_t}{y_t}\ll \frac{dy_t}{y_t},\quad |t|>1$$ 
interpreted in the infinitesimal sense. Combining (\ref{eq:smalloscc}) with the bound (\ref{eq:Sary1}) we see that the remaining contribution in (\ref{eq:term2}), for $t$ between $\epsilon $ and $\epsilon \ell$, can be bounded by 
 \begin{align}
 \int_{\epsilon }^{\epsilon\ell} |\mathbf{t}_f|^{2+\eps} y_t^{-1/2-\eps}\frac{y_t}{r} dt\ll \frac{|\mathbf{t}_f|^{2+\eps}}{r}\int_{h}^{r} y^{-1/2-\eps} dy\ll_\eps |\mathbf{t}_f|^{2+\eps}r^{-1/2-\eps}.
 \end{align}  
Finally in the Eisenstein case $F=E$ using the bound (\ref{eq:boundtilde}) for the coefficients in the constant term, the elementary inequality $|y^{1/2}\log y|\ll_\eps y^{1/2+\eps}+y^{1/2-\eps}$ to deal with the case where $t_E=0$, and the considerations above we conclude that the last term (\ref{eq:term3}) is bounded by
\begin{align*}
& \ll\int_0^{\epsilon}(|A_E|+|B_E|) (y_t^{1/2+\eps}+y_t^{1/2-\eps}) dt+\int_\epsilon^{\epsilon\ell} (|A_E|+|B_E|) (y_t^{1/2+\eps}+y_t^{1/2-\eps}) dt\\
&\ll_{N,\eps} |\mathbf{t}_E|^\eps (r^{1/2+\eps}+r^{1/2-\eps})+|\mathbf{t}_E|^\eps \int_h^r (y^{1/2+\eps}+y^{1/2-\eps}) \frac{dy}{y}\ll |\mathbf{t}_E|^\eps( r^{1/2+\eps}+r^{1/2-\eps}).
\end{align*}
Putting everything together and changing coordinates back to $y_t$ we conclude that 
 \begin{align}\label{eq:geotoadd} \int_{\CC} F(g) d\mu_{\CC}(g)=&(-1)^{k/2+1}\left(\int_h^r f(\gamma\infty+iy) \frac{dy}{y}+(-1)^{k/2}\int_h^r f(\gamma^{-1}\infty+iy) \frac{dy}{y}\right)\\
&+O_{N,\eps}\left(|\mathbf{t}_f|^{3+\eps}(r^{1/2+\eps}+r^{-1/2-\eps})\right),  \end{align}
with the improved error-term $O_{N,\eps}\left(|\mathbf{t}_f|^{3+\eps}r^{-1/2-\eps}\right)$ in the cuspidal case. Finally we need to add in and subtract off the contributions from the symmetric difference of the interval $(q^{-1},\infty)$ and $(h,r)$ as well as subtracting the contribution of the constant term $f_\infty(y)$. These contributions can be bounded using (\ref{eq:Sary1}) and (\ref{eq:boundtilde}) as follows:
\begin{align*}
\int^{q^{-1}}_h 
\left(f(\gamma^{\pm 1}\infty+iy)-f_\infty(y)\right) \frac{dy}{y}&\ll_{N,\eps} |\mathbf{t}_f|^{2+\eps}q^{3/2+\eps}(q^{-1}-h)\\ &\ll |\mathbf{t}_f|^{2+\eps}\frac{q^{3/2+\eps}}{(a+d)^2 q}\ll |\mathbf{t}_f|^{2+\eps} r^{-1/2-\eps},\\
 \int_{r}^\infty \left(f(\gamma^{\pm 1}\infty+iy)-f_\infty(y)\right) \frac{dy}{y}&\ll_{N,\eps} |\mathbf{t}_f|^{2+\eps}\int_{r}^\infty y^{-1/2-\eps} \frac{dy}{y}\ll |\mathbf{t}_f|^{2+\eps}r^{-1/2-\eps},\\
 \int^r_h 
f_\infty(y)\frac{dy}{y}&\ll_{N,\eps} |\mathbf{t}_f|^{\eps}(r^{1/2+\eps}+r^{1/2-\eps}).
\end{align*}
Thus we conclude in the cuspidal case that 
\begin{align}  
\int_h^r f(\gamma^{\pm 1}\infty+iy) \frac{dy}{y}= \int_{q^{-1}}^\infty f(\gamma^{\pm 1} \infty+iy)
 \frac{dy}{y}+O_{N,\eps}(|\mathbf{t}_f|^{2+\eps} r^{-1/2-\eps}).
\end{align}
We substitute the above into (\ref{eq:geotoadd}) and apply Proposition \ref{prop:linetoadd} to get the desired formula using that 
$$r=\tfrac{\sqrt{(a+d)^2-4}}{2q}\asymp \tfrac{|a+d|}{q}.$$
The bounds for the coefficients $c_F^\pm$ follow by combining Lemma \ref{lem:boundwhittaker} and the fact that $|\kappa_f|\in \{0,1\}$, see (\ref{eq:kappaf}).  
In the Eisenstein case we pick up an extra error-term $O_{N,\eps}\left(|\mathbf{t}_f|^{3+\eps}r^{1/2+\eps}\right)$ and then the result follows similarly applying Proposition \ref{prop:linetoadd} and observing that the last two terms in (\ref{eq:Eichlerint}) are bounded by 
$$\ll_\eps |\mathbf{t}_E|^\eps q^{-1/2+\eps}\ll |\mathbf{t}_E|^\eps( r^{1/2+\eps}+r^{-1/2-\eps}), $$ 
by using (\ref{eq:boundtilde}), so that it can be absorbed into the error-term. 
\end{proof}
 \begin{remark}
 The proof (and the statement) of Proposition \ref{prop:main} is closely related to the fact that central values of additive twist $L$-series associated to Maa{\ss} forms define \emph{quantum modular forms} in the sense of Zagier \cite{Zagier10} as proved by Sary Drappeau and the author in \cite{DrNo22}. These properties do not rely on the arithmetic properties of $\Gamma_0(N)$ and the corresponding statements of Proposition \ref{prop:main} for any discrete, cofinite subgroup of $\PSL_2(\R)$ with a cusp at $\infty$ can be obtained by the same proof with a possible worse dependence on the spectral data. 
 \end{remark}
 
\section{Weyl sums}
In this section we will bound the Weyl sums not contributing to the main term (i.e. the non-residual spectrum). The idea is to relate the Weyl sums to Dirichlet twists of the $L$-functions of the relevant Hecke--Maa{\ss} form. This is exactly the content of the Birch--Stevens formula (which takes the place of the Waldspurger formula in the automorphic approach to Duke's theorem). First of all we will develop such a formula in sufficient generality for our purposes.

\subsection{Birch--Stevens formula}
As it does not require any extra work, we will for future reference in this section work with Hecke--Maa{\ss} newforms (cupsidal or Eisenstein) of \emph{general nebentypus} $\eta_0 \modulo N$ (see e.g. \cite[Section 6]{DuFrIw02} for a definition). We define the additive twist $L$-series of such Hecke--Maa{\ss} newforms by the exact same formula (\ref{eq:defadditive}) in terms of the Hecke eigenvalues. 

Given a Dirichlet character $\chi\modulo q$ and a Hecke--Maa{\ss} newform $f$ of nebentypus $\eta_0$ which is either a cusp form or a linear combination of Eisenstein series, we define
$$L(s,f,\chi):=\sum_{n\geq 1}\frac{\lambda_f(n)\chi(n)}{n^s},\quad \Re s>1+\Re (it_f), $$
and elsewhere by meromorphic continuation, where $\lambda_f(n)$ denotes the Hecke eigenvalues of $f$. Note that the case of Eisenstein series (possibly off the critical line) correspond to; 
$$\lambda_f(n)=\sigma_{s_f-\frac{1}{2},s_f-\frac{1}{2}}(n,\chi_1,\chi_2),$$
where $\sigma_{s_1,s_2}(n,\chi_1,\chi_2)$ is defined in (\ref{eq:sigmachichi}), $\chi_i\modulo m_i$ are primitive with $\chi_1\chi_2(-1)=1$ and $s_f\in \C$. We emphasize that the only input needed in the proof of the Birch--Stevens formula is the Hecke relations. 
In the case of non-primitive characters and oldforms we will have certain local factors showing up. Given a Hecke--Maa{\ss} newform $f$, a Dirichlet character $\chi\modulo q$, and integers $n,\ell\geq 1$, they are defined by;
 \begin{align}\label{eq:localweight}\nu_s(f, \chi, n;\ell )
 := \sum_{\substack{n_1n_2n_3=n,\\(n_2n_3,\ell)|n_3}} \mu(n_1)\overline{\chi(n_1)} n_2^{1-2s} \eta_0(n_2)\mu(n_2)\chi(n_2)n_3^{1-s}\lambda_f\left(\tfrac{n_3}{(n_3,\ell)}\right)(n_3,\ell)^s\chi\left(\tfrac{\ell}{(n_3,\ell)}\right)  .  \end{align}
For brevity we put $\nu(f, \chi, n;\ell ):=\nu_{1/2}(f, \chi, n;\ell )$. Then we have the following general formula.
 
\begin{prop}[Birch--Stevens formula]\label{BirchStevens} 
Let $f\in \mathcal{A}_k(\Gamma_0(N),\eta_0)$ be a Hecke--Maa{\ss} newform of level $N$ and nebentypus $\eta_0$ which is either a cusp form or a linear combination of Eisenstein series. Let $\chi \modulo q$ be a  Dirichlet character and $\ell\geq 1$. Then for $s\in \C$ we have the following equality of meromorphic functions:
\begin{align}
\sum_{a\in (\Z/q\Z)^\times} \overline{\chi(a)}L(s, f ,\tfrac{\ell a}{q})=\nu_s(f, \chi^*,\tfrac{q}{\cc(\chi)};\ell) \tau(\overline{\chi^\ast}) L(s,f, \chi^*),
\end{align}
where $\chi^*\modulo \cc(\chi)$ denotes the unique primitive character that induces $\chi$, $\tau(\overline{\chi^*})$ is a Gau{\ss} sum and $\nu_s(\ldots)$ is as in (\ref{eq:localweight}).
\end{prop}
\begin{proof}
For $\Re s\gg_f 1$, we have by absolute convergence that 
\begin{align}\label{cross} \sum_{a\in (\Z/q\Z)^\times} \overline{\chi}(a) L(s,f,\tfrac{\ell a}{q}) = \sum_{n\geq 1} \frac{\lambda_f(n)}{n^s}\left(\sum_{a\in (\Z/q\Z)^\times}\overline{\chi}(a)e(\ell n  a/q)\right).   \end{align}
The inner sum is a Gauss sum and by \cite[Lemma 3]{Sh75}, we get
$$\sum_{a\in (\Z/q\Z)^\times}\overline{\chi}(a)e(\ell n a/q)= \tau(\overline{\chi^*}) \sum_{d\mid ( \ell n,q^\ast)}d\, \overline{\chi^*}\left(\frac{q^\ast}{d}\right) \mu\left(\frac{q^\ast}{d}\right) \chi^*\left(\frac{\ell n}{d}\right), $$
where $\chi^*\modulo \cc(\chi)$ denotes the unique primitive character that induces $\chi$ and $q^\ast=\tfrac{q}{\cc(\chi)}$. Plugging this into (\ref{cross}) and interchanging the sums, we arrive at
\begin{align} 
\tau(\overline{\chi}^*)\sum_{d\mid q^\ast} d\,\overline{\chi^*}\left(\frac{q^\ast}{d}\right) \mu\left(\frac{q^\ast}{d}\right) \sum_{\substack{n\geq 1:\\d|\ell n }} \frac{\lambda_f(n)}{n^s} \chi^*\left(\frac{\ell n}{d}\right) .  \end{align}
Now the condition $d|\ell n$ mean exactly that $n=n'\tfrac{d}{(d,\ell)}$ for some $n'\geq 1$. Since $f$ is a newform, we have the Hecke relations
$$   \lambda_f(m_1m_2)= \sum_{\substack{h\mid (m_1,m_2)}}\eta_0(h)\mu(h) \lambda_f\left(\frac{m_1}{h}\right)\lambda_f\left(\frac{m_2}{h}\right),  $$
which we apply with $m_1=n'$ and $m_2=\tfrac{d}{(d,\ell)}$. 
Thus summing over $h|\tfrac{d}{(d,\ell)}$ and writing $n'=m h$ for some $m\geq 1$, we arrive at
\begin{align}
\tau(\overline{\chi}^*) \sum_{d\mid q^\ast} d\,\overline{\chi^*}\left(\frac{q^\ast}{d}\right) \mu\left(\frac{q^\ast}{d}\right) \sum_{h|\tfrac{d}{(d,\ell)}} \eta_0(h)\mu(h)\lambda_f(\tfrac{d}{h(d,\ell)})\sum_{m\geq 1}\frac{\lambda_f(m)}{(\tfrac{d}{(d,\ell)}mh)^s}\chi^*\left(\frac{\ell \tfrac{d}{(d,\ell)}mh}{d}\right).   
\end{align}
Putting $n_1=q^\ast/d$, $n_2=h$, $n_3=d/h$, using multiplicativity of $\chi^\ast$ as well as the fact that 
$$(n_2n_3,\ell)|n_3\Rightarrow (n_2n_3,\ell)=(n_3,\ell),$$
we arrive at the wanted for $\Re s$ sufficiently large and thus for general $s$ by meromorphic continuation. \end{proof}

\subsection{Bounds for Weyl sums}
Combining all of the above we are ready to prove our bounds for the Weyl sums. Recall Definition \ref{def:psi} of a double coset embedding $\psi:(\Z/q\Z)^\times \hookrightarrow \Gamma_0(q)$, the associated oriented closed geodesics (\ref{eq:CCapsiN}) and canonical measures (\ref{eq:muapsiN}), as well as our choice of orthonormal Hecke bases (\ref{eq:BB0}) and (\ref{eq:BBE}).
\begin{thm}\label{thm:weylsums} Let $F\in \mathcal{B}_0(\Gamma_0(N))$ be a cuspidal Hecke--Maa{\ss} form of level $N$. Consider a double coset embedding $\psi$ of level $q$ with $N|q$, and a coset $cH\subset (\Z/q\Z)^\times$ with $H\leq (\Z/q\Z)^\times $ a subgroup. Then for any $\eps>0$ we have 
\begin{equation}\label{eq:weylsums1} \sum_{a\in cH} \int_{\CC^\psi_a(N)} F(g)d\mu_a^\psi(g) \ll_{N,\eps} |\mathbf{t}_F|^{9/2+\eps}q^{7/8+\eps}+|\mathbf{t}_F|^{3+\eps}\sum_{a\in cH} \left(\tfrac{q}{|\tr(\psi(a))|}\right)^{1/2+\eps}. \end{equation}
Similarly, for an Eisenstein series $E\in \mathcal{B}_e(\Gamma_0(N))$ we have 
\begin{equation}\label{eq:weylsums2}  \sum_{a\in cH} \int_{\CC^\psi_a(N)}E(g)d\mu_a^\psi(g) \ll_{N,\eps} |\mathbf{t}_E|^{11/6+\eps}q^{5/6+\eps}+|\mathbf{t}_E|^{3+\eps}\sum_{a\in cH} \left(\left(\tfrac{q}{|\tr(\psi(a))|}\right)^{1/2+\eps}+\left(\tfrac{q}{|\tr(\psi(a))|}\right)^{-1/2-\eps}\right). \end{equation} 
\end{thm}
\begin{proof}
Consider first of all the cuspidal case. By character orthogonality and Proposition \ref{prop:main} we get
\begin{align}
&\sum_{a\in cH} \int_{\CC^\psi_a(N)} F(g)d\mu_a^\psi(g) \\
=&\sum_{a\in cH}\left(\left(c_F^+L(1/2,F,\tfrac{a}{q})+c_F^-L(1/2,F,\tfrac{-a}{q})\right)+O_\eps\left( |\mathbf{t}_F|^{3+\eps}\left(\tfrac{q}{|\tr(\psi(a))|}\right)^{1/2+\eps}\right)\right) \\
\label{eq:innersum}=& \frac{|H|}{\varphi(q)} \sum_{\chi\, (q): \chi_{|H}=1}\chi(c)\sum_{a\in (\Z/q\Z)^\times} \overline{\chi(a)}\biggr(c_F^+L(1/2,F,\tfrac{a}{q})+c_F^- L(1/2,F,\tfrac{-a}{q})\biggr)\\
&+O_\eps\left(|\mathbf{t}_F|^{3+\eps}\sum_{a\in cH}\left( \tfrac{q}{|\tr(\psi(a))|}\right)^{1/2+\eps}\right). 
\end{align}
Let $\ell\geq 1$ be an integer and $f^\ast$ a Hecke--Maa{\ss} newform of level $N'$ with  $\ell N'|N$ such that $F=(F^\ast)_\ell$ as in (\ref{eq:coefONB}) where $F^\ast$ denotes the lift of $f^\ast$ to the unit tangent bundle. Using the formula (\ref{eq:oldformAT}) for oldform additive twist $L$-series we conclude by the Birch--Stevens formula as in Proposition \ref{BirchStevens} that the inner-sum (in the double sum) of equation (\ref{eq:innersum}) is equal to 
\begin{equation}\label{eq:Weylbndpf} (c_F^+ +\chi(-1)c_F^-)\tau(\overline{\chi^\ast}) \sum_{d|\ell} c_{F^\ast}(d,\ell) L(1/2,f^\ast,\chi^\ast) \nu(f^\ast, \chi^\ast, \tfrac{q}{\cc(\chi)};d),   \end{equation}
where $ |c_{F^\ast}(d,\ell)|\ll_{N,\eps} |\mathbf{t}_F|^\eps$ and $|c_{F}^\pm|\ll_N |\mathbf{t}_F|^{3/2}$ (here the term with $d|\ell$ corresponds to the contribution from the term $F^\ast(dg)$ in (\ref{eq:coefONB}) which in view of (\ref{cross}) shows up as the last parameter $\nu(\ldots; d)$). By applying the subconvexity bound of Blomer--Harcos \cite[eq.\ (1.3)]{BlomerHarcos08} as well as the Kim--Sarnak bound \cite{KiSa03} for Hecke eigenvalues to bound $\nu(\ldots)$ and $|\tau(\overline{\chi^\ast})|= \sqrt{\mathbf{c}(\chi)}$, the quantity (\ref{eq:Weylbndpf}) is bounded by
$$ \ll_{N,\eps} |\mathbf{t}_F|^{9/2+\eps} q^\eps \cc(\chi)^{1/2} \cc(\chi)^{3/8} q^{39/64} \cc(\chi)^{-39/64}\ll_{N,\eps} |\mathbf{t}_F|^{9/2+\eps}  q^{7/8+\eps}.$$
This gives the wanted since the number of Dirichlet characters $\chi \modulo q$ such that $\chi_{|H}=1$ is exactly $\tfrac{\varphi(q)}{|H|}$. In the Eisenstein case we proceed as above using the bound (\ref{eq:eisensteinmain}) instead: by (\ref{eq:Eisensteinlift}) and the Birch--Stevens formula there is some primitive Dirichlet character $\eta\modulo m$ with $m^2|N$ and $t\in \R$ such that the inner sum is equal to  
\begin{equation}\label{eq:oldeiscalc}(c_E^+ +\chi(-1)c_E^-) \tau(\overline{\chi^\ast}) \sum_{d|\ell}c_{t,\eta}(d,\ell)L(1/2,\tilde{E}_{\eta,it}, \chi^\ast) \nu(\tilde{E}_{\eta,it},\chi^\ast,  \tfrac{q}{\cc(\chi)};d),  \end{equation}
where $\tilde{E}_{\eta,it}(z):=\tilde{E}_{\eta,0}(z,1/2+it)$ in terms of the normalized Eisenstein series (\ref{eq:EisensteinFouriertilde}) and we have the bounds 
$$|c_E^\pm|\ll_{N} |\mathbf{t}_E|^{3/2} ,\quad |c_{t,\eta}(d,\ell)|\ll_{N,\eps} |\mathbf{t}_E|^\eps.$$ Since the Hecke eigenvalues of $\tilde{E}_{\eta,it}$ are given by the convolution product $(|\cdot|^{-it}\eta)\ast (|\cdot|^{it}\overline{\eta})$) we get the factorization
$$L(1/2,\tilde{E}_{\eta,it}  , \chi^\ast)=  L(1/2+it,\eta\chi^\ast)L(1/2-it,\overline{\eta}\chi^\ast).$$
Thus by bounding the Gau{\ss} sum and the term $\nu(\ldots)$ as above (using a simple bound for the divisor function) and using the Petrow--Young bound \cite{PetrowYoung19};
$$L(1/2+it,\eta\chi^\ast)\ll_\eps (\cc(\chi)(|t|+1) m)^{1/6+\eps}, $$
we conclude the bound (\ref{eq:weylsums2}).
\end{proof}
In the case of weight $2$ holomorphic forms we get the following improvement uniform in the level.
\begin{cor}\label{cor:holoWeyl}
Let $f\in \mathcal{S}_2(\Gamma_0(N))$ be a cuspidal holomorphic Hecke eigenform of weight $2$ and level $N$ normalized so that the first Fourier coefficient is $1$. Let $\psi$ be a double coset embedding of level $q$ with $N|q$. Let $cH\subset (\Z/q\Z)^\times$ be a coset with $H\leq (\Z/q\Z)^\times $ a subgroup. Then we have
$$\sum_{a\in cH} \int_{\CC^\psi_a(N)} f(z)dz \ll_{\eps} N^{1/4} q^{7/8+\eps}+ N^{3/4} q^{3/4+\eps}.$$
\end{cor}
\begin{proof}
We apply the the same Fourier theoretic argument as in the proof of Theorem \ref{thm:weylsums} combined with the Birch--Stevens formula (Proposition \ref{BirchStevens}) to reduce to $L$-functions, using here Proposition \ref{prop:holoaddtwist} in place of Proposition \ref{prop:main} (so that there is no error-term). The result now follows by the Blomer--Harcos bound (cf. \cite[Theorem 2]{BlomerHarcos08})
\begin{align}
\nonumber L(1/2,f , \chi^\ast)&\ll_\eps N^{1/4+\eps}\mathbf{c}(\chi)^{3/8+\eps} +N^{1/2+\eps}(N,\mathbf{c}(\chi))^{1/4} \mathbf{c}(\chi)^{1/4+\eps}\\
&\leq N^{1/4+\eps}q^{3/8+\eps} +N^{3/4+\eps} q^{1/4+\eps},
\end{align}
together with the bound
 $$\tau(\overline{\chi^\ast})\nu(f, \chi^*,\tfrac{q}{\cc(\chi)};1)\ll_\eps \mathbf{c}(\chi)^{1/2} \sum_{n_1n_2n_3=q/\cc(\chi)}n_3^{1/2+\eps} \ll_\eps q^{1/2+\eps}, $$
using the Ramanujan bound for weight $2$ holomorphic forms and a divisor bound.
\end{proof}
\section{The main terms}
In this section we will obtain lower bounds for the contribution corresponding to the main terms in Theorems \ref{thm:first} and \ref{thm:second}. As such the considerations in this section plays the role of Siegel's lower bound $|\Cl_K^+|\gg_\eps D_K^{1/2-\eps}$ in the automorphic approach to Duke's Theorem as explained in the introduction, see (\ref{eq:Siegel}).
\subsection{Distribution of multiplicative inverses}\label{sec:distri} 
When lower bounding the lengths of $q$-orbits of closed geodesics, special care will have to be given to those $a\in (\Z/q\Z)^\times$ for which $a+\overline{a}$ with $a\overline{a}\equiv 1\modulo q$ is close to $0$ modulo $q$ i.e.\ when the minimal trace  $|t_a|$ is very small (with $t_a$ defined as in (\ref{eq:ta})). This will require information about the fine scale distribution of $(\tfrac{a}{q},\tfrac{\overline{a}}{q})$ on the torus. This is a very classical topic (see for instance \cite{Humphries22} and the references therein). We will below provide a proof of the exact estimates we will need. We will need power saving estimates for Kloosterman sums over cosets of subgroups $H\leq (\Z/q\Z)^\times$ defined as follows for $c\in (\Z/q\Z)^\times$; 
\begin{equation}
S_{cH}(m,n;q):=\sum_{a\in cH}e(\tfrac{ma+n\overline{a}}{q}).
\end{equation}
The first approach to bounding the above is to relate the above (via Fourier theory) to twisted Kloosterman sums. 
\begin{lemma}\label{lem:WeylSaliebnd}Let $q\geq 2$ and $m,n$ be integers. Assume that either $q|n$ or $q$ is cube-free. Then for any subgroup $H\leq (\Z/q\Z)^\times$ and any coset $cH\subset (\Z/q\Z)^\times$ we have 
$$|S_{cH}(m,n;q)|\leq \sigma_0(q)(m,n,q)^{1/2}q^{1/2}.$$
\end{lemma} 
\begin{proof}
By Fourier expanding over Dirichlet characters modulo $q$ we get
\begin{equation}\label{eq:ScH}S_{cH}(m,n;q)=\frac{|H|}{\varphi(q)}\sum_{\chi(q): \chi_{|H}=1}\overline{\chi}(c) S_\chi(m,n;q),\end{equation}
where the sum is over Dirichlet characters modulo $q$ with trivial restriction to $H$ and 
$$S_\chi(m,n;q):=\sum_{a\in (\Z/q\Z)^\times}\chi(a)e(\tfrac{ma+n\overline{a}}{q}).$$
If $n=0$, then $S_\chi(m,0;q)$ reduces to a Gauss sum which implies by \cite[Lemma 3.2]{IwKo} that
$$|S_\chi(m,0;q)|= \left|\tau(\chi^\ast)\sum_{d|(m,q/q^\ast)}d\, \overline{\chi}^\ast(\tfrac{m}{d})\mu(\tfrac{q}{q^\ast d}) \right|\leq \sigma_0(q)(m,q)^{1/2}q^{1/2},$$
where $\chi^\ast\modulo q^\ast$ denotes the unique primitive character inducing $\chi \modulo q$. Similarly, if $q$ is cube-free then we have the bound 
$$|S_\chi(m,n;q)|\leq \sigma_0(q)(m,n,q)^{1/2}q^{1/2},$$
which follows from \cite[Section 9]{KniLi13} as pointed out in \cite{Humphries22}. Inserting these bounds into the Fourier expansion (\ref{eq:ScH}) we get the wanted.
\end{proof}
We remark that a similar argument using \cite[Section 9]{KniLi13} yields the bound 
$$|S_{cH}(m,n;q)|\leq \sigma_0(q)(m,n,q)^{1/2}q^{5/6},$$
for general $q$. However, a simple Cauchy--Schwarz yields the following stronger bound\footnote{The author would like to thank Igor Shparlinski for suggesting this.}. 
\begin{cor}\label{cor:Klooster}
Let $q\geq 2$, $cH\subset (\Z/q\Z)^\times$ a coset of a subgroup $H\leq (\Z/q\Z)^\times$, and $m,n\in \Z$. Then for all $\eps>0$ we have the bound
\begin{equation}|S_{cH}(m,n;q)|\ll_\eps (m,n,q)^{1/4}q^{1/2+\eps}|H|^{1/4}.\end{equation}
\end{cor}
\begin{proof}
We start from the trivial equality
$$|H|S_{cH}(m,n;q)=\sum_{b\in H} \sum_{a\in cH}e(\tfrac{mab+n\overline{ab}}{q}).$$
This implies by Cauchy--Schwarz that
\begin{align*} 
|H|^2|S_{cH}(m,n;q)|^2\leq |H| \left( \sum_{b\in H} \left|\sum_{a\in cH} e(\tfrac{mab+n\overline{ab}}{q})\right|^2 \right)
= |H| \sum_{a_1,a_2\in cH} S(m(a_1-a_2),n(\overline{a_1}-\overline{a_2});q).   
\end{align*}
By the Salie--Weil bound \cite[eq.\ (9.6)]{KniLi13} we see that the above is bounded by
\begin{align}
\leq |H|
\sum_{a_1,a_2\in cH} \sigma_0(q)q^{1/2} (m(a_1-a_2),n(\overline{a_1}-\overline{a_2}),q)^{1/2}.\end{align}
Now observe that $\overline{a_1}-\overline{a_2}=\overline{a_1}\overline{a_2}(a_2-a_1)$ which implies that 
\begin{align}
(m(a_1-a_2),n(\overline{a_1}-\overline{a_2}),q)
&\leq ((a_1-a_2)m,(a_1-a_2)n,q)\\ &\leq (a_1-a_2,q)(m,n,q)=(c^{-1}(a_1-a_2),q)(m,n,q).\end{align}
 Thus we arrive at the following bound;
\begin{align}
|H| |S_{cH}(m,n;q)|^2 \leq\sigma_0(q)q^{1/2} (m,n,q)^{1/2}  \sum_{a_1,a_2\in H} (a_1-a_2,q)^{1/2}.
\end{align}
The inner sum we can bound by
$$
\sum_{a_1,a_2\in H} (a_1-a_2,q)^{1/2}
\leq \sum_{d|q} d^{1/2} \sum_{a\in H} |H \cap ((a+d\Z)/q\Z)| 
$$
and using that $|H \cap ((a+d\Z)/q\Z)|\leq \min(|H|, q/d)$ we arrive at the bound
\begin{align}
|H|\left|S_{cH}(m,n;q)\right|^2
&\leq \sigma_0(q)q^{1/2} (m,n,q)^{1/2}\left(\sum_{d|q: d\leq q/|H|} d^{1/2} |H|^{2}+\sum_{d|q: d> q/|H|} d^{-1/2}q|H|\right) \\
 &\ll_\eps(m,n,q)^{1/2}q^{1/2+\eps}(|H|^{3/2}q^{1/2+\eps}),\end{align} 
which yields the wanted bound.
\end{proof}
Note that the above yields a non-trivial bound whenever $|H|\gg q^{\delta}$ with $\delta>2/3$.
\begin{remark}\label{rem:Igor} A natural barrier for the above methods is (at best) the exponent $q^{1/2}$ which is trivial when $|H|\leq q^{1/2}$. Using different methods, Ostafe--Shparlinski--Voloch \cite{OsShpJose22} obtained a power saving estimate in the case where $q$ is prime for subgroups $H$ of size $\gg q^{\delta}$ with $\delta>\tfrac{3}{7}$. Most likely one can also go beyond the square-root barrier for general prime powers using classical techniques from the theory of exponential sums. If furthermore $n=0$, then a power saving estimate was obtained for any $\delta>0$ by Bourgain--Konyagin \cite{BourgainKonyagin03}. We would like to thank Igor Shparlinski for pointing this out.
\end{remark}

Now we are ready to prove the fine scale equidistribution result we will need. 
\begin{prop}\label{prop:smalltrace}
Let $T\geq 1$ be a parameter and $\phi_T:(\R/\Z)^2\rightarrow \C$ a smooth function satisfying
$$ \left|\frac{\partial^n \phi_T}{\partial x^n}(x,y)\right|,\left|\frac{\partial^n \phi_T}{\partial y^n}(x,y)\right|\leq C_n T^n,\quad \forall n\geq 1, $$
for some constants $C_n>0$ independent of $T$. Let $q\geq 2$ and let $cH\subset (\Z/q\Z)^\times$ be a coset of a subgroup $H\leq (\Z/q\Z)^\times$. Then for any $\eps>0$ we have 
\begin{align}\label{eq:equidaabar}
\sum_{a\in cH} \phi_T \left(\frac{a}{q}, \frac{\overline{a}}{q}\right)= |H|\int_{(\R/\Z)^2}\phi_T(x,y)dxdy+ O_{\eps}\left(|\!| \phi_T|\!|_1 T^{2+\eps}q^{1/2+\eps}|H|^{1/4}+q^{-100}\right),
\end{align} 
where the implied constant is allowed to depend on the constants $C_1,C_2,\ldots$ and 
$$|\!| \phi|\!|_1=\int_{(\R/\Z)^2}|\phi(x,y)|dxdy.$$
In the case where either $q$ is cube-free or $\phi_T$ depends only on the first argument, then we have the improved error term 
\begin{equation}\label{eq:improvedE}O_{\eps}\left(|\!| \phi_T|\!|_1 T^{2+\eps}q^{1/2+\eps}+q^{-100}\right).\end{equation} 
\end{prop} 
\begin{proof}
By Fourier expansion on the torus we get
\begin{align*}
&\sum_{a\in cH} \phi_T \left(\frac{a}{q}, \frac{\overline{a}}{q}\right)&\\
&= |H|\int_{(\R/\Z)^2}\phi_T (x,y)dxdy+ \sum_{(m,n)\neq (0,0)}\int_{(\R/\Z)^2}\phi_T (x,y)e(-mx-ny)dxdy S_{cH}(m,n;q). 
\end{align*}
By partial integration we see that for $(m,n)\neq (0,0)$ with $|m|\geq |n|$ and any integer $A\geq 2$, we have by the assumption on $\phi_T$ that
$$\int_{(\R/\Z)^2}\phi_T (x,y)e(-mx-ny)dxdy\ll_A |m|^{-A} \int_{(\R/\Z)^2}  \left|\frac{\partial^A \phi_T}{\partial x^A}(x,y)\right|dxdy  \ll_{A} \frac{T^{A}}{(|n|+|m|)^A}. $$
Thus we may truncate the sum at $|m|,|n|\leq q^\eps T^{1+\eps}$ at the cost of  $O(q^{-100})$, say. 
By taking absolute values and applying Corollary \ref{cor:Klooster}, the remaining part of the sum is bounded by $|\!| \phi_T|\!|_1$ 
multiplied by
\begin{align} q^{1/2+\eps}|H|^{1/4}\sum_{(m,n)\neq (0,0): |m|,|n|\leq q^\eps T^{1+\eps}}  (m,n,q)^{1/4}&\ll_\eps q^{1/2+\eps}|H|^{1/4} \sum_{d|q}d^{1/4} \left\lfloor\tfrac{q^\eps T^{1+\eps}}{d}\right\rfloor\left(\left\lfloor\tfrac{q^\eps T^{1+\eps}}{d}\right\rfloor+1\right) \\
&\ll_\eps   T^{2+\eps} q^{1/2+3\eps}|H|^{1/4}  \end{align} 
as wanted. If either $q$ is cube-free or if $\phi_T$ depends only on the first parameter, then we get the improved exponent by inserting the bound from Lemma \ref{lem:WeylSaliebnd} instead. This yields the estimate with the improved error-term. \end{proof}
We will apply the above to obtain a number of useful estimates involving the minimal trace $t_a$ as defined in (\ref{eq:ta}). We recall the following standard construction of a smooth approximation of the indicator function $1_{[0,1]}:\R\rightarrow \{0,1\}$ on the unit interval. Fix $\varphi:\R\rightarrow \R_{\geq 0}$ a smooth bump function supported in $[-1,1]$ with $\int_{-1}^1 \varphi(t)dt=1$. Then for $T\geq 1$ we define $ \varphi_{T}(x):=T\varphi( Tx)$ which is supported in $[-T^{-1},T^{-1}]$ and satisfies $\int_{-1}^1 \varphi_T(x)dx=1$. This will serve as an approximation of the Dirac measure at $x=0$ and we define the smooth function $\psi_{T}:\R\rightarrow [0,1]$ as the following convolution; 
$$ \psi_{T}(y)=\int_\R 1_{[0,1]}(y-x)\varphi_T(x)dx, $$
satisfying the following key properties which are standard to prove;  
\begin{equation} \label{eq:bytheabove2}|\psi_{T}(y)|\leq 1, \quad y\in \R,\qquad \left|\frac{d^n\psi_{T}}{dy^n}(y)\right|\ll_n T^n,\quad n\geq 1,\end{equation} 
\begin{equation}\label{eq:bytheabove}
\psi_{T}(y)=\begin{cases}1,& y\in (T^{-1},1-T^{-1})\\ 0,& y\in (-\infty,-T^{-1})\cup (1+T^{-1},\infty)\end{cases},\quad \int_{\R}\psi_{T}(x)x^mdx= \frac{1}{m+1}+O_m(T^{-1}),\quad m\geq 0.
\end{equation} 


\begin{cor}\label{cor:ta}
Let $q\geq 2$ and let $cH\subset (\Z/q\Z)^\times$ be a coset of a subgroup $H\leq (\Z/q\Z)^\times$. Then for $T\geq 1$, $\delta\in (-\tfrac{1}{2},\tfrac{1}{2})$ and $\eps>0$ the following holds:
\begin{align}\label{eq:cor1}|\{a\in cH: |\tfrac{t_a}{q}-\delta|\leq T^{-1}\}|&\ll_\eps T^{-1}|H|+T^{1+\eps} q^{1/2+\eps}|H|^{1/4},\\
\label{eq:cor4}|\{a\in cH: |\{\tfrac{a}{q}\}-1/2-\delta|\leq T^{-1}\}|&\ll_\eps T^{-1}|H|+T^{1+\eps} q^{1/2+\eps} ,\\
\label{eq:cor2}\sum_{a\in cH} \frac{t_a}{q}&\ll_\eps |H|^{3/4}q^{1/6+\eps},\\
\label{eq:cor3}\sum_{\substack{0<a<q:\\ (a\modulo q)\in cH}} \frac{a}{q}&=\tfrac{1}{2}|H|+O_\eps(|H|^{2/3}q^{1/6+\eps}).\end{align}
 
\end{cor}
\begin{proof} For $\varphi:\R\rightarrow \C$ smooth and of compact support, we obtain a function $\phi:(\R/\Z)^2\rightarrow \C$ by 
$$\phi(x,y):=\sum_{k\in \Z} \varphi(x+y+k).$$
By Fubini's Theorem, a change of variable, unfolding and another change of variable, we get the following useful fact; 
\begin{equation}\label{eq:COVunfolding}\int_{(\R/\Z)^2} \phi(x,y) dxdy=\sum_{k\in \Z} \int_{0}^1\int_{k}^{k+1}\varphi(x+y)dxdy=\int_0^1\int_{-\infty}^{\infty} \varphi(x+y)dxdy = \int_{-\infty}^{\infty} \varphi(x)dx.\end{equation}
The first estimate follows directly from Proposition \ref{prop:smalltrace} and (\ref{eq:COVunfolding}) by letting 
$$\phi_T(x,y)= \sum_{k\in \Z}\varphi(T(x+y+k-\delta)),$$ 
where $\varphi:\R\rightarrow \R_{\geq 0}$ is some smooth  positive function with compact support contained in $(-\tfrac{1}{2},\tfrac{1}{2})$ and $\varphi(y)=1$ for $y\in (-\tfrac{1}{4},\tfrac{1}{4})$, say, so that $|\!|\phi_T|\!|_1\asymp T^{-1}$. The second estimate follows similarly by applying Proposition \ref{prop:smalltrace} with the improved error-term (\ref{eq:improvedE}) to 
$$ \phi_T(x,y)= \sum_{k\in \Z}\varphi(T(x+k-\delta-1/2)).$$ 

To prove the third estimate, note first of all that the left-hand side is trivially bounded by $\frac{1}{2}|H|$ and so we can assume that $|H|\geq q^{2/3}$. Consider now the test function
$$\phi_T(x,y)=-\tfrac{1}{2}+\sum_{k\in \Z}(x+y+k+\tfrac{1}{2})\psi_{T}(x+y+k+\tfrac{1}{2}),$$
with $\psi_T$ defined as above. Then we have by the definition (\ref{eq:ta}) of $t_a$, the properties (\ref{eq:bytheabove2}), (\ref{eq:bytheabove}) and the estimate (\ref{eq:cor1}) that
\begin{align*}
\sum_{a\in cH} |\tfrac{t_a}{q}-\phi_T(\tfrac{a}{q},\tfrac{\overline{a}}{q})|&\ll \#\{a\in cH: \tfrac{t_a}{q}\in [-\tfrac{1}{2}-T^{-1},-\tfrac{1}{2}]\cup [\tfrac{1}{2},\tfrac{1}{2}+T^{-1}]\}\\
&\ll_\eps T^{-1}|H|+T^{1+\eps}q^{1/2+\eps}|H|^{1/4}. 
\end{align*}
Now we apply Proposition \ref{prop:smalltrace} with the test function $\phi_T$. By the equality (\ref{eq:COVunfolding}) and the estimate (\ref{eq:bytheabove}) we see that the main term integral amounts to
$$\int_{(\R/\Z)^2}\phi_T(x,y)dxdy=-\frac{1}{2}+\int_{\R}\psi_{T}(x)xdx=-\frac{1}{2}+\frac{1}{2}+O(T^{-1})\ll T^{-1}, $$
and so using the trivial bound $|\!|\phi_T|\!|_1\ll 1$ we arrive at the bound
$$ \ll_\eps T^{-1}|H|+T^{1+\eps}q^{1/2+\eps}|H|^{1/4}+T^{2+\eps}q^{1/2+\eps}|H|^{1/4}.  $$
Picking $T=|H|^{1/4}q^{-1/6}\geq 1$ yields the wanted (using here that $|H|\geq q^{2/3}$). The final estimate follows similarly using the test function 
$$\phi_T(x,y)=\sum_{k\in \Z}(x+k)\psi_{T}(x+k), $$
the improved error term (\ref{eq:improvedE}), the previous estimate (\ref{eq:cor4}) and picking $T=|H|^{1/3}q^{-1/6}$ when $|H|\geq q^{1/2}$.
\end{proof}


\subsection{Lengths of closed geodesics}
A key application of the above is the following lower bound for the closed geodesics associated to an arbitrary double coset embedding. 
\begin{prop}\label{prop:lowerboundgeo}
Let $N\geq 1$ and $\delta>2/3$. Let $q\geq 2$ be an integer such that $N|q$ and consider a double coset embedding $\psi$ of level $q$. Let $H\leq (\Z/q\Z)^\times $ be a subgroup with $|H|\geq q^{\delta}$ and $cH\subset (\Z/q\Z)^\times$ a coset. Then we have 
\begin{equation} \sum_{a\in cH} |\CC_a^\psi(N)|\gg |H|\log q. \end{equation}
If furthermore $q$ is cube-free, the same conclusion holds as long as $\delta>1/2$.
\end{prop}
\begin{proof}
We apply the estimate  (\ref{eq:equidaabar}), resp. (\ref{eq:improvedE}), from Proposition \ref{prop:smalltrace} (with $T=1$) to the function 
$$(x,y)\mapsto \sum_{k\in \Z}\varphi(x+y+k),$$ 
where $\varphi:\R\rightarrow \R_{\geq 0}$ is some smooth  bump function with support contained in $(\tfrac{1}{5},\tfrac{1}{2})$ and $\varphi(y)=1$ for $y\in (\tfrac{1}{4},\tfrac{1}{3})$, say. This implies that for a positive proportion of $a\in cH$ we have that $|t_a|\geq \tfrac{1}{6}q$, with $t_a$ defined as in (\ref{eq:ta}). Thus we conclude by Lemma \ref{lem:lowerbndtrace} that for a positive proportion of  $a\in cH$ we have $|\tr(\psi(a))|\geq \tfrac{1}{6}q$ and thus by (\ref{eq:geodesiclength}) it follows that $|\CC_a^\psi(N)|\gg \log q $. This gives the wanted lower bound for the geodesic lengths.
\end{proof}
\begin{remark}
Note that in the case where $q$ is prime, a similar lower bound holds whenever $|H|\gg q^{\delta}$ with $\delta>3/7$ by the results mentioned in Remark \ref{rem:Igor}.
\end{remark}

\subsection{Eisenstein contribution in homology}\label{sec:Eiscontri}
In this section we will determine the main term in the homological equidistribution problem as in Theorem \ref{thm:second} (i.e.\ the contribution coming from the Eisenstein element). Combining Hecke's limit formula (\ref{eq:limitformula}), the Fourier expansion (\ref{eq:E2ast}) of the modified Eisenstein series $E^\ast_2$ and the formula for the normalized Whittaker function as in (\ref{eq:explicitWtilde}), we conclude that 
\begin{align}
\label{eq:E}E(z)&:=\tfrac{-\pi y}{6}E^\ast_2(z)\\
\nonumber&=-\tfrac{\pi }{6}y+\tfrac{1}{2}+4\pi y\sum_{n\geq 1} \sigma_1(n)e^{2\pi i n z}= -\tfrac{\pi}{6}y+\tfrac{1}{2}+\sum_{n\geq 1}\frac{\sigma_1(n)}{n}\tilde{W}_{1,\frac{1}{2}} (4\pi n y) e(nx),
\end{align} 
is a Hecke--Maa{\ss} newform of weight $2$ and level $1$ (since $E_{\infty, 2}(z,s)$ is of weight $2$ and level $1$). Furthermore, again by the limit formula (\ref{eq:limitformula}) and the fact that Eisenstein series are Hecke eigenforms \cite[eq.\ (4.13)]{Young19} we conclude (by interchanging the limit and the finite sum in the definition (\ref{eq:heckeoperator})) that 
$$T_n E= \frac{\sigma_{1}(n)}{n^{1/2}}E=\sigma_{1/2,1/2}(n)E,\quad n\geq 1,$$ 
where $T_n$ denotes the Hecke operator as defined in (\ref{eq:heckeoperator}), $\sigma_{s_1,s_2}(n)=\sum_{ab=n}a^{-s_1}b^{s_2}$ and $\sigma_1(n)=\sum_{d|n}d$.
In view of this the local weights (\ref{eq:localweight}) are given by 
\begin{equation} \label{eq:nuEisenstein}\nu(\chi,q;\ell):=\nu_{1/2}(E,\chi,q;\ell)=\sum_{\substack{n_1n_2n_3=q,\\(n_2n_3,\ell)|n_3}} \mu(n_1)\overline{\chi(n_1)}\mu(n_2)\chi(n_2)\sigma_1\left(\tfrac{n_3}{(n_3,\ell)}\right)(n_3,\ell)\chi\left(\tfrac{\ell}{(n_3,\ell)}\right),\end{equation}
which, as we shall see, take a particularly nice form. The key result is the following Birch--Stevens formula for the Eisenstein class.
\begin{prop}\label{cor:Eiscontr}
Let $N\geq 2$ be prime and denote by $\omega_{E}(N)$ the Eisenstein class in cohomology as defined in (\ref{eq:eisensteinclass}). For each $q\geq 2$ such that $N|q$ consider a double coset embedding $\psi$ of level $q$ and a Dirichlet character $\chi \modulo q$. Then we have 
\begin{align}
\nonumber &\sum_{a\in (\Z/q\Z)^\times} \overline{\chi(a)} \langle [\CC^\psi_a(N)], \omega_{E}(N)\rangle_\mathrm{cap}\\
\label{eq:eiscohoformula}&= \frac{6(1-\chi(-1))}{\pi^2}\frac{|\chi^\ast(N)-1|^2}{N-1} |L(1,\chi^\ast)|^2 \cc(\chi) \nu(\chi^\ast,\tfrac{q^\ast}{(q^\ast,N^\infty)};1)+\sum_{a\in (\Z/q\Z)^\times} \overline{\chi(a)}\frac{\tr\,( \psi(a))}{q},  \end{align} 
where $\chi^\ast\modulo \cc(\chi)$ is the unique primitive character inducing $\chi$ and $q^\ast=\tfrac{q}{\cc(\chi)}$. Here the first term is interpreted as $0$ if  $\chi$ is principal.
\end{prop}
In proving this we will need the following lemma on the local weights appearing in the Birch--Stevens formula as in Proposition \ref{BirchStevens}. It will be convenient to define 
\begin{equation}\label{eq:sigma1}\sigma_1(p^m):= \frac{p^{m+1}-1}{p-1},\end{equation} 
for all primes $p$ and integers $m\in \Z$.
\begin{lemma}\label{lem:E2weightlowerbnd}
Let $q\geq 2$ and let $\chi\modulo D$ be a Dirichlet character (not necessarily primitive). Then the weight $\nu(\chi,q;1)$ is a real number satisfying
$$\nu(\chi,q;1)\geq \varphi(q). $$
More precisely for $p$ prime and $m\geq 1$, we have  
$$ \nu(\chi,p^m;1)=\frac{p^{m-1}|p-\chi(p)|^2-|1-\chi(p)|^2}{p-1}. $$
\end{lemma}
\begin{proof}
We observe that $\nu(\chi,q;1)$ is a convolution of multiplicative functions and so is itself a multiplicative function of $q$. By definition we see that 
\begin{equation}\label{eq:chipm1}\nu(\chi,p^m;1)=\sigma_1(p^m)-\sigma_1(p^{m-1})(\chi(p)+\overline{\chi}(p))+\sigma_1(p^{m-2})|\chi(p)|^2,\quad m\geq 1,\end{equation}
which also holds for $m=1$ using the above extension (\ref{eq:sigma1}) of $\sigma_1$. We conclude the required formula after a short calculation. Now it follows readily from (\ref{eq:chipm1}) that the expression is minimized when $\chi(p)=1$ in which case $\nu(\chi,p^m;1)=p^{m-1}(p-1)=\varphi(p^m)$. Thus the result follows by multiplicativity.
\end{proof}
\begin{lemma}\label{lem:Eisweights}
Let $q\geq 2$ be an integer, $\ell$ a prime and $\chi\modulo D$ a Dirichlet character (not necessarily primitive). Then we  have 
$$\nu(\chi,q;\ell)-\nu(\chi,q;1)=
\begin{cases}- (1-\chi(\ell))\nu(\chi,q;1),& (\ell,q)=1,
\\ - |1-\chi(\ell)|^2\nu(\chi,\tfrac{q}{(q,\ell^\infty)};1),& \ell|q.
\end{cases} $$
\end{lemma}
\begin{proof}
Let $\ell^m|\!| q $ with $m\geq 0$. We start by observing that by the explicit formula for the weights (\ref{eq:nuEisenstein}) we have the factorization 
$$\nu(\chi,q;\ell)=\nu(\chi,\ell^m;\ell) \nu(\chi,\tfrac{q}{\ell^m};1).$$ 
This means that $\nu(\chi,q;\ell)=\nu(\chi,q;1)\chi(\ell)$ if $(q,\ell)=1$, which yields the wanted formula in this case. Now assume that $m\geq 1$. We see again by the explicit formula (\ref{eq:nuEisenstein}) that
$$ \nu(\chi,\ell^m;\ell)=\ell\sigma_1(\ell^{m-1})-\ell\sigma_1(\ell^{m-2})(\chi(\ell)+\overline{\chi}(\ell))+\ell\sigma_1(\ell^{m-3})|\chi(\ell)|^2, $$
using again the extension (\ref{eq:sigma1}) of $\sigma_1$. Subtracting the formula (\ref{eq:chipm1}) from the above and using that $\ell\sigma_1(\ell^{j-1})-\sigma_1(\ell^{j})=-1$ for any $j\in \Z$, yield the wanted formula for $m\geq 1$.  
\end{proof}

\begin{proof}[Proof of Proposition \ref{cor:Eiscontr}] 
If $\chi=\chi_0$ is principal then the first term of (\ref{eq:eiscohoformula}) vanishes. The explicit formula for integrals of the holomorphic Eisenstein series of level $N$ along closed geodesics in terms of Dedekind sums $s(a,q)$ as in \cite[eq.\ (3.20)]{Nor22} gives
\begin{align*} &\langle [\CC^\psi_a(N)], \omega_{E}(N)\rangle_\mathrm{cap}\\
&= \frac{1}{N-1}\left[\frac{\tr(\psi(a))}{q/N}-12\, s(a,q/N)-3\, \sgn(\tr(\psi(a))) -\left(\frac{\tr(\psi(a))}{q}-12\, s(a,q)-3\, \sgn(\tr(\psi(a)))\right)  \right]\\
&= \frac{12(s(a,q)-s(a,q/N))}{N-1}+\frac{\tr(\psi(a))}{q}.  \end{align*}
The result follows in this case by summing the above over $a\in (\Z/q\Z)^\times $ and using the standard symmetry $s(-a,q)=-s(a,q)$. 

Now let us turn to the case where $\chi$ is non-principal. We have by holomorphicity 
\begin{equation}\label{eq:contourshiftE} \langle [\CC^\psi_a(N)], \omega_{E}(N)\rangle_\mathrm{cap}= \int_{w}^{\psi(a) w} E_{2,N}(z)dz, \end{equation}
for any $w\in \Hb$. Let 
$$x^+_a:=\psi(a)\infty,\quad x^-_a:=\psi(a)^{-1}\infty,$$
and put $w=x^-_a+iu/q$ in equation (\ref{eq:contourshiftE}) which satisfies 
\begin{equation}\label{eq:propertx-}\psi(a) (x^-_a+i\tfrac{u}{q})=x^+_a +i\tfrac{u^{-1}}{q} \quad \text{and}\quad j_{\psi(a)}(x^-_a+iu)^2=-1,\quad u>0.\end{equation}
By choosing the contour consisting of two straight lines connecting $w$, $x^-_a +iu^{-1}/q$ and $\psi(a) w $ we arrive at
\begin{align}\label{eq:steptowards}
\langle [\CC^\psi_a(N)], \omega_{E}(N)\rangle_\mathrm{cap}&=x^+_a-x^-_a+i\int_{u/q}^{1/q} E_{2,N}(x^-_a+iy) dy+i\int_{1/q}^{u^{-1}/q} E_{2,N}(x^-_a+iy) dy.
\end{align}
Recalling that $z\mapsto \Im (z) E_{2,N}(z)$ is automorphic of weight $2$ we conclude by the change of variable $y\mapsto y/q$, the properties (\ref{eq:propertx-}) and the change of variable $y\mapsto y^{-1}/q$ that 
\begin{align*}
\int_{u/q}^{1/q} E_{2,N}(x^-_a+iy) dy=\int_{u}^{1} \frac{y}{q} E_{2,N}(x^-_a+i\tfrac{y}{q}) \frac{dy}{y}&= -\int_{u}^{1} \frac{y^{-1}}{q} E_{2,N}(x^+_a+i\tfrac{y^{-1}}{q}) \frac{dy}{y}\\
&=-\int^{u^{-1}/q}_{1/q} E_{2,N}(x^+_a+iy) dy.   
\end{align*}
Now we insert the above into (\ref{eq:steptowards}) and subtract $1$ from each integrand (which cancel out). By the Fourier expansion (\ref{eq:EislevelN}) $E_{2,N}(z)-1$ decays rapidly as $\Im z\rightarrow \infty$ and so by letting $u\rightarrow 0$ we conclude
\begin{align*}
\langle [\CC^\psi_a(N)], \omega_{E}(N)\rangle_\mathrm{cap}&=x^+_a-x^-_a+i\lim_{u\rightarrow 0} \int_{u/q}^{u^{-1}/q} E_{2,N}(x^-_a+iy) dy\\
&= \frac{\tr(\psi(a))}{q}+i\int_{1/q}^{\infty} (E_{2,N}(x^-_a+iy)-1)dy-i\int_{1/q}^{\infty} (E_{2,N}(x^+_a+iy)-1) dy. \end{align*}
By the definition of $E_{2,N}$ we can write the integrals in terms of the modified Eisenstein series (\ref{eq:E2ast}) as
\begin{align}
&\int_{1/q}^{\infty} (E_{2,N}(x^\pm_a+iy)-1)dy\\
&=\frac{1}{N-1}\left(\int_{1/q}^{\infty} (N E^\ast_{2}(N (x^\pm_a+iy))-N+\tfrac{3}{\pi y })dy-\int_{1/q}^{\infty} (E^\ast_{2}(x^\pm_a+iy)-1+\tfrac{3}{\pi y })dy\right).
\end{align}  
Now we apply Proposition \ref{prop:linetoadd} (after the change of variable $\tfrac{y}{q}\mapsto y$) for $f=\nu_{\ell,1} E$ with $\ell\in\{1,N\}$ and $E$ defined as in (\ref{eq:E}) which yields for $s\in \C$
\begin{align*}
&\int_{1/q}^{\infty} (\ell E^\ast_{2}(\ell(x^+_a+iy))-\ell+\tfrac{3}{\pi y })(qy)^{s-1/2}dy-\int_{1/q}^{\infty} (\ell E^\ast_{2}(\ell(x^-_a+iy))-\ell+\tfrac{3}{\pi y }) (qy)^{1/2-s}dy\\
&= \frac{-6}{\pi}\int_{1/q}^{\infty} \left(E(\ell(x^+_a+iy))+\tfrac{\pi }{6}\ell y-\tfrac{1}{2}\right)(qy)^{s-1/2}\frac{dy}{y}+\frac{6}{\pi}\int_{1/q}^{\infty} \left(E(\ell(x^-_a+iy))+\tfrac{\pi }{6}\ell y-\tfrac{1}{2}\right) (qy)^{1/2-s}\frac{dy}{y}\\
&= \frac{-6}{\pi}\left(q^{s-1/2}2(2\pi \ell)^{1/2-s}\Gamma(s+\tfrac{1}{2})L(s, E,\ell \gamma\infty)+\frac{ \pi \ell}{6q}\left(\frac{1}{3/2-s}-\frac{1}{s+1/2}\right)+\frac{1}{s-1/2}\right),
\end{align*}
using that $\gamma_s^+(\mathbf{t}_{\nu_{\ell,1}E})=q^{s-1/2}2(2\pi)^{1/2-s}\Gamma(s+1/2)$ and the formula (\ref{eq:oldformAT}) for additive twist $L$-series of oldforms. Now when considering the sum over $a\in (\Z/q\Z)^\times$ twisted by $\chi\modulo q$ non-principal, we see that the poles at $s=1/2$ cancel by orthogonality of characters. Thus by letting $s\rightarrow 1/2$ in the Birch--Stevens formula (Proposition \ref{BirchStevens}), we get
\begin{align}
&\sum_{a\in (\Z/q\Z)^\times} \overline{\chi(a)} \langle [\CC^\psi_a(N)], \omega_{E}(N)\rangle_\mathrm{cap}\\
\label{eq:firsttermEisenstein}&=  \frac{12i}{\pi}\frac{\nu(\chi^*,q^\ast;N)-\nu(\chi^*,q^\ast;1)}{N-1}\tau(\overline{\chi^\ast})L(1/2, E, \chi^\ast)+\sum_{a\in (\Z/q\Z)^\times} \overline{\chi(a)}\tfrac{\tr(\psi(a))}{q},\end{align}
where $q^\ast=\tfrac{q}{\cc(\chi)}$. By Lemma \ref{lem:Eisweights} we arrive at the following expression for the first term in (\ref{eq:firsttermEisenstein}) valid in all cases; 
\begin{align}  \frac{-12i }{\pi}\frac{|1-\chi^\ast(N)|^2}{N-1} \nu(\chi^\ast,\tfrac{q^\ast}{(q^\ast,N^\infty)};1) \tau(\overline{\chi^\ast})L(1/2, E, \chi^\ast),   
\end{align}
using here that $N|q$ which means that if $(N,q^\ast)=1$ then $N|\cc(\chi)$ (so that $\chi^\ast(N)=0$).
We can express the Hecke eigenvalues of $E$ as the convolution product $|\cdot|^{1/2}\ast |\cdot|^{-1/2}$, which implies that 
$$\tau(\overline{\chi^\ast}) L(1/2, E, \chi^\ast)=\tau(\overline{\chi^\ast}) L(1,\chi^\ast)L(0,\chi^\ast),$$
vanishing if $\chi^\ast$ is even. For $\chi^\ast$ odd we get by the functional equation for Dirichlet $L$-functions \cite[Ch.\ 9, eq.\ (14)]{Davenport00} that 
$$\tau(\overline{\chi^\ast}) L(1/2, E, \chi^\ast)=\frac{i c(\chi)}{\pi } |L(1,\chi^\ast)|^2.$$
Putting all of this together yields the wanted result.
\end{proof}
We are now ready to prove our main lower bound for the Eisenstein contribution which is even uniform in the level $N$. 
\begin{cor}\label{cor:Eisensteincontr}
For each $\eps>0$ there are constants $c_1(\eps),c_2(\eps)>0$ such that the following holds: Let $N\geq 2$ be a prime and $q\geq 2$ an integer such that $N|q$. Consider a double coset embedding $\psi$ of level $q$ and a subgroup $H\leq (\Z/q\Z)^\times$  satisfying
\begin{enumerate}
\item $-1\notin H$,
\item $\{a\in (\Z/q\Z)^\times: a\equiv 1\modulo q/N \}\not\subset H $.
\end{enumerate}
Then we have 
\begin{align}\label{eq:wantedlowerbound}
\sum_{a\in H}  \langle [\CC_a^\psi(N)], \omega_{E}(N)\rangle_\mathrm{cap}
\geq c_1(\eps) \frac{q^{1-\eps}}{N}-c_2(\eps)|H|^{3/4}q^{1/6+\eps}    +\sum_{a\in H} n_\psi(a),\end{align}
for $q$ large enough.
\end{cor}
\begin{proof}
By character orthogonality and Proposition \ref{cor:Eiscontr} we have
\begin{align}
\nonumber\sum_{a\in H}  \langle [\CC_a^\psi(N)], \omega_{E}(N)\rangle_\mathrm{cap} &=\frac{|H|}{\varphi(q)}\sum_{\substack{\chi\modulo q\\ \chi_{|H}=1}}\,\sum_{a\in (\Z/q\Z)^\times}\overline{\chi(a)}\langle [\CC_a^\psi(N)], \omega_{E}(N)\rangle_\mathrm{cap}\\
\label{eq:Eiscontrcalc}&=\sum_{a\in H}\frac{\tr( \psi(a))}{q}+\frac{12|H|}{\pi^2\varphi(q)}\sum_{\substack{\chi\modulo q \\ \chi_{|H}=1\\ \chi(-1)=-1}}\tfrac{|\chi^\ast(N)-1|^2\cc(\chi)}{N-1} |L(1,\chi^\ast)|^2  \nu(\chi^\ast,\tfrac{q^\ast}{(q^\ast,N^\infty)};1).   \end{align}
To deal with the first term, we observe that by Lemma \ref{lem:lowerbndtrace} and Corollary \ref{cor:ta} we have 
$$\sum_{a\in H}\frac{\tr( \psi(a))}{q}=\sum_{a\in H} n_\psi(a)+\sum_{a\in H}\frac{t_a}{q}= \sum_{a\in H} n_\psi(a)+O_\eps(|H|^{3/4}q^{1/6+\eps}).$$ 
To give a lower bound for the sum over characters in (\ref{eq:Eiscontrcalc}) (note that all the summands are non-negative), we consider the following subset;
$$\mathcal{M}:=\{\chi \modulo q: (q^\ast,N)=1,\chi_{|H}=1,\chi(-1)=-1 \},$$
where $q^\ast=q/\cc(\chi)$. We see by the principle of inclusion and exclusion
$$|\mathcal{M}|=\frac{\varphi(q)}{|H|}-\frac{\varphi(q/N)}{|H\modulo q/N|}-\frac{\varphi(q)}{|H'|}+\frac{\varphi(q/N)}{|H'\modulo q/N|},$$
where $H':=\langle H,-1\rangle\leq (\Z/q\Z)^\times$. Note that $|H'|=2|H|$ by the assumption $-1\notin H$. By simple group theory we see that 
$$\frac{\varphi(q/N)}{|H\modulo q/N|}=\frac{\varphi(q)}{|H|} \frac{|\{ h\in H: h\equiv 1\modulo q/N \}|}{|\{ a\in (\Z/q\Z)^\times: a\equiv 1\modulo q/N \}|},$$
and furthermore the above is an integer dividing $\frac{\varphi(q)}{|H|}$ (and similarly for $H'$). This means that 
\begin{align}
|\mathcal{M}|&=\frac{\varphi(q)}{2|H|} \left(1-\frac{2|\{ h\in H: h\equiv 1\modulo q/N \}|}{|\{ a\in (\Z/q\Z)^\times: a\equiv 1\modulo q/N \}|}+\frac{|\{ h\in H': h\equiv 1\modulo q/N \}|}{|\{ a\in (\Z/q\Z)^\times: a\equiv 1\modulo q/N \}|}\right)\\
&\geq \frac{\varphi(q)}{2|H|} \left(1-\frac{|\{ h\in H: h\equiv 1\modulo q/N \}|}{|\{ a\in (\Z/q\Z)^\times: a\equiv 1\modulo q/N \}|}\right)\geq \frac{\varphi(q)}{4|H|}, \end{align}
where we used the second assumption on $H$ in the last inequality. Note that $(N,q^\ast)=1$  implies $N|\cc(\chi)$ and thus $\chi^\ast(N)=0$. This means that for each $\chi\in M$, we get a positive contribution in (\ref{eq:Eiscontrcalc}) of 
$$\tfrac{|\chi^\ast(N)-1|^2\cc(\chi)}{N-1} |L(1,\chi^\ast)|^2  \nu(\chi^\ast,\tfrac{q^\ast}{(q^\ast,N^\infty)};1) =\tfrac{\cc(\chi)}{N-1} |L(1,\chi^\ast)|^2  \nu(\chi^\ast,q^\ast;1)\gg_\eps \frac{q^{1-\eps}}{N}, $$
by Lemma \ref{lem:E2weightlowerbnd} and the lower bounds $L(1,\chi^\ast)\gg_\eps \cc(\chi)^{-\eps}$ and $\varphi(n)\gg_\eps n^{1-\eps}$. Combining this with the lower bound on $|\mathcal{M}|$ as above yields the wanted inequality (\ref{eq:wantedlowerbound}).
\end{proof}
\subsubsection{In the case where $-1\in H$}
Similarly we obtain the following lower bound in the case where $-1\in H$.
\begin{cor}\label{cor:-1inH}
Let $N\geq 2$ be prime. Let $q\geq 2$ be such that $N|q$ and let $H\leq (\Z/q\Z)^\times$ be a subgroup containing $-1$. Then for any coset $cH\subset (\Z/q\Z)^\times$ and $\eps>0$ we have 
\begin{align}
\sum_{a\in cH}  \langle [\CC_a^\psi(N)], \omega_{E}(N)\rangle_\mathrm{cap}
= \sum_{a\in cH} n_\psi(a)+O_\eps(|H|^{3/4}q^{1/6+\eps}).\end{align} 
\end{cor}
\begin{proof}  
By the assumption $-1\in H$, we have by  Proposition \ref{cor:Eiscontr} that
$$\sum_{a\in cH}  \langle [\CC_a^\psi(N)], \omega_{E}(N)\rangle_\mathrm{cap}=\sum_{a\in cH}\frac{\tr( \psi(a))}{q},$$
and thus the result follows by Lemma \ref{lem:lowerbndtrace} and the bound (\ref{eq:cor2}) of Corollary \ref{cor:ta} as above.
\end{proof}

\section{Proof of equidistribution}
In this section we will prove our main results using the above. We begin by some technical reductions which are necessary for the proof of Theorem \ref{thm:first}.
\subsection{Reduction to regular $\psi$}
Let $N\geq 1$ and let $q\geq 2$ be such that $N|q$. Let $\psi$ be a double coset embedding of level $q$ and $cH\subset (\Z/q\Z)^\times$ a coset. Then we want to prove that for any smooth, compactly supported function $\phi : \mathbf{T}^1(Y_0(N))\rightarrow \C$, we have (with an explicit error term) 
$$ \frac{\sum_{a\in cH}\int_{\CC^\psi_a(N)} \phi (t) d\mu^\psi_a(t)  }{\sum_{a\in cH}|\CC^\psi_a(N)|}\rightarrow \int_{\mathbf{T}^1(Y_0(N))} \phi (g) dg,\quad \text{as }q\rightarrow \infty.$$
Special care will have to be given to the $a\in (\Z/q\Z)^\times$ for which $\psi(a)$ is \emph{singular} in the sense that $|\tr(\psi(a))|<q^{1-\eps}$ for some (small) $\eps>0$. Given a double coset embedding $\psi: (\Z/q\Z)^\times\hookrightarrow \Gamma_0(N)$ of level $q$, we define for $\eps>0$  another  double coset embedding of level $q$  by 
$$\psi_\eps(a):=\begin{cases} \psi(a),& |\tr(\psi(a))|> q^{1-\eps}\\ \psi(a)\begin{psmallmatrix}1 & 1 \\ 0 & 1 \end{psmallmatrix} ,& |\tr(\psi(a))|\leq  q^{1-\eps}\end{cases}.$$
We say that $\psi$ is \emph{$\eps$-regular} if $\psi=\psi_\eps$ and notice that $\psi_\eps$ is itself $\eps$-regular for $q$ large enough since 
$$|\tr(\psi(a)\begin{psmallmatrix}1 & 1 \\ 0 & 1 \end{psmallmatrix})|=|q+\tr(\psi(a))|\geq q-|\tr(\psi(a))|.$$ 
We have the following result which allows us to reduce to the regular case. 
\begin{lemma}\label{lem:regular}
Fix $\delta>2/3$. Let $N\geq 1$ and $q\geq 2$ be integers such that $N|q$. Consider a double coset embedding $\psi$ of level q and a coset $cH\subset (\Z/q\Z)^\times$ of a subgroup $H\leq (\Z/q\Z)^\times$ of size $\geq q^\delta$. Then for any bounded function $\phi: \Gamma_0(N)\backslash \PSL_2(\R)\rightarrow \C$ and $\eps>0$ we have that
\begin{equation}\label{eq:regular}\left|\frac{\sum_{a\in cH}\int_{\CC^\psi_a(N)} \phi (t) d\mu^\psi_a(t)  }{\sum_{a\in cH}|\CC^\psi_a(N)|}-\frac{\sum_{a\in cH}\int_{\CC_a^{\psi_\eps}(N)} \phi (t) d\mu_a^{\psi_\eps}(t)  }{\sum_{a\in cH}|\CC^{\psi_\eps}_a(N)|}\right|\ll_\eps |\!|\phi|\!|_\infty\left(q^{-\eps}+\frac{q^{1/2+3\eps}}{|H|^{3/4}}\right).\end{equation}
\end{lemma}
\begin{proof}
First of all we see that
\begin{align}\nonumber
\left|\frac{\sum_{a\in cH}\left(\int_{\CC^\psi_a(N)} \phi (t) d\mu^\psi_a(t)-\int_{\CC_a^{\psi_\eps}(N)} \phi (t) d\mu_a^{\psi_\eps}(t)\right) }{\sum_{a\in cH}|\CC^\psi_a(N)|}\right| &
\ll |\!|\phi|\!|_\infty \frac{\sum_{a\in cH: |\tr (\psi(a))|\leq q^{1-\eps}}|\CC^\psi_a(N)|+|\CC_a^{\psi_\eps}(N)| }{\sum_{a\in cH}|\CC^\psi_a(N)|} \\
\label{eq:lemma71}&\ll |\!|\phi|\!|_\infty \frac{\log q\cdot |\{a\in cH: |\tr(\psi(a))|\leq q^{1-\eps}\}|}{\sum_{a\in cH}|\CC^\psi_a(N)|},
\end{align}
since $|\tr(\psi_\eps(a))|\leq q+|\tr(\psi(a))|\ll q$ in the sum on the right-hand side of (\ref{eq:lemma71}). Now by the bound (\ref{eq:cor1}) from Corollary \ref{cor:ta} (with $T=q^\eps$ and $\delta=0$) we see that
$$|\{a\in cH: |\tr(\psi(a))|\leq q^{1-\eps}\}|\leq |\{a\in cH: \tfrac{|t_a|}{q}\leq q^{-\eps}\}|\ll_\eps |H|q^{-\eps}+q^{1/2+3\eps}|H|^{1/4},$$
and by Proposition \ref{prop:lowerboundgeo} we have
$$\sum_{a\in cH}|\CC^\psi_a(N)|\gg |H|\log q,  $$
which shows that the left-hand side of (\ref{eq:lemma71}) is bounded by the right-hand side of (\ref{eq:regular}).  
Similarly we conclude that 
$$\left|\sum_{a\in cH}\int_{\CC_a^{\psi_\eps}(N)} \phi (t) d\mu_a^{\psi_\eps}(t)\left(\frac{1}{\sum_{a\in cH}|\CC^{\psi}_a(N)|}-\frac{1}{\sum_{a\in cH}|\CC^{\psi_\eps}_a(N)|}\right)\right|\ll|\!|\phi|\!|_\infty\left(q^{-\eps}+\frac{q^{1/2+3\eps}}{|H|^{3/4}}\right) $$
by Corollary \ref{cor:ta}, in view of Lemma \ref{lem:lowerbndtrace}, and Proposition \ref{prop:lowerboundgeo}. This gives the wanted by the triangle inequality.
\end{proof}
This means that we may assume that the double coset embedding $\psi$ is $\eps$-regular for some $\eps>0$ as long as $|H|\geq q^{\delta}$ for $\delta>2/3$. We are now ready to prove our main result.
\begin{thm} Fix $N\geq 1$, $0<\kappa<2$ and $\delta>7/8$. For every $q\geq 2$ such that $N|q$, consider a coset $cH\subset (\Z/q\Z)^\times$ with $H\leq (\Z/q\Z)^\times$ a subgroup of size $\geq q^\delta$ and a double coset embedding $\psi:(\Z/q\Z)^\times \hookrightarrow \Gamma_0(N)$  of level $q$ such that $|\tr (\psi(a))|\leq q (\log q)^{\kappa}$ for all $a\in cH$. Then for any smooth and compactly supported function $\phi: \mathbf{T}^1(Y_0(N))\rightarrow \C$ and $\eps>0$ we have that 
$$ \frac{\sum_{a\in cH}\int_{\CC^\psi_a(N)} \phi (t) d\mu^\psi_a(t)  }{\sum_{a\in cH}|\CC^\psi_a(N)|}= \int_{\mathbf{T}^1(Y_0(N))} \phi (g) dg+O_{\phi,\eps}((\log q)^{\kappa/2-1+\eps}).$$
\end{thm}
\begin{proof}
By Weyl's Criterion as in Lemma \ref{lem:WC} and by Proposition \ref{prop:lowerboundgeo}, it suffices to prove the following bound for the Weyl sums; 
$$\sum_{a\in cH} \int_{\CC_a^\psi(N)} F(t)d\mu_a^\psi(t)\overset{?}{\ll_\eps} |\mathbf{t}_F|^A (q^{7/8+\eps}+|H|(\log q)^{\kappa/2+\eps}),$$ 
for $F\in \BB_0(\Gamma_0(N))\cup \BB_e(\Gamma_0(N))$ and some (fixed) $A>0$. By applying Theorem \ref{thm:weylsums}, the problem is reduced to bounding the remaining terms depending on the traces of $\psi$ on the right-hand sides of (\ref{eq:weylsums1}) and (\ref{eq:weylsums2}). By the assumption that $|\tr (\psi(a))|\leq q (\log q)^{\kappa}$ we see that 
$$ \left(\frac{|\tr (\psi(a))|}{q}\right)^{1/2+\eps}\leq (\log q)^{\kappa/2+2\eps},   $$
as wanted. By Lemma \ref{lem:regular} and the assumption that $|H|\geq q^{\delta}$ with $\delta>7/8$ we may assume that $\psi$ is $\eps$-regular for some $0<\eps<1/100$ and thus by an application of the estimate (\ref{eq:cor1}) from Corollary \ref{cor:ta} we get 
\begin{align}\nonumber \sum_{a\in cH}\left(\frac{q}{|\tr (\psi(a))|}\right)^{1/2+\eps}&\leq |\{a\in cH: |\tr (\psi(a))|\geq q\}|+\sum_{1\leq m\leq \eps \log q}\sum_{\substack{a\in cH:\\ \tfrac{q}{2^{m}}\leq |\tr (\psi(a))|<\tfrac{q}{2^{m-1}}}}2^{m(1/2+\eps)}\\
&\ll  |H|+|H|\sum_{1\leq m\leq \eps \log q} 2^{-m} 2^{m(1/2+\eps)}+q^{1/2+3\eps} |H|^{1/4}\ll_\eps |H|.   \end{align}
Thus we conclude the wanted effective convergence using Weyl's Criterion as in Lemma \ref{lem:WC}.
\end{proof}
\begin{remark}\label{rem:reduce}
Notice that when $N=1$, $q=p$ is prime and $H= (\Z/p\Z)^\times$, the Weyl sum of a Hecke--Maa{\ss} form $f$ is of the shape 
$$c_fL(1/2,f)\nu(f,\chi_0, p;1)= c_fL(1/2,f)(p^{1/2}\lambda_f(p)-2),$$
for some constant $c_f$, \emph{plus} an error-term that is negligible compared to the main term (i.e.\ the last terms on the right-hand side of (\ref{eq:weylsums1}) and (\ref{eq:weylsums2})). Thus equidistribution as $p\rightarrow \infty$ in this case follows from any power saving bound for the Hecke eigenvalues. This is exactly because the Weyl sums in this case are essentially obtained by integrating the automorphic forms along the Hecke orbits  $\{\gamma_ti \R:t\in\mathbf{P}^1(\mathbb{F}_p)\}$ with notation as in (\ref{eq:Heckeorbit}) above.  
\end{remark}
\subsection{Concentration in homology}

Recall that for a finite dimensional vector space $V$ we defined the \emph{sphere associated to $V$} as $\mathbf{S}(V):=(V-\{0\})/\R_{>0}$ equipped with the quotient topology and denoted by $v\mapsto \overline{v}$ the canonical projection for $v\neq 0$. By the Weyl's Criterion as in Lemma \ref{lem:WCvector}, the convergence 
$$\overline{v_n}\rightarrow \overline{v_0}\quad \text{in }\mathbf{S}(V),$$  is  equivalent to showing for some (fixed) norm $\left|\!\left| \cdot\right|\!\right|:V\rightarrow \R_{\geq 0}$ that
\begin{equation}\label{eq:weyllinear} \left|\!\left| \tfrac{v_n}{\left|\!\left|  v_n\right|\!\right|}-\tfrac{v_0}{\left|\!\left|  v_0\right|\!\right|} \right|\!\right|\rightarrow 0,\quad \text{as }n\rightarrow \infty.\end{equation}
In our setting, we consider the vector space $V_N=H_1(Y_0(N),\R)$ with $N$ prime. We will actually obtain bounds for the left-hand side of (\ref{eq:weyllinear}) \emph{uniform} in terms of $N$ for certain natural norms on $V_N$ (this corresponds to the \emph{level aspect} in our equidistribution problem cf.\ \cite{LiuMasriYoung13}). Given a basis $B^\ast$ of the dual $V_N^\ast\cong H^1(Y_0(N),\R)$ we define the following norm on $V_N$;
\begin{equation}\label{eq:normB}|\!|v|\!|_{B^\ast}:=\max_{\omega\in B^\ast}|\langle v, \omega\rangle_\mathrm{cap}|.\end{equation}
It is a standard fact \cite[eq.\ (3)]{Zagier85} that the following matrices generate $\Gamma_0(N)$; 
\begin{equation}\label{eq:S}\left\{\begin{psmallmatrix} 1 & 1 \\0 & 1 \end{psmallmatrix}\right\}\cup \left\{\begin{psmallmatrix} a & -(aa^{\ast} +1)/N \\N & -a^\ast \end{psmallmatrix}: 0< a<N \right\},\end{equation} 
where $0< a^\ast <N$ is such that $aa^{\ast} \equiv -1\modulo N$. For each prime $N$ consider a basis $B$ of $H_1(Y_0(N), \R)$ containing $v_E(N)$ and consisting of images of elements of (\ref{eq:S}) under the isomorphism
$$H_1(Y_0(N), \Z)\cong \Gamma_0(N)^\mathrm{ab}/(\Gamma_0(N)^\mathrm{ab})_\mathrm{tor}.$$ 
Let $B^\ast $ be the dual basis and $|\!|\cdot|\!|_{N}:=|\!|\cdot|\!|_{B^\ast}$ the associated norm  on $H_1(Y_0(N),\R)$ as defined in (\ref{eq:normB}). The system of norms $(|\!|\cdot|\!|_{N})_{N\text{ prime}}$ will serve as a set of {\lq\lq}natural norms{\rq\rq} when varying $N$. They have a nice interpretation in terms of group theory as we will see.  

\begin{thm}\label{thm:mainhomo} Let $N\geq 2$ be a prime, $0<\eps<1/100$ and  let $|\!|\cdot|\!|_N$ be a norm on $H_1(Y_0(N),\R)$ defined as above. For $q\geq 2$ such that $N|q$, let $H\subset (\Z/q\Z)^\times$ be a subgroup and let $\psi$ a double coset embedding of level $q$ such that 
\begin{enumerate}
\item $-1\notin H$,
\item $\{a\in (\Z/q\Z)^\times: a\equiv 1\modulo q/N \}\not\subset H$,
\item \label{eq:nsumbound2}$\left|\sum_{a\in H}n_\psi(a)\right| \leq \frac{q^{1-\eps}}{N}$.
\end{enumerate}
Then we have 
\begin{equation}\label{eq:thmlevelaspect1}  \left|\! \left|\frac{\sum_{a\in H} [\mathcal{C}^\psi_{a}(N)]}{|\!|\sum_{a\in H}  [\mathcal{C}^\psi_{a}(N)]|\!|_N}-\frac{v_E(N)}{|\!|v_E(N)|\!|_N}\right|\! \right|_N \ll_{\eps}N^{9/4}q^{-1/8+\eps}. \end{equation}
\end{thm}   
\begin{proof}
We start by noting that the left-hand side of (\ref{eq:thmlevelaspect1}) is trivially bounded by $2$ by the triangle inequality. Thus we may assume that $q\geq N^{18} $. Let $B$ be a basis for $V_N$ associated to a subset of (\ref{eq:S}) containing $v_E(N)$ and let $B^\ast\subset H^1(Y_0(N),\R)$ be the dual basis with respect to the cap product pairing. By expanding in the Hecke basis (\ref{eq:Heckebasis}) we have for $\omega\in B^\ast$ that
\begin{align}\nonumber  \sum_{a\in H} \langle [\mathcal{C}^\psi_{a}(N)], \omega \rangle_\mathrm{cap} 
\label{eq:Wald1} =&\sum_{f\in \mathcal{B}_N, \pm} \langle v_f^{\pm}, \omega\rangle_\mathrm{cap} \sum_{a\in H} \langle [\mathcal{C}^\psi_a(N)], \omega_f^{\pm} \rangle_\mathrm{cap}  \\
&+\langle v_{E}(N), \omega\rangle_\mathrm{cap} \sum_{a\in H} \langle [\mathcal{C}^\psi_a(N)], \omega_{E}(N) \rangle_\mathrm{cap}  .  \end{align}
Now that by the definition of the cap product pairing and of the classes $\omega_f^\pm$ in (\ref{eq:omegapmclass}) we have 
\begin{equation}
\langle [\mathcal{C}^\psi_a(N)], \omega_f^{\pm} \rangle_\mathrm{cap}= \frac{1}{1+i\pm (1-i)} \left(2\pi i\int_{\mathcal{C}^\psi_a(N)} f(z)dz\pm \overline{2\pi i\int_{\mathcal{C}^\psi_a(N)} f(z)dz}\right).
\end{equation}
Thus by inserting the bound from Corollary \ref{cor:holoWeyl} (note that the first term is dominating since $q\geq N^{18}$)  we conclude that the cuspidal contribution in (\ref{eq:Wald1}) is bounded by 
 \begin{align}\ll_\eps N^{1/4}q^{7/8+\eps/2} \sum_{f\in\BB_N,\pm} |\langle v_f^{\pm}, \omega\rangle_\mathrm{cap}|\ll_\eps N^{5/4}q^{7/8+\eps},
 \end{align}
where we in the last inequality applied Cauchy--Schwarz and the $L^2$-bound for cap product pairings from \cite[Theorem 6.1]{Nor22}.
Now by the lower bound from Corollary \ref{cor:Eisensteincontr} (applied with $\eps/2$) and assumption (\ref{eq:nsumbound2}), we conclude for $q$ sufficiently large that 
$$ \sum_{a\in H}\langle [\mathcal{C}^\psi_{a}(N)], \omega_E(N) \rangle_\mathrm{cap} \geq c_1(\eps/2)\frac{q^{1-\eps/2}}{N}- c_2(\eps/2)|H|^{3/4}q^{1/6+\eps/2}-\frac{q^{1-\eps}}{N}\gg_\eps \frac{q^{1-\eps/2}}{N},$$
using also $N\leq q^{1/18}$ and $\eps<1/100$ so that $|H|^{3/4}q^{1/6+\eps/2}\leq q^{11/12+1/200}\leq q^{1-1/100}q^{-1/18}\leq q^{1-\eps}/N$. Now by applying the effective Weyl's Criterion as in Lemma \ref{lem:WCvector} to the vectors  
$$ \frac{\sum_{a\in H}[\mathcal{C}^\psi_{a}(N)]}{\sum_{a\in H}\langle [\mathcal{C}^\psi_{a}(N)], \omega_E(N) \rangle_\mathrm{cap}},$$
we conclude the wanted bound (\ref{eq:thmlevelaspect1}).  
\end{proof}
Now if we keep $N$ fixed we arrive at the following concentration result by Weyl's Criterion as in equation (\ref{eq:Weylfinalconclusion}) of Lemma \ref{lem:WCvector}. 
\begin{cor}\label{cor:mainhomo} Fix a prime $N\geq 3$ and $\eps>0$. For each positive integer $q\geq 2$ with $N|q$, consider a subgroup $H\leq (\Z/q\Z)^\times$ and a double coset embedding $\psi$ of level $q$ such that 
\begin{enumerate}
\item $-1\notin H$,
\item $\{a\in (\Z/q\Z)^\times: a\equiv 1\modulo q/N \}\not\subset H$,
\item \label{eq:nsumbound}$\left|\sum_{a\in H}n_\psi(a)\right| \leq q^{1-\eps}$.
\end{enumerate}
Then we have as $q\rightarrow \infty$ that 
$$\overline{\sum_{a\in H} [\mathcal{C}^\psi_{a}(N)]}\rightarrow \overline{v_E(N)},$$
in $\mathbf{S}(V_N)$.
\end{cor} 
Finally let us remark that condition (\ref{eq:nsumbound2}) in the corollary above is implied by condition (\hyperref[item:3]{3}) in the introduction, saying that the absolute trace is minimal: to see this note that in view of Lemma \ref{lem:lowerbndtrace} the trace being minimal corresponds to $n_\psi(a)=0$ for all $a\in H$ with $a+\overline{a}\not\in \{0,\pm 1,\pm 2\}\modulo q$ and $n_\psi(a)\in \{\pm1\}$ otherwise and we have for any $\eps>0$ the bound
\begin{align}
\nonumber &|\{a\in (\Z/q\Z)^\times :a+\overline{a}\in \{0,\pm 1,\pm 2\}\modulo q\}|\\
\label{eq:stupidstuff}=&|\{a\in (\Z/q\Z)^\times : \exists k\in \{0,\pm 1,\pm 2\}\text{ s.t. }q|(a^2+ka+1)\}| \ll_\eps q^{1/2+\eps},\quad \text{as }q\rightarrow \infty. 
\end{align} 
In conclusion Corollary \ref{cor:mainhomo} implies Theorem \ref{thm:second}. 
\begin{remark}
The natural bases of $H_1(Y_0(N),\R)$ constructed above are analogues of the set of isomorphism classes of supersingular elliptic curves $\mathcal{E\ell\ell}^{ss}(\mathbb{F}_{p^2})$ defined over $\F_{p^2}$ considered in the context of supersingular reduction of CM-elliptic curves. For details on this analogy see \cite[Section 2.1]{Nor22}. 
\end{remark}
\subsubsection{The case $-1\in H$}
In this section we will obtain results in the complementary case where the subgroup $H$ contains $-1$. In this case we will have to require the traces of the double coset embeddings to be biased.  
\begin{thm}\label{thm:mainH-1} Let $N\geq 2$ be prime and $\delta>1/6$ and let $|\!|\cdot|\!|_N$ be a norm on $H_1(Y_0(N),\R)$ defined above. For $q\geq 2$ such that $N|q$, let $H\subset (\Z/q\Z)^\times$ be a subgroup containing $-1$ and let $\psi$ be a double coset embedding of level $q$ such that 
\begin{equation}\label{eq:nsumbound3}\sum_{a\in cH}n_\psi(a)\geq M\geq |H|^{3/4} q^\delta.\end{equation}
Then we have for each $\eps>0$ that 
\begin{equation}\label{eq:thmlevelaspect}  \left|\! \left|\frac{\sum_{a\in H} [\mathcal{C}^\psi_{a}(N)]}{|\!|\sum_{a\in H}  [\mathcal{C}^\psi_{a}(N)]|\!|_N}-\frac{v_E(N)}{|\!|v_E(N)|\!|_N}\right|\! \right|_N \ll_{\eps}\frac{N^{5/4}q^{7/8+\eps}+N^{7/4} q^{3/4+\eps}}{M}. \end{equation}
\end{thm} 
\begin{proof}
The proof follows from Corollary \ref{cor:holoWeyl} exactly as above using now Corollary \ref{cor:-1inH} and the assumption (\ref{eq:nsumbound3}) to lower bound the contribution from the Eisenstein class $\omega_E(N)$.
\end{proof}
Again we conclude the following by Weyl's Criterion as in Lemma \ref{lem:WCvector}. 
\begin{cor}\label{cor:-1} Fix a prime $N\geq 2$ and $\eps>0$. For each $q\geq 2$ with $N|q$ consider a double coset embedding $\psi$ of level $q$ and a coset $cH\subset (\Z/q\Z)^\times$ of a subgroup $H$ containing $-1$ such that
$$\sum_{a\in cH} n_\psi(a)\geq |H|^{3/4}q^{1/6+\eps}+q^{7/8+\eps}.$$
Then we have as $q\rightarrow \infty$ that 
$$\overline{\sum_{a\in cH} [\mathcal{C}^\psi_{a}(N)]}\rightarrow \overline{v_E(N)},$$
in $\mathbf{S}(V_N)$.
\end{cor}

\subsubsection{An application to group theory}
Finally we will give a group theoretic application of Theorems \ref{thm:mainhomo} and \ref{thm:mainH-1}. For this part we will assume that $N$ is prime and $N\equiv -1\modulo 12$. Recall from the introduction that this means that $\Gamma_0(N)$ is torsion-free \cite[eq. (2.12),(2.13)]{Iw} and thus $\Gamma_0(N)$ is a free group on $2g+1$ generators where $g=\tfrac{N}{12}+O(1)$. In this case Doan--Kim--Lang--Tan \cite{DKLPT22} have recently constructed a subset of the generators (\ref{eq:S2}) containing the matrix $\begin{psmallmatrix} 1& 1\\0&1 \end{psmallmatrix}$ which defines a set of free generators of $\Gamma_0(N)$. We will be interested in the representation of matrices in this free basis which in general is a very hard question.    
\begin{proof}[Proof of Corollary \ref{cor:groupintro}]
Assume for now that $q$ is square-free. Let $\psi$ be the double coset embedding of level $q$ defined by
\begin{equation}\label{eq:doublecosetexample}\psi(a)=\begin{psmallmatrix} a & -\frac{aa^\ast+1}{q}\\q&-a^\ast \end{psmallmatrix}, \end{equation}
for $0< a,a^\ast<q$ s.t. $aa^\ast\equiv -1\modulo q$. Notice that the above matrices are indeed hyperbolic by the assumptions: clearly $\psi(a)\in \Gamma_0(N)$ which is torsion-free by the assumption that $N\equiv -1\modulo 12$ and so the only way hyperbolicity could fail is that the trace is $\pm 2$ which implies that $a(\pm 2-a)\equiv 1\mod q \Leftrightarrow q|(a\mp 1)^2$ and so $a\in \{1,q-1\}$ since $q$ is assumed square-free. In this last case the matrix is explicitly checked to be hyperbolic by the assumption $q\geq N\geq 11$. Recall that by (\ref{eq:cor2}) of Corollary \ref{cor:ta} we have
\begin{equation}\label{eq:sumstupid}\sum_{a\in H} \frac{\tr(\psi(a))}{q}= \sum_{a\in H}n_\psi(a)+\sum_{a\in H} \frac{t_a}{q}=\sum_{a\in H}n_\psi(a)+O_\eps(|H|^{3/4}q^{1/6+\eps}).\end{equation}
On the other hand, by (\ref{eq:cor3}) of Corollary \ref{cor:ta} we conclude that
\begin{equation}\label{eq:sumstupid2}\sum_{a\in H} \frac{\tr(\psi(a))}{q}=\left(\sum_{\substack{0<a<q:\\ (a\modulo q)\in H}} \frac{a}{q}\right)-\left(\sum_{\substack{0<a<q:\\ (a\modulo q)\in -H}} \frac{a}{q}\right)=  |H|(\tfrac{1}{2}-\tfrac{1}{2})+O_\eps(|H|^{2/3}q^{1/6+\eps}).\end{equation}
This means that 
$$\sum_{a\in H}n_\psi(a)\ll_\eps |H|^{2/3}q^{1/6+\eps}+|H|^{3/4}q^{1/6+\eps}\ll_\eps q^{11/12+\eps}, $$
and since $N\ll_\delta q^{1/\delta}$ with $\delta>18$, condition (\ref{eq:nsumbound2}) of Theorem \ref{thm:mainhomo} is satisfied for $q$ large enough and $\eps>0$ small enough. If we assume that a conjugate of the matrix (\ref{eq:specificmatrices2}) is in fact contained in the subgroup generated by $\sigma_1,\ldots, \sigma_{2g}$ then we get
$$ \left|\! \left|\frac{\sum_{a\in H} [\mathcal{C}^\psi_{a}(N)]}{|\!|\sum_{a\in H}  [\mathcal{C}^\psi_{a}(N)]|\!|_{B^\ast}}-\frac{v_E(N)}{|\!|v_E(N)|\!|_{B^\ast}}\right|\! \right|_{B^\ast}\geq 1, $$
where $B^\ast$ is the dual basis of the basis of $H_1(Y_0(N),\R)$ corresponding to $\left\{\begin{psmallmatrix} 1& 1\\0&1 \end{psmallmatrix}, \sigma_1,\ldots, \sigma_{2g}\right\}$. This contradicts the bound (\ref{eq:thmlevelaspect1}) of Theorem \ref{thm:mainhomo} by the assumption that $q\geq c(\delta) N^{\delta}$ with $\delta>18$. If $q$ is not square-free the matrix (\ref{eq:doublecosetexample}) is not hyperbolic whenever $a\neq \pm 1$ and $q|(a\pm 1)^2$. In these cases we multiply the matrix (\ref{eq:doublecosetexample}) by $\begin{psmallmatrix} 1& 1\\0&1 \end{psmallmatrix}$ from the right to obtain a hyperbolic matrix and consider this double coset embedding instead. This has a negligible effect on the sums (\ref{eq:sumstupid}) and (\ref{eq:sumstupid2}) by the estimate (\ref{eq:stupidstuff}) and so we conclude in this case as well. 
\end{proof}
Similarly we get the following corollary in the case where $-1\in H$ in which we need to consider matrices associated to $a\in H$ with \emph{positive} traces (meaning that the associated double coset embeddings are biased) so that Theorem \ref{thm:mainH-1} applies.
\begin{cor}\label{cor:group-1}
Let $N\equiv -1\modulo 12$ be prime and let 
$$\left\{\begin{psmallmatrix} 1& 1\\0&1 \end{psmallmatrix}, \sigma_1,\ldots, \sigma_{2g}\right\},$$ 
be a subset of (\ref{eq:S}) which defines a set of free generators of $\Gamma_0(N)$. Then for each $\delta>7/8$ there exists a constant $c(\delta)>0$ such that the following holds: Let $q\geq N^4$ be an integer such that $N|q$ and let  $0<a_1,\ldots, a_m<q$ with $m\geq c(\delta) N^{5/4} q^{\delta}$ be  pairwise distinct integers coprime to $q$ such that 
$$ \{a_i \modulo q: 1\leq i\leq m\}\subset (\Z/q\Z)^\times,$$
is a coset of a subgroup containing $-1$. For $1\leq i\leq m$, let $0<\overline{a_i}<q$ be defined by $q|a_i\overline{a_i}-1$. Then no conjugate of  
\begin{equation}\label{eq:specificmatrices}
\begin{psmallmatrix} a_1 & \frac{a_1\overline{a_1}-1}{q}\\q&\overline{a_1} \end{psmallmatrix}\begin{psmallmatrix} a_2 & \frac{a_2\overline{a_2}-1}{q}\\q&\overline{a_2} \end{psmallmatrix}\cdots \begin{psmallmatrix} a_{m} & \frac{a_{m}\overline{a_m}-1}{q}\\q&\overline{a_m} \end{psmallmatrix}\in \Gamma_0(N),\end{equation}
is contained in the subgroup of $\Gamma_0(N)$ generated by the matrices $\sigma_1,\ldots, \sigma_{2g}$.

\end{cor}

\begin{proof} Let 
$$\{a_i \modulo q: 1\leq i\leq m\}=cH,$$
with $H\leq (\Z/q\Z)^\times$ a subgroup and $c\in (\Z/q\Z)^\times$, and define the double coset embedding
$$\psi(a)=\begin{psmallmatrix} a & \frac{a\overline{a}-1}{q}\\q&\overline{a}\end{psmallmatrix}, 1<a< q-1,\qquad \psi(\pm 1)=\begin{psmallmatrix} \pm 1 & 1\\q&\pm (q+1) \end{psmallmatrix}.$$
Note that indeed the matrices above are hyperbolic. By equation (\ref{eq:cor3}) of Corollary \ref{cor:ta} we have
$$\sum_{a\in cH} \frac{\tr(\psi(a))}{q}=\left(\sum_{\substack{0<a<q:\\ (a\modulo q)\in cH}} \frac{a}{q}\right)+\left(\sum_{\substack{0<a<q:\\ (a\modulo q)\in \overline{c} H}} \frac{a}{q}\right) =|H|(\tfrac{1}{2}+\tfrac{1}{2})+O_\eps(|H|^{2/3}q^{1/6+\eps}),$$
which means in view of (\ref{eq:sumstupid})  that 
\begin{align*}\sum_{a\in cH}n_\psi(a)= |H|(1+O_\eps(|H|^{-1/3}q^{1/6+\eps}+|H|^{-1/4}q^{1/6+\eps}))&= |H|(1+O_\eps(q^{1/6-7/32+\eps}))
\\ &=|H|(1+o(1)),\end{align*}
 by the assumption on $H$. This shows that condition (\ref{eq:nsumbound3}) of Theorem \ref{thm:mainH-1} is satisfied, using here again that $|H|\gg q^{7/8}$ so that $|H|\gg |H|^{3/4}q^{1/6+\eps}$ for $\eps$ sufficiently small. Now the result follows exactly as above using Theorem \ref{thm:mainH-1} in place of Theorem \ref{thm:mainhomo}, where the right-hand side of (\ref{eq:thmlevelaspect}) tends to zero since $\sum_{a\in cH}n_\psi(a)\gg |H|\gg_\delta N^{5/4}q^{\delta} $ and $q\geq N^4$.
\end{proof}

\bibliography{mybib.bib}

@article{Sarnak82,
	author = {Peter Sarnak},
	doi = {https://doi.org/10.1016/0022-314X(82)90028-2},
	issn = {0022-314X},
	journal = {Journal of Number Theory},
	number = {2},
	pages = {229-247},
	title = {Class numbers of indefinite binary quadratic forms},
	url = {https://www.sciencedirect.com/science/article/pii/0022314X82900282},
	volume = {15},
	year = {1982}}

@article{No19,
	author = {Nordentoft, Asbj{\o}rn Christian},
	doi = {10.1515/crelle-2021-0013},
	fjournal = {Journal f{\"u}r die Reine und Angewandte Mathematik},
	issn = {0075-4102},
	journal = {J. Reine Angew. Math.},
	language = {English},
	pages = {255--293},
	title = {Central values of additive twists of cuspidal {{\(L\)}}-functions},
	volume = {776},
	year = {2021},
	zbl = {1532.11062},
	zbmath = {7377465}}

@book{Zelditch92,
	author = {Zelditch, Steven},
	doi = {10.1090/memo/0465},
	fseries = {Memoirs of the American Mathematical Society},
	isbn = {978-0-8218-2526-6; 978-1-4704-0891-6},
	issn = {0065-9266},
	language = {English},
	publisher = {Providence, RI: American Mathematical Society (AMS)},
	series = {Mem. Am. Math. Soc.},
	title = {Selberg trace formulae and equidistribution theorems for closed geodesics and {Laplace} eigenfunctions: {Finite} area surfaces},
	volume = {465},
	year = {1992},
	zbl = {0753.11023},
	zbmath = {32356}}

@article{OsShpJose22,
	author = {Ostafe, Alina and Shparlinski, Igor E. and Voloch, Jos{\'e} Felipe},
	doi = {10.1093/imrn/rnac226},
	fjournal = {IMRN. International Mathematics Research Notices},
	issn = {1073-7928},
	journal = {Int. Math. Res. Not.},
	language = {English},
	number = {16},
	pages = {14196--14238},
	title = {Equations and character sums with matrix powers, {Kloosterman} sums over small subgroups, and quantum ergodicity},
	volume = {2023},
	year = {2023},
	zbl = {1537.81028},
	zbmath = {7794877}}

@article{ErlandssonSouto22,
	author = {Erlandsson, Viveka and Souto, Juan},
	doi = {10.1093/imrn/rnad156},
	fjournal = {IMRN. International Mathematics Research Notices},
	issn = {1073-7928},
	journal = {Int. Math. Res. Not.},
	language = {English},
	number = {13},
	pages = {10298--10318},
	title = {Counting and equidistribution of reciprocal geodesics and dihedral groups},
	volume = {2024},
	year = {2024},
	zbl = {1560.22031},
	zbmath = {7935865}}

@article{BlomerBrumleyKhayutin22,
	adsurl = {https://ui.adsabs.harvard.edu/abs/2022arXiv221206280B},
	archiveprefix = {arXiv},
	author = {{Blomer}, Valentin and {Brumley}, Farrell and {Khayutin}, Ilya},
	doi = {10.48550/arXiv.2212.06280},
	eid = {arXiv:2212.06280},
	eprint = {2212.06280},
	journal = {arXiv e-prints},
	month = dec,
	pages = {arXiv:2212.06280},
	primaryclass = {math.NT},
	title = {{The mixing conjecture under GRH}},
	year = 2022}

@article{PetrowYoung19,
	author = {Petrow, Ian and Young, Matthew P.},
	doi = {10.1215/00127094-2022-0069},
	fjournal = {Duke Mathematical Journal},
	issn = {0012-7094},
	journal = {Duke Math. J.},
	language = {English},
	number = {10},
	pages = {1879--1960},
	title = {The fourth moment of {Dirichlet} {{\(L\)}}-functions along a coset and the {Weyl} bound},
	url = {discovery.ucl.ac.uk/id/eprint/10158522/1/1908.10346v3.pdf},
	volume = {172},
	year = {2023},
	zbl = {1544.11068},
	zbmath = {7732798}}

@article{BruggemanMotohashi05,
	author = {Bruggeman, Roelof W. and Motohashi, Yoichi},
	doi = {10.1515/crll.2005.2005.579.75},
	fjournal = {Journal f{\"u}r die Reine und Angewandte Mathematik},
	issn = {0075-4102},
	journal = {J. Reine Angew. Math.},
	language = {English},
	pages = {75--114},
	title = {A new approach to the spectral theory of the fourth moment of the {Riemann} zeta-function},
	url = {dspace.library.uu.nl/handle/1874/377841},
	volume = {579},
	year = {2005},
	zbl = {1064.11059},
	zbmath = {2143147}}

@article{ILS00,
	author = {Iwaniec, Henryk and Luo, Wenzhi and Sarnak, Peter},
	doi = {10.1007/BF02698741},
	fjournal = {Publications Math{\'e}matiques},
	issn = {0073-8301},
	journal = {Publ. Math., Inst. Hautes {\'E}tud. Sci.},
	language = {English},
	pages = {55--131},
	title = {Low lying zeros of families of {{\(L\)}}-functions},
	url = {https://eudml.org/doc/104166},
	volume = {91},
	year = {2000},
	zbl = {1012.11041},
	zbmath = {1656533}}

@article{GolubevKamber23,
	author = {Golubev, Konstantin and Kamber, Amitay},
	doi = {10.1017/fms.2023.40},
	fjournal = {Forum of Mathematics, Sigma},
	issn = {2050-5094},
	journal = {Forum Math. Sigma},
	language = {English},
	note = {Id/No e48},
	pages = {51pp.},
	title = {On {Sarnak}'s density conjecture and its applications},
	volume = {11},
	year = {2023},
	zbl = {1548.11084},
	zbmath = {7699402}}

@incollection{Darmon06,
	author = {Darmon, Henri},
	booktitle = {Proceedings of the international congress of mathematicians (ICM), Madrid, Spain, August 22--30, 2006. Volume II: Invited lectures},
	isbn = {978-3-03719-022-7},
	language = {English},
	pages = {313--345},
	publisher = {Z{\"u}rich: European Mathematical Society (EMS)},
	title = {Heegner points, {Stark}-{Heegner} points, and values of {{\(L\)}}-series},
	year = {2006},
	zbl = {1157.11023},
	zbmath = {5057401}}

@article{BlomerMichelUnipotentMixing,
	author = {Blomer, Valentin and Michel, Philippe},
	doi = {10.1007/s11854-023-0326-8},
	fjournal = {Journal d'Analyse Math{\'e}matique},
	issn = {0021-7670},
	journal = {J. Anal. Math.},
	language = {English},
	number = {1},
	pages = {25--57},
	title = {The unipotent mixing conjecture},
	volume = {151},
	year = {2023},
	zbl = {1543.11050},
	zbmath = {7793420}}

@article{ELMV09,
	author = {Einsiedler, Manfred and Lindenstrauss, Elon and Michel, Philippe and Venkatesh, Akshay},
	doi = {10.1215/00127094-2009-023},
	fjournal = {Duke Mathematical Journal},
	issn = {0012-7094},
	journal = {Duke Math. J.},
	language = {English},
	number = {1},
	pages = {119--174},
	title = {Distribution of periodic torus orbits on homogeneous spaces},
	volume = {148},
	year = {2009},
	zbl = {1172.37003},
	zbmath = {5555678}}

@article{Nor22,
	author = {Nordentoft, Asbj{\o}rn Christian},
	doi = {10.1017/fms.2023.85},
	fjournal = {Forum of Mathematics, Sigma},
	issn = {2050-5094},
	journal = {Forum Math. Sigma},
	language = {English},
	note = {Id/No e91},
	pages = {38pp.},
	title = {Concentration of closed geodesics in the homology of modular curves},
	volume = {11},
	year = {2023},
	zbmath = {7753226}}

@article{BourgainKontorovich17,
	author = {Bourgain, Jean and Kontorovich, Alex},
	doi = {10.4171/JEMS/694},
	fjournal = {Journal of the European Mathematical Society (JEMS)},
	issn = {1435-9855},
	journal = {J. Eur. Math. Soc. (JEMS)},
	language = {English},
	number = {5},
	pages = {1331--1359},
	title = {Beyond expansion. {II}: {Low}-lying fundamental geodesics.},
	volume = {19},
	year = {2017},
	zbl = {1403.11054},
	zbmath = {6725180}}

@incollection{Sarnak07,
	author = {Sarnak, Peter},
	booktitle = {Analytic number theory. A tribute to Gauss and Dirichlet. Proceedings of the Gauss-Dirichlet conference, G\"ottingen, Germany, June 20--24, 2005},
	isbn = {978-0-8218-4307-9},
	language = {English},
	pages = {217--237},
	publisher = {Providence, RI: American Mathematical Society (AMS)},
	title = {Reciprocal geodesics},
	year = {2007},
	zbl = {1198.11039},
	zbmath = {5233965}}

@article{AkaEinsiedler16,
	author = {Aka, Menny and Einsiedler, Manfred},
	doi = {10.1017/etds.2014.68},
	fjournal = {Ergodic Theory and Dynamical Systems},
	issn = {0143-3857},
	journal = {Ergodic Theory Dyn. Syst.},
	language = {English},
	number = {2},
	pages = {335--342},
	title = {Duke's theorem for subcollections},
	volume = {36},
	year = {2016},
	zbl = {1355.37004},
	zbmath = {6585342}}

@article{BourgainKonyagin03,
	author = {Bourgain, Jean and Konyagin, S. V.},
	doi = {10.1016/S1631-073X(03)00281-4},
	fjournal = {Comptes Rendus. Math{\'e}matique. Acad{\'e}mie des Sciences, Paris},
	issn = {1631-073X},
	journal = {C. R., Math., Acad. Sci. Paris},
	language = {English},
	number = {2},
	pages = {75--80},
	title = {Estimates for the number of sums and products and for exponential sums over subgroups in fields of prime order.},
	volume = {337},
	year = {2003},
	zbl = {1041.11056},
	zbmath = {1981380}}

@article{Skubenko62,
	author = {Skubenko, Boris F.},
	fjournal = {Izvestiya Akademii Nauk SSSR. Seriya Matematicheskaya},
	issn = {0373-2436},
	journal = {Izv. Akad. Nauk SSSR, Ser. Mat.},
	language = {Russian},
	pages = {721--752},
	title = {Die asymptotische {Verteilung} der {Gitterpunkte} auf einem einschaligen {Hyperboloid} und {Ergodens{\"a}tze}},
	volume = {26},
	year = {1962},
	zbl = {0107.04201},
	zbmath = {3174568}}

@article{BlomerHarcos08.2,
	author = {Blomer, Valentin and Harcos, Gergely},
	doi = {10.1215/00127094-2008-038},
	fjournal = {Duke Mathematical Journal},
	issn = {0012-7094},
	journal = {Duke Math. J.},
	language = {English},
	number = {2},
	pages = {321--339},
	title = {The spectral decomposition of shifted convolution sums},
	volume = {144},
	year = {2008},
	zbl = {1246.11108},
	zbmath = {5317182}}

@article{FouvryKowMich15,
	author = {Fouvry, {\'E}tienne and Kowalski, Emmanuel and Michel, Philippe},
	doi = {10.1007/s00039-015-0310-2},
	fjournal = {Geometric and Functional Analysis. GAFA},
	issn = {1016-443X},
	journal = {Geom. Funct. Anal.},
	language = {English},
	number = {2},
	pages = {580--657},
	title = {Algebraic twists of modular forms and {Hecke} orbits},
	volume = {25},
	year = {2015},
	zbl = {1344.11036},
	zbmath = {6438346}}

@article{Humphries22,
	author = {Humphries, Peter},
	doi = {10.1093/qmath/haab015},
	fjournal = {The Quarterly Journal of Mathematics},
	issn = {0033-5606},
	journal = {Q. J. Math.},
	language = {English},
	number = {1},
	pages = {1--16},
	title = {Distributing points on the torus via modular inverses},
	volume = {73},
	year = {2022},
	zbmath = {7499690}}

@book{KniLi13,
	author = {Knightly, Andrew and Li, Charles},
	doi = {10.1090/S0065-9266-2012-00673-3},
	fseries = {Memoirs of the American Mathematical Society},
	isbn = {978-0-8218-8744-8; 978-1-4704-1006-3},
	issn = {0065-9266},
	language = {English},
	publisher = {Providence, RI: American Mathematical Society (AMS)},
	series = {Mem. Am. Math. Soc.},
	title = {Kuznetsov's trace formula and the {Hecke} eigenvalues of {Maass} forms},
	volume = {224, no. 1055},
	year = {2013},
	zbl = {1314.11038},
	zbmath = {6304957}}

@article{Young19,
	author = {Young, Matthew P.},
	doi = {10.1016/j.jnt.2018.11.007},
	fjournal = {Journal of Number Theory},
	issn = {0022-314X},
	journal = {J. Number Theory},
	language = {English},
	pages = {1--48},
	title = {Explicit calculations with {Eisenstein} series},
	volume = {199},
	year = {2019},
	zbl = {1454.11083},
	zbmath = {7030973}}

@article{Khayutin19,
	author = {Khayutin, Ilya},
	fjournal = {Annals of Mathematics. Second Series},
	issn = {0003-486X},
	journal = {Ann. Math. (2)},
	language = {English},
	number = {1},
	pages = {145--276},
	title = {Joint equidistribution of {CM} points},
	volume = {189},
	year = {2019},
	zbl = {1431.11080},
	zbmath = {7003147}}

@article{HumphriesNordentoft22,
	adsurl = {https://ui.adsabs.harvard.edu/abs/2022arXiv221105890H},
	archiveprefix = {arXiv},
	author = {{Humphries}, Peter and {Nordentoft}, Asbj{\o}rn Christian},
	eid = {arXiv:2211.05890},
	eprint = {2211.05890},
	journal = {arXiv e-prints},
	month = nov,
	pages = {arXiv:2211.05890},
	primaryclass = {math.NT},
	title = {{Sparse Equidistribution of Geometric Invariants of Real Quadratic Fields}},
	year = 2022}

@article{DKLPT22,
	adsurl = {https://ui.adsabs.harvard.edu/abs/2022arXiv220913937D},
	archiveprefix = {arXiv},
	author = {{Doan}, Nhat Minh and {Kim}, Sang-hyun and {Lang}, Mong Lung and {Peow Tan}, Ser},
	eid = {arXiv:2209.13937},
	eprint = {2209.13937},
	journal = {arXiv e-prints},
	month = sep,
	pages = {arXiv:2209.13937},
	primaryclass = {math.NT},
	title = {{Optimal independent generating system for the congruence subgroups $\Gamma_0(p)$}},
	year = 2022}

@book{Hatcher02,
	author = {Hatcher, Allen},
	isbn = {0-521-79540-0},
	language = {English},
	publisher = {Cambridge: Cambridge University Press},
	title = {Algebraic topology},
	year = {2002},
	zbl = {1044.55001},
	zbmath = {2103273}}

@article{DrNo22,
	adsurl = {https://ui.adsabs.harvard.edu/abs/2022arXiv220814346D},
	archiveprefix = {arXiv},
	author = {{Drappeau}, Sary and {Nordentoft}, Asbj{\o}rn Christian},
	eid = {arXiv:2208.14346},
	eprint = {2208.14346},
	journal = {arXiv e-prints},
	month = aug,
	pages = {arXiv:2208.14346},
	primaryclass = {math.NT},
	title = {{Central values of additive twists of Maa{\ss} forms $L$-functions}},
	year = 2022}

@article{HoffLock94,
	author = {Hoffstein, Jeffrey and Lockhart, Paul},
	doi = {10.2307/2118543},
	fjournal = {Annals of Mathematics. Second Series},
	issn = {0003-486X},
	journal = {Ann. Math. (2)},
	language = {English},
	number = {1},
	pages = {161--176, appendix 177--181},
	title = {Coefficients of {Maass} forms and the {Siegel} zero. {Appendix}: {An} effective zero-free region, by {Dorian} {Goldfeld}, {Jeffrey} {Hoffstein} and {Daniel} {Lieman}},
	volume = {140},
	year = {1994},
	zbl = {0814.11032},
	zbmath = {709811}}

@article{Zagier85,
	author = {Don {Zagier}},
	doi = {10.4153/CMB-1985-044-8},
	fjournal = {{Canadian Mathematical Bulletin}},
	issn = {0008-4395},
	journal = {{Can. Math. Bull.}},
	language = {English},
	msc2010 = {14H45 14H52 11F33 11F11},
	pages = {372--384},
	publisher = {Cambridge University Press, Cambridge; Canadian Mathematical Society, Ottawa, ON},
	title = {{Modular parametrizations of elliptic curves}},
	volume = {28},
	year = {1985},
	zbl = {0579.14027}}

@article{LMY15,
	author = {Sheng-Chi {Liu} and Riad {Masri} and Matthew P. {Young}},
	doi = {10.1186/s40687-015-0040-y},
	fjournal = {{Research in the Mathematical Sciences}},
	issn = {2522-0144},
	journal = {{Res. Math. Sci.}},
	language = {English},
	msc2010 = {11F67 11F66 11G05},
	note = {Id/No 22},
	pages = {23pp.},
	publisher = {Springer International Publishing (SpringerOpen), Cham},
	title = {{Rankin-Selberg \(L\)-functions and the reduction of CM elliptic curves}},
	volume = {2},
	year = {2015},
	zbl = {1380.11060}}

@article{DuImTo18,
	author = {Duke, William and Imamo\={g}lu, \"{O}zlem and T\'{o}th, \'{A}rp\'{a}d},
	doi = {10.1007/s40687-018-0138-0},
	fjournal = {{Research in the Mathematical Sciences}},
	issn = {2522-0144},
	journal = {{Res. Math. Sci.}},
	language = {English},
	msc2010 = {11F11 11R29 11R11 11F37},
	note = {Id/No 20},
	number = {2},
	pages = {21pp.},
	publisher = {Springer International Publishing (SpringerOpen), Cham},
	title = {{Kronecker's first limit formula, revisited}},
	volume = {5},
	year = {2018},
	zbl = {1441.11085}}

@book{Davenport00,
	author = {Harold {Davenport}},
	fjournal = {{Graduate Texts in Mathematics}},
	isbn = {0-387-95097-4},
	issn = {0072-5285},
	journal = {{Grad. Texts Math.}},
	language = {English},
	msc2010 = {11-02 11-01 11Nxx 11Lxx 11Mxx 11P32 11N35},
	pages = {x + 177},
	publisher = {New York, NY: Springer},
	title = {{Multiplicative number theory. Revised and with a preface by Hugh L. Montgomery.}},
	volume = {74},
	year = {2000},
	zbl = {1002.11001}}

@article{BlomerHarcos08,
	author = {Blomer, Valentin and Harcos, Gergely},
	doi = {10.1515/CRELLE.2008.058},
	fjournal = {Journal f\"{u}r die Reine und Angewandte Mathematik. [Crelle's Journal]},
	issn = {0075-4102},
	journal = {J. Reine Angew. Math.},
	mrclass = {11F66 (11M41)},
	mrnumber = {2431250},
	mrreviewer = {Guillaume Ricotta},
	pages = {53--79},
	title = {Hybrid bounds for twisted {$L$}-functions},
	url = {https://doi.org/10.1515/CRELLE.2008.058},
	volume = {621},
	year = {2008}}

@article{Waldspurger85,
	author = {Waldspurger, Jean-Loup},
	fjournal = {Compositio Mathematica},
	issn = {0010-437X},
	journal = {Compositio Math.},
	mrclass = {11F70 (11F67 22E55)},
	mrnumber = {783511},
	mrreviewer = {Stephen Gelbart},
	number = {2},
	pages = {173--242},
	title = {Sur les valeurs de certaines fonctions {$L$} automorphes en leur centre de sym\'{e}trie},
	url = {http://www.numdam.org/item?id=CM_1985__54_2_173_0},
	volume = {54},
	year = {1985}}

@book{Bump97,
	author = {Bump, Daniel},
	doi = {10.1017/CBO9780511609572},
	isbn = {0-521-55098-X},
	mrclass = {11F70 (11F41 11R39 22E50 22E55)},
	mrnumber = {1431508},
	mrreviewer = {Solomon Friedberg},
	pages = {xiv+574},
	publisher = {Cambridge University Press, Cambridge},
	series = {Cambridge Studies in Advanced Mathematics},
	title = {Automorphic forms and representations},
	url = {https://doi.org/10.1017/CBO9780511609572},
	volume = {55},
	year = {1997}}

@article{EinLindMichVenk12,
	author = {Einsiedler, Manfred and Lindenstrauss, Elon and Michel, Philippe and Venkatesh, Akshay},
	doi = {10.4171/LEM/58-3-2},
	fjournal = {L'Enseignement Math\'{e}matique. Revue Internationale. 2e S\'{e}rie},
	issn = {0013-8584},
	journal = {Enseign. Math. (2)},
	mrclass = {11F11 (11E16 11F37 14C22 37A45 53C22)},
	mrnumber = {3058601},
	mrreviewer = {Thomas R. Shemanske},
	number = {3-4},
	pages = {249--313},
	title = {The distribution of closed geodesics on the modular surface, and {D}uke's theorem},
	url = {https://doi.org/10.4171/LEM/58-3-2},
	volume = {58},
	year = {2012}}

@article{BlomerMilicevic15,
	author = {Blomer, Valentin and Mili\'{c}evi\'{c}, Djordje},
	doi = {10.1007/s00039-015-0318-7},
	fjournal = {Geometric and Functional Analysis},
	issn = {1016-443X},
	journal = {Geom. Funct. Anal.},
	mrclass = {11F66 (11F72 11L07)},
	mrnumber = {3334233},
	mrreviewer = {Ravi Raghunathan},
	number = {2},
	pages = {453--516},
	title = {The second moment of twisted modular {$L$}-functions},
	url = {https://doi.org/10.1007/s00039-015-0318-7},
	volume = {25},
	year = {2015}}

@incollection{MichelVenk06,
	author = {Michel, Philippe and Venkatesh, Akshay},
	booktitle = {International {C}ongress of {M}athematicians. {V}ol. {II}},
	mrclass = {11F66 (11F67 11M41 37A45)},
	mrnumber = {2275604},
	mrreviewer = {Gergely Harcos},
	pages = {421--457},
	publisher = {Eur. Math. Soc., Z\"{u}rich},
	title = {Equidistribution, {$L$}-functions and ergodic theory: on some problems of {Y}u. {L}innik},
	year = {2006}}

@article{LiuMasriYoung13,
	author = {Liu, Sheng-Chi and Masri, Riad and Young, Matthew P.},
	doi = {10.1112/S0010437X13007033},
	fjournal = {Compositio Mathematica},
	issn = {0010-437X},
	journal = {Compos. Math.},
	mrclass = {11M41},
	mrnumber = {3078642},
	mrreviewer = {Gergely Harcos},
	number = {7},
	pages = {1150--1174},
	title = {Subconvexity and equidistribution of {H}eegner points in the level aspect},
	url = {https://doi.org/10.1112/S0010437X13007033},
	volume = {149},
	year = {2013}}

@article{Popa06,
	author = {Popa, Alexandru},
	doi = {10.1112/S0010437X06002259},
	fjournal = {Compositio Mathematica},
	issn = {0010-437X},
	journal = {Compos. Math.},
	mrclass = {11F67 (11F27 11F70)},
	mrnumber = {2249532},
	mrreviewer = {Gergely Harcos},
	number = {4},
	pages = {811--866},
	title = {Central values of {R}ankin {$L$}-series over real quadratic fields},
	url = {https://doi.org/10.1112/S0010437X06002259},
	volume = {142},
	year = {2006}}

@article{HarcosMichel06,
	author = {Harcos, Gergely and Michel, Philippe},
	doi = {10.1007/s00222-005-0468-6},
	fjournal = {Inventiones Mathematicae},
	issn = {0020-9910},
	journal = {Invent. Math.},
	mrclass = {11F66 (11F67 11M41)},
	mrnumber = {2207235},
	mrreviewer = {K. Soundararajan},
	number = {3},
	pages = {581--655},
	title = {The subconvexity problem for {R}ankin-{S}elberg {$L$}-functions and equidistribution of {H}eegner points. {II}},
	url = {https://doi.org/10.1007/s00222-005-0468-6},
	volume = {163},
	year = {2006}}

@incollection{Zagier10,
	author = {Zagier, Don},
	booktitle = {Quanta of maths},
	mrclass = {11F99 (33D90 57M27)},
	mrnumber = {2757599},
	mrreviewer = {Sander Zwegers},
	pages = {659--675},
	publisher = {Amer. Math. Soc., Providence, RI},
	series = {Clay Math. Proc.},
	title = {Quantum modular forms},
	volume = {11},
	year = {2010}}

@article{KiSa03,
	author = {Kim, Henry H.},
	doi = {10.1090/S0894-0347-02-00410-1},
	fjournal = {Journal of the American Mathematical Society},
	issn = {0894-0347},
	journal = {J. Amer. Math. Soc.},
	mrclass = {11F70 (11R39 22E46)},
	mrnumber = {1937203},
	mrreviewer = {Mahdi Asgari},
	note = {With appendix 1 by Dinakar Ramakrishnan and appendix 2 by Kim and Peter Sarnak},
	number = {1},
	pages = {139--183},
	title = {Functoriality for the exterior square of {${\rm GL}_4$} and the symmetric fourth of {${\rm GL}_2$}},
	url = {https://doi.org/10.1090/S0894-0347-02-00410-1},
	volume = {16},
	year = {2003}}

@book{Li68,
	author = {Linnik, Yuri},
	mrclass = {10.65},
	mrnumber = {0238801},
	pages = {ix+192},
	publisher = {Springer-Verlag New York Inc., New York},
	series = {Translated from the Russian by M. S. Keane. Ergebnisse der Mathematik und ihrer Grenzgebiete, Band 45},
	title = {Ergodic properties of algebraic fields},
	year = {1968}}

@article{Iw87,
	author = {Iwaniec, Henryk},
	doi = {10.1007/BF01389423},
	fjournal = {Inventiones Mathematicae},
	issn = {0020-9910},
	journal = {Invent. Math.},
	mrclass = {11F37 (11F30)},
	mrnumber = {870736},
	mrreviewer = {Marvin I. Knopp},
	number = {2},
	pages = {385--401},
	title = {Fourier coefficients of modular forms of half-integral weight},
	url = {https://doi.org/10.1007/BF01389423},
	volume = {87},
	year = {1987}}

@article{Du88,
	author = {Duke, William},
	doi = {10.1007/BF01393993},
	fjournal = {Inventiones Mathematicae},
	issn = {0020-9910},
	journal = {Invent. Math.},
	mrclass = {11F11 (11E32 11F30 11F37)},
	mrnumber = {931205},
	mrreviewer = {Mark Sheingorn},
	number = {1},
	pages = {73--90},
	title = {Hyperbolic distribution problems and half-integral weight {M}aass forms},
	url = {https://doi.org/10.1007/BF01393993},
	volume = {92},
	year = {1988}}

@article{DuFrIw02,
	author = {Duke, William and Friedlander, John B. and Iwaniec, Henryk},
	doi = {10.1007/s002220200223},
	fjournal = {Inventiones Mathematicae},
	issn = {0020-9910},
	journal = {Invent. Math.},
	mrclass = {11F66 (11F30 11M41 11R29)},
	mrnumber = {1923476},
	mrreviewer = {K. Soundararajan},
	number = {3},
	pages = {489--577},
	title = {The subconvexity problem for {A}rtin {$L$}-functions},
	url = {https://doi.org/10.1007/s002220200223},
	volume = {149},
	year = {2002}}

@article{Sh75,
	author = {Shimura, Goro},
	doi = {10.1112/plms/s3-31.1.79},
	fjournal = {Proceedings of the London Mathematical Society. Third Series},
	issn = {0024-6115},
	journal = {Proc. London Math. Soc. (3)},
	mrclass = {10D15},
	mrnumber = {0382176},
	mrreviewer = {I. I. Pjateckii-Sapiro},
	number = {1},
	pages = {79--98},
	title = {On the holomorphy of certain {D}irichlet series},
	url = {https://doi.org/10.1112/plms/s3-31.1.79},
	volume = {31},
	year = {1975}}

@book{Sh94,
	author = {Shimura, Goro},
	isbn = {0-691-08092-5},
	mrclass = {11Fxx (11-02 11G05 11G40)},
	mrnumber = {1291394},
	note = {Reprint of the 1971 original, Kan\^o Memorial Lectures, 1},
	pages = {xiv+271},
	publisher = {Princeton University Press, Princeton, NJ},
	series = {Publications of the Mathematical Society of Japan},
	title = {Introduction to the arithmetic theory of automorphic functions},
	volume = {11},
	year = {1994}}

@book{Iw,
	author = {Iwaniec, Henryk},
	doi = {10.1090/gsm/053},
	edition = {Second},
	isbn = {0-8218-3160-7},
	mrclass = {11F72 (11F12 11F37)},
	mrnumber = {1942691},
	pages = {xii+220},
	publisher = {American Mathematical Society, Providence, RI; Revista Matem\'atica Iberoamericana, Madrid},
	series = {Graduate Studies in Mathematics},
	title = {Spectral methods of automorphic forms},
	url = {https://doi.org/10.1090/gsm/053},
	volume = {53},
	year = {2002}}

@book{IwKo,
	author = {Iwaniec, Henryk and Kowalski, Emmanuel},
	doi = {10.1090/coll/053},
	isbn = {0-8218-3633-1},
	mrclass = {11-02 (11Fxx 11Lxx 11Mxx 11Nxx)},
	mrnumber = {2061214},
	mrreviewer = {K. Soundararajan},
	pages = {xii+615},
	publisher = {American Mathematical Society, Providence, RI},
	series = {American Mathematical Society Colloquium Publications},
	title = {Analytic number theory},
	url = {https://doi.org/10.1090/coll/053},
	volume = {53},
	year = {2004}}
\bibliographystyle{alpha}
\end{document}